\documentclass[10pt,leqno,a4paper]{article}
\usepackage{amsmath}
\usepackage{amsthm}
\usepackage{amssymb}

\usepackage[hypertex]{hyperref}
\usepackage[english]{babel}
\usepackage[totalheight=22 true cm, totalwidth=14 true cm]{geometry}
\usepackage{setspace}
\usepackage{mathrsfs}
\usepackage[all]{xy}

\newtheorem{thm}{Theorem}[subsection]
\newtheorem{cor}[thm]{Corollary}
\newtheorem{lemma}[thm]{Lemma}

\newtheorem{prop}[thm]{Proposition}

\newtheorem{defn}[thm]{Definition}
\newtheorem{conj}[thm]{Conjecture}

\theoremstyle{remark}

\theoremstyle{definition}

\newtheorem{rmk}[thm]{Remark}
\newtheorem{exa}[thm]{Example}
\newtheorem{notation}[thm]{Notation}
\numberwithin{equation}{thm}

\def\beq{\begin{equation}}
\def\eeq{\end{equation}}

\def\crash#1{}
\def\N{{\mathbb N}}
\def\Z{{\mathbb Z}}
\def\P{{\mathbb P}}
\def\Q{{\mathbb Q}}
\def\R{{\mathbb R}}
\def\C{{\mathbb C}}
\def\A{{\mathbb A}}

\def\H{{\mathbb H}}
\def\G{{\mathbb G}}

\def\l{\left}
\def\r{\right}
\def\[[{\l[\l[}
\def\]]{\r]\r]}
\def\p{\prime}

\def\an{{\rm an}}
\def\rk{{\rm rk}}
\def\Aut{{\rm Aut}}
\def\cf{\emph{cf.}~}
\def\ie{\emph{i.e.}~}
\def\lc{\emph{loc.cit.}~}

\def\ds{\displaystyle}

\def\cA{{\mathcal A}}
\def\cB{{\mathcal B}}
\def\cC{{\mathcal C}}
\def\cD{{\mathcal D}}
\def\cE{{\mathcal E}}
\def\cF{{\mathcal F}}
\def\cG{{\mathcal G}}

\def\cM{{\mathcal M}}

\def\cO{{\mathcal O}}
\def\cH{{\mathcal H}}

\def\cK{{\mathcal K}}

\def\cP{{\mathcal P}}
\def\cR{{\mathcal R}}
\def\cS{{\mathcal S}}

\def\cV{{\mathcal V}}
\def\cU{{\mathcal U}}
\def\cX{{\mathcal X}}
\def\cY{{\mathcal Y}}
\def\cZ{{\mathcal Z}}

\def\sA{{\mathscr A}}

\def\sF{{\mathscr F}}
\def\sH{{\mathscr H}}

\def\sM{{\mathscr M}}
\def\sP{{\mathscr P}}
\def\sT{{\mathscr T}}
\def\sY{{\mathscr Y}}

\def\sV{{\mathscr V}}

\def\sX{{\mathscr X}}

\def\sZ{{\mathscr Z}}

\def\fA{{\mathfrak A}}
\def\fB{{\mathfrak B}}
\def\fC{{\mathfrak C}}
\def\fD{{\mathfrak D}}
\def\fE{{\mathfrak E}}
\def\fF{{\mathfrak F}}

\def\fX{{\mathfrak X}}
\def\fY{{\mathfrak Y}}
\def\fZ{{\mathfrak Z}}
\def\fT{{\mathfrak T}}

\def\fV{{\mathfrak V}}

\def\fc{{\mathfrak c}}

\def\wtilde{\widetilde}
\def\what{\widehat}

\def\veps{\varepsilon}
\def\a{\alpha}

\def\na{\nabla}

\def\sm{{\rm sm}}
\def\st{{\rm st}}
\def\an{{\rm an}}
\def\alg{{\rm alg}}
\def\sp{{\rm sp}}

\def\Spf{{\rm Spf\,}}
\def\Spec{{\rm Spec\,}}

\def\id{{\rm id\,}}
\def\ker{{\rm Ker\,}}
\def\dis{{\rm dis\,}}

\def\st{{\rm st\,}}
\def\ul{\underline}
\def\ol{\overline}

\def\limpr{\mathop{\underleftarrow{\rm lim}}}

\def\limind{\mathop{\lim\limits_{\displaystyle\rightarrow}}}

\def\kc{{k^\circ}}
\def\kcc{{k^{\circ\circ}}}
\def\kt{\widetilde{k}}
\def\wt{\what{\otimes}}

\author{Francesco Baldassarri\thanks{Universit\`{a} di Padova,
Dipartimento di matematica pura e applicata, Via Trieste, 63, 35121 Padova, Italy.}}

\title{Continuity of the radius of convergence of  differential equations on $p$-adic analytic curves}

\begin{document}

\date{March 28, 2010}

\maketitle


\tableofcontents

\setcounter{section}{-1}
\section{Introduction}
\subsection{Radius of convergence of $p$-adic differential equations}
\par
It is well-known that any system of linear differential equations with complex analytic coefficients on a complex open disk admits a full set of solutions convergent on the whole disk.
It is also well known that the same fact does not hold over a
non-archimedean field $k$ that contains the field of $p$-adic numbers. For example, solutions of the equation $\ds \frac{d \, f}{dT} = f$ are convergent only on an open disks of radius $|p|^{\frac{1}{p-1}}$. The question is actually of interest over any 
non-archimedean field $k$ \emph{of characteristic zero} (we do not  exclude the case of a trivially valued $k$). What is the behaviour of the radius of convergence of  a system of differential equations with $k$-analytic coefficients as a function on points of the $k$-analytic affine line~? Of course, one has to give a precise meaning to this question since points of $k$-analytic spaces are not necessarily $k$-rational. We recall in fact  \cite[1.2.2]{Berkovich} that to any point
$x$ of a $k$-analytic space $X$, one associates a completely valued extension
field ${\sH}(x)$ of $k$, called the {\it residue field at $x$}; the point $x$ is {\it $k$-rational} (resp. \emph{rigid}) iff $\sH(x) = k$ (resp. $[\sH(x) : k]$ is finite)\,; $X(k)$  (resp. $X_0$) denotes the subset of $k$-rational (resp. rigid) points of $X$. On a smooth $X$, $k$-rational points admit a fundamental system of open neighborhoods 
which are open polydisks,
but this is not the case for general points, not even  using the \'etale topology \cite[Corollary 2.3.3]{BerkInt}. 
\smallskip
\par
Let $X$ be a relatively compact analytic domain in the $k$-analytic affine line $\A^1=\A^1_k$, and suppose we are given a system 
 \beq \label{diffCHDWintro} 
\Sigma : \frac {d\, \vec y}{dT}=G\,\vec y \; ,
\eeq
of linear differential equations, with $G$ a $\mu \times \mu$ matrix of $k$-analytic functions on $X$. If $x\in X$ is a $k$-rational point, let $R(x) = R(x,\Sigma)$ denote the radius of the maximal open disk in $X$ with center at $x$ on which all solutions of $\Sigma$ converge. If $x$ is not necessarily $k$-rational, the 
$\sH(x)$-analytic space $X_{\sH(x)} := X\what{\otimes} \sH(x) \subset \A^1_{\sH(x)}$ contains a canonical $\sH(x)$-rational point $x^\p$ over $x$ which corresponds to the induced character $\sH(x)[T] \to \sH(x)$, and we define $R(x)$ as the number $R(x^\p)$ for the system of equations on 
$X_{\sH(x)}$ induced by $\Sigma$. A more precise formulation of the above question is as follows. What is the behavior of the function $X \to \R_{>0}$~, $x \mapsto R(x)$~? For example, is it continuous~? The latter question is not trivial since a precise formula for the radius of convergence involves the infimum limit of an infinite number of continuous real valued functions on $X$ (\cf \S  \ref{raddef}). 
\par
A problem of the above type was considered for the first time in the $p$-adic case in a paper by Christol and Dwork (in a slightly different setting). Namely, let $X$ be an affinoid domain in $\A^1$ that contains the open annulus
\beq \label{openann}
B(r_1,r_2) = \{ x \in \A^1 \, | \, r_1 < |T(x)| < r_2 \,\} \; .
\eeq
Then there is a continuous embedding $[r_1,r_2] \hookrightarrow X$, $r \mapsto t_{0,r}$, where $t_{0,r}$ is the maximal point of the closed disk of radius $r$ with center at zero. Christol and Dwork proved that, for any system of differential equations $\Sigma$ on $X$, the function $[r_1,r_2] \to \R_{>0}$~, $r \mapsto \inf (r, R(t_{0,r}))$ is continuous \cite[Th\'eor\`eme 2.5]{ChristolDwork}.
\par
The following example, which we learned from Christol,   should convince the reader that the variation of the radius of convergence presents, in general, some non obvious  features.
\begin{exa}\label{Dworkfct} (\cite{Ch2},\cite[III.10.6]{RoCh}) 
Let $k$ be a non-archimedean extension of $\Q_{p}$, $p>2$,  containing an element $\pi$ such that $\pi^{p-1} =-p$. It is well-known that  the \emph{Dwork exponential} 
$$F(T) = \exp \pi (T - T^{p}) =  \exp \l( \pi T + \frac{(\pi T)^{p}}{p}\r)
$$  converges precisely for 
$|T| <  p^{\beta}$, where $\beta = { \frac{p-1}{p^{2}}}$. 
We can check this by comparison with the Artin-Hasse exponential series 
\beq
 \exp (T + \frac{T^{p}}{p} +  \frac{T^{p^{2}}}{p^{2}} + \dots ) \in \Z_{p}[[T]] \; ,
\eeq
 evaluated at $\pi T$, which converges for $|T(x)| < p^{\frac{1}{p-1}}$ ($ > p^{\beta}$).
 A crucial remark here, and all over the present example, is that if we have a relation
 $$f(T) = g(T)h(T)
 $$
 among invertible power series $f,g,h \in k[[T]]$, and if the radii of convergence of $g$ and $g^{-1}$  (resp. $h$ and $h^{-1}$) coincide, while the radii of convergence of $g$ and $h$ are distinct, then the radius of convergence of $f(T)$ is the minimum of the radii of convergence of  $g(T)$ and $h(T)$.
 \par
Now, $F(T)$ is a local power series solution at 0  of the differential equation
 \beq \label{Dworkexp} 
\Sigma ~: \frac {d\, y}{dT}= \, \pi (1 - pT^{p-1})\, y \; ,
\eeq
so that $R(0,\Sigma) = p^{\beta}$.
%
A more interesting result  is that if we set $\delta := \frac{1}{p-1} - \frac{1}{p^{2}} $, so that
 \beq \beta =  \frac{p-1}{p^{2}} < \frac{1}{p} <  \delta = \frac{1}{p-1} - \frac{1}{p^{2}} <  \frac{1}{p-1}\;  , 
 \eeq
then, for $R(t) = R(t,\Sigma)$,
 \beq \label{radformula}
   R(t) = \left\{\begin{array}{ccc} p^{\beta} &\mbox{if} & |t| < p^{\delta}\;\; , \\ \\
        p|t|^{1-p} &\mbox{if} &  |t| > p^{\delta} \;\; .
        \end{array} \right .
 \eeq
 So, at least if $|t| \not= p^{\delta}$, the following picture illustrates the behavior of $R(t)$.
 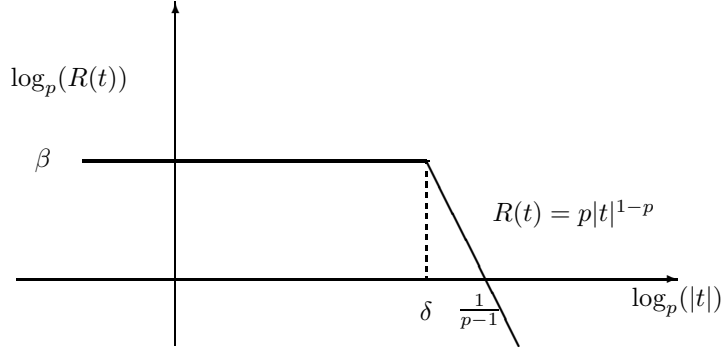
\begin{figure}[ht]
\begin{picture}(220,130)(-50,0)     
\put(0,25){\vector(1,0){250}}       
\put(60,0){\vector(0,1){130}}         
\put(155,25){\dashbox{2}(0,45){}} 
\thicklines              
\put(25,70){\line(1,0){130}}
\put(155,70){\line(1,-2){35}}
\put(10,70){\makebox(0,0){${\beta}$}}
\put(250,11){\makebox(0,0)[b]{$\log_p(|t|)$}}
\put(155,13){\makebox(0,0){$\delta$}}
\put(175,13){\makebox(0,0){$\frac1{p-1}$}}
\put(20,100){\makebox(0,0){$\log_p(R(t))$}}
\put(180,50){\makebox(0,0)[l]{$R(t)=p|t|^{1-p}$}}
\end{picture}
\caption{\label{fig 1} Radius of convergence of  $\Sigma$ at $t$, for $|t| \not= p^{\delta}$.}
\end{figure}
\begin{proof} 
It follows from \cite{Pu}  that if 
 $\pi_{1}$ is such that $\pi_{1}^{p} + p\pi_{1} = \pi$
  the series 
 \beq
 \exp ( \pi_{1} x + \pi x^{p}/p) 
 \eeq
converges precisely for  $|x| <1$. Therefore, if  $\zeta$ satisfies $\zeta^{p} = p$ and $b:= \zeta \frac{\pi_{1}}{\pi}$, the radius of convergence of 
  \beq
 \exp \l( \pi (x^{p} +  b x) \r) = {\rm E}(x^{p} + b \; x) \; ,
 \eeq
where  ${\rm E}(x) :=  \exp \pi x$, is exactly  $|\zeta|^{-1} = p^{1/p}$. 
 The calculation of $R(t)$ now follows from the product decomposition of the 
 solution $f(x,t) \in \sH(t)[[x-x(t)]]$ of $\Sigma$ at $t$ with $f(t,t) =1$, 
 \begin{multline} \label{proddec} 
f(x,t) = {\rm E}(x^{p} -x - t^{p} +t) = \\
{\rm E}\l((x-t)^{p} + b(x-t)\r)\;  {\rm E}\l((-1-b+pt^{p-1})(x-t)\r) \;
\prod_{k=2}^{p-1}{\rm E}\l({{p}\choose{k}} t^{p-k} (x-t)^{k}\r)
\; .
 \end{multline}
 In fact, the previous estimate shows that the dominant term in (\ref{proddec} ) is 
 ${\rm E}\l((-1-b+pt^{p-1})(x-t)\r)$. 
 
 \begin{figure}[ht]
\begin{picture}(300,170)(-50,0)     
\put(0,25){\vector(1,0){280}}       
\put(60,0){\vector(0,1){165}}         
\put(125,25){\dashbox{2}(0,50){}}  
\put(155,25){\dashbox{2}(0,20){}}
\put(25,75){\line(1,0){140}}
\put(75,125){\line(1,-1){80}}      
\put(25,125){\line(2,-1){160}}    
\thicklines              
\put(25,45){\line(1,0){130}}
\put(155,45){\line(1,-1){45}}
\put(15,75){\makebox(0,0){${\frac1p}$}}           
\put(10,45){\makebox(0,0){${\beta}$}}
\put(250,11){\makebox(0,0)[b]{$\log_p(|t|)$}}
\put(122,13){\makebox(0,0){$\frac1p$}}
\put(155,13){\makebox(0,0){$\delta$}}
\put(175,13){\makebox(0,0){$\frac1{p-1}$}}
\put(-3,140){\makebox(0,0){$\log_p(|x-t|)$}}
\put(100,110){\makebox(0,0)[l]{$|x-t|=p|t|^{1-p}$}}
\put(175,55){\makebox(0,0)[l]{$|x-t|=p^{1/k}|t|^{1-p/k}$}}
\end{picture}
\caption{\label{fig 2} Radii of convergence of factors in  (\ref{proddec})}
\end{figure}
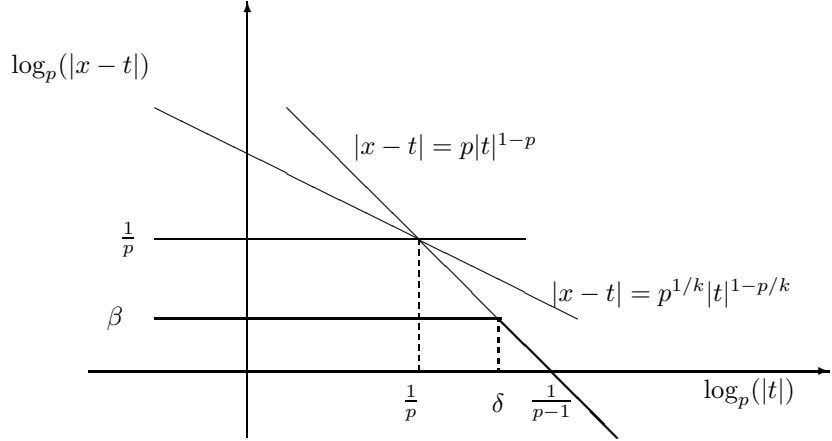

Therefore 
 \beq
 R(t) = |-1-b+pt^{p-1}|^{-1} \;.
 \eeq 
 Now, $b^{p} = -1 + \frac{p\pi_{1}}{\pi}$, so that $|1+b^{p}| = p^{-1 + 1/p} > |p|$. Consideration of the Newton polygon of the polynomial 
 $$(x-1)^{p} - b^{p} = x^{p}- p x^{p-1} + \dots -1 - b^{p}
 $$
 which kills $1+b$, shows finally  that 
 $$
 |1+b| = p^{-\beta} \;.
 $$
  This concludes the proof since 
  $$ |-1-b+pt^{p-1}| = \max(p^{-\beta}, |p||t|^{p-1}) \; ,
  $$
  unless the equality $p^{-\beta} = |p||t|^{p-1}$ takes place, in which case $|t| = p^{\delta}$. 
 \end{proof}

 In the region $|t| = p^{\delta}$ we do not get from (\ref{radformula}) a complete description of $R(t)$. The transfer theorem (\cite[IV.5]{DGS} or (\ref{contiguity}) below)  implies that $R(t) \geq p^{\beta}$, and  the main result of this paper,  \ie continuity of $t \mapsto R(t)$, together with  proposition \ref{advances2}, imply $R(t) = p^{\beta}$  in almost all residue classes 
 $D = D(c, (p^{\delta})^{-})$ with $|c| =p^{\delta}$. This is because,  by formula \ref{pentes2}, if $R(x) > p^{\beta}$ in  $D$, then in fact $R(x) \geq  p^{\beta + \delta}/|x-c|$ all over $D$. Then, if $x \in D$, $R(x) < p^{\beta + \veps}$, for $\veps >0$, can only take place in 
 $ A_{D,\veps} := \{ x \in D \,| \, |x-c| \geq p^{\delta - \veps} \,\}$. But a union $\bigcup_{D} A_{D,\veps}$ contains a neighborhood of $t_{0,\delta}$  in $D(c, (p^{\delta})^{+})$ only if $D$ runs over a finite set of classes.

  \par
 By a  closer examination of (\ref{proddec}) Christol shows that there are  $p-1$ exceptional open disks of radius  $p^{\delta}$, centered in the zeros  $c$ of the polynomial  $-1 - b + pc^{p-1} = 0$, where $R(t)$ is strictly bigger than $p^{\beta}$.  Notice that $|c| = p^{\delta}$ for such zeros.  In such an exceptional class, one obtains \emph{a priori}
 \beq \label{excclass}
   R(t) = \left\{\begin{array}{ccc} 
  p^{\frac{1}{2}\l( 1- \delta (p-2)\r) } &\mbox{if} & |t-c| < p^{\frac{1}{2}\l( 1- \delta (p-2)\r) } \; \;, \\  \\
|t-c|^{-1} p^{ 1- \delta (p-2)} &\mbox{if} & p^{\frac{1}{2}\l( 1- \delta (p-2)\r) }  <  |t-c| < p^{\delta} 
        \;\; .
        \end{array} \right .
 \eeq
 We point out that 
\beq
p^{\beta} < p^{\frac{1}{2}\l( 1- \delta (p-2)\r) }  < p^{\frac{1}{p}} \; ,
\eeq
so that the value of $R(t)$ in the exceptional classes is bigger than the value at their boundary, as  expected.
\par
Now, for 
$$|t-c| = p^{\frac{1}{2}\l( 1- \delta (p-2)\r) } \; ,$$
we have  
$R(t) \geq p^{\frac{1}{2}\l( 1- \delta (p-2)\r) }$ by the transfer principle. But  $R(t) > p^{\frac{1}{2}\l( 1- \delta (p-2)\r) }$ is impossible, since the disk $D(t, R(t)^{-})$ would properly intersect the region $p^{\frac{1}{2}\l( 1- \delta (p-2)\r) }  <  |t-c| < p^{\delta}$, where the second line in (\ref{excclass}) gives a different exact value.
\par 
The conclusion is that if  $-1 - b + pc^{p-1} = 0$ one has 
\beq \label{excclass2}
   R(t) = \left\{\begin{array}{ccc} 
  p^{\frac{1}{2}\l( 1- \delta (p-2)\r) } &\mbox{if} & |t-c| \leq p^{\frac{1}{2}\l( 1- \delta (p-2)\r) } \;\; , \\  \\
|t-c|^{-1} p^{ 1- \delta (p-2)} &\mbox{if} & p^{\frac{1}{2}\l( 1- \delta (p-2)\r) }  \leq  |t-c| < p^{\delta} 
         \;\; .
        \end{array} \right .
 \eeq
 The behavior of $R(t)$ in any of the $p-1$ exceptional classes is summarized by the following picture.
 \begin{figure}[ht]
\begin{picture}(220,140)(-50,0)     
\put(0,25){\vector(1,0){250}}       
\put(60,0){\vector(0,1){140}}         
\put(135,25){\dashbox{2}(0,75){}}  
\put(165,25){\dashbox{2}(0,45){}} 
\put(60,70){\dashbox{2}(100,0){}} 
\thicklines              
\put(135,100){\line(1,-1){30}}
\put(25,100){\line(1,0){110}}
\put(165,70){\line(1,-2){35}}
\put(0,90){\makebox(0,0){$\frac{1}{2}\l( 1- \delta (p-2)\r)$}}
\put(250,11){\makebox(0,0)[b]{$\log_p(|t-c|)$}}
\put(125,15){\makebox(0,0){$\frac{1}{2}\l( 1- \delta (p-2)\r)$}} 
\put(165,15){\makebox(0,0){$\delta$}}
\put(35,70){\makebox(0,0){$\beta$}}
\put(185,15){\makebox(0,0){$\frac1{p-1}$}}
\put(20,130){\makebox(0,0){$\log_p(R(t))$}}
\put(190,50){\makebox(0,0)[l]{$R(t)=p|t-c|^{1-p}$}}
\put(170,90){\makebox(0,0)[l]{$R(t)=p^{ 1- \delta (p-2)} |t-c|^{-1}$}}
\end{picture}
\caption{\label{fig 3} Radius of convergence of $\Sigma$  in an exceptional class}
\end{figure}
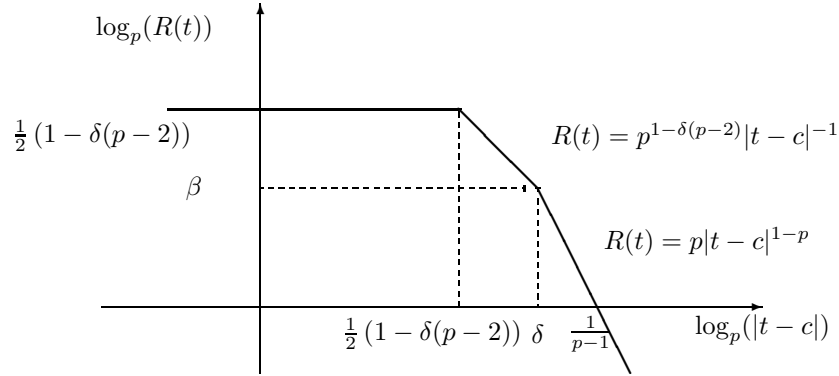

The previous calculations give a complete description of the function $R(t, \Sigma)$ on  $\A^1$. The conclusion is summarized in Fig. 4, where the graph $\Gamma \subset \A^{1}$ has the property that, for any $t \in \A^{1}(k)$, $R(t, \Sigma)$ is constant in any open disk  in $\A^{} \setminus \Gamma$ containing $x$.  Such a $\Gamma$ is said to \emph{control} the system $\Sigma$. In a subsequent paper \cite{polygon} we will show that every connection on a curve is controlled by a finite graph in a much stronger sense. 
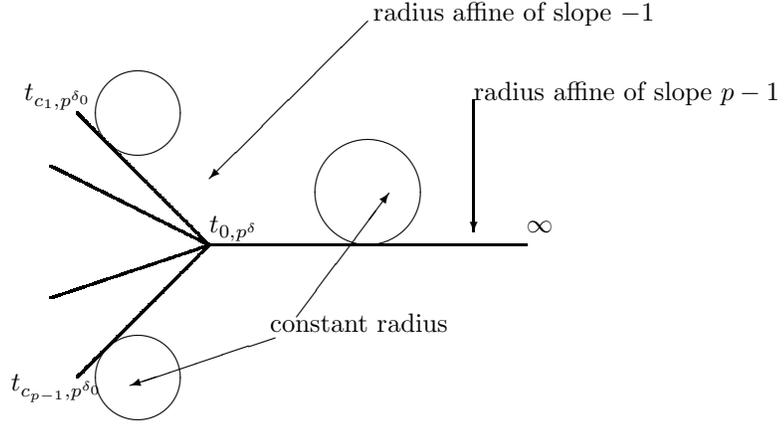
\begin{figure}[ht]
\begin{picture}(180,210)(-100,-20)
\def\NL#1#2{\qbezier(#1)(#1)(#2)}
\linethickness{0.7pt}
\put(80,85){$t_{0,p^{\delta}}$}
\put(200,85){$\infty$}
\NL{80,80}{200,80}
\put(10,135){$t_{c_{1},p^{\delta_{0}}}$}
\NL{30,130}{80,80}
\NL{30,30}{80,80}
\NL{20,60}{80,80}
\put(53, 30){\circle{30}}
 \put(105, 45){\vector( - 3, - 1){55}}
 \put(103,47){constant radius}
  \put(113, 53){\vector( 3, 4){35}}
  \put(140, 165){\vector( - 1, - 1){60}}
  \put(142,165){radius affine of slope $-1$}
    \put(180,135){\vector( 0, - 1){50}}
  \put(180,135){radius affine of slope $p-1$}
\put(53, 130){\circle{30}}
\put(140, 100){\circle{80}}
\NL{20,110}{80,80}
\put(5,25){$t_{c_{p-1},p^{\delta_{0}}}$}
\end{picture}
\caption{\label{fig4} The controlling graph $\Gamma$ of $\Sigma$}
\end{figure}

\end{exa}

\par
A particular case of our main result states that  for any relatively compact analytic domain $X$ in $\A^1$ the function $x \mapsto R(x)$ is continuous on the whole space $X$. Our result in fact establishes the property of continuity for more general analytic curves. (Recall that in comparison with the complex analytic case a smooth $k$-analytic space is not necessarily locally isomorphic to an affine space.)
To formulate the result, we first define a so-called  ``normalized''  radius of convergence $\cR(x)$, which is related to the function $R(x)$ via a simple formula but does not depend on the embedding of the analytic domain $X$ into $\A^1$. After that we explain how the function $\cR(x)$ can be defined in a more general setting on one-dimensional smooth\footnote{in the sense of rigid geometry, \ie \emph{rig-smooth} \cf definition \ref{rigsmooth} below} affinoids, which are not necessarily analytic subdomains of $\A^1$, and even in arbitrary dimensions. 
\smallskip

\par
Let $X$ be an affinoid domain in $\A^1_k$, and assume at first that $X$ is strict and that $k$ is non-trivially valued and algebraically closed. Recall that the set of all points of $X$ which have no open neighborhood isomorphic to an open disk, is a closed subset of $X$, called the \emph{analytic skeleton} of $X$ and denoted by $S(X)$ \cite[\S 4]{Berkovich}. The complement $X \setminus S(X)$ is a disjoint union of open sets $\cD_y(X)$, where $y \in \cD_y(X)$ is a $k$-rational point, isomorphic to the standard open unit disk 
 $$D_k(0,1^-) = \{ x \in \A_k^1 \,|\,  |T(x)| < 1 \, \} \; ,$$
 via a \emph{normalized} coordinate $T_y: \cD_y(X) \xrightarrow{\ \sim\ } D_k(0,1^-)$ such that $T_y(y) = 0$.
 Furthermore, each of the above open disks has a unique boundary point at $S(X)$, and so there is a canonical continuous retraction map 
 $\tau_X: X \to S(X)$.
 The topological space  $S(X)$ carries a canonical structure of a finite polygon ${\bf S}(X) = (S(X),\cV(X),\cE(X))$ whose set of vertices $\cV(X)$  consists of the points  which have no neighborhood isomorphic to an open annulus: they are all of type (2) in the sense of  \cite[1.4.4]{Berkovich}. The set of open edges $\cE(X)$ of ${\bf S}(X)$ is formed by the connected components of $S(X) \setminus \cV(X)$; an  open edge $E = E(u,v)$ connects precisely two vertices $u,v$ and $E \cup \{u,v\}$ is a closed subset of $X$, 
 canonically homeomorphic to the closed interval $[r,1]$ for a well defined $r \in |k| \cap (0,1)$, 
 with $u \mapsto 1$, $v \mapsto r$. 
  All points of the Shilov boundary $\Gamma(X)$ of $X$ are among the vertices of ${\bf S}(X)$.   For $x \in S(X)$, the set 
 $\tau_X^{-1}(x) \setminus \{ x \}$ is a disjoint union of maximal open disks $\cD_y(X)$ with boundary point $x$ in $X$, all of the same radius $r(x)$ with respect to the standard coordinate in $\A^1$. Then, $r(x)$ is the  \emph{radius} of the point $x \in \A^1$ defined on p. 78 of \cite{Berkovich} and the function $r : S(X) \to \R_{>0}$ is continuous. Notice that if $r(x) \notin |k|$, $\tau_X^{-1}(x) \setminus \{ x \} = \emptyset$. If $k^\p$ is an algebraically closed non-archimedean field over $k$, then ${\bf S}(X  \what{\otimes} k^\p)  \xrightarrow{\ \sim\ }  {\bf S}(X)$ under the natural projection.  
 In the general case (with $k$ possibly trivially valued), 
  we define $S(X)$ 
  as the image in $X$ of $S(X  \what{\otimes} k^\p)$, for any non-trivially valued algebraically closed non-archimedean extension field $k^\p/k$ such that $X  \what{\otimes} k^\p$ is strict, and similarly for open edges and vertices.  We have a canonical continuous retraction $\tau_X: X \to S(X)$ induced by the similar retraction on $X  \what{\otimes} k^\p$. We call the triple ${\bf S}(X) = (S(X),\cV(X),\cE(X))$, a \emph{generalized subpolygon} of $X$. 
 \par
 Suppose we are given a system of differential equations on the affinoid domain $X \subset \A^1$ with analytic coefficients, \ie a free  $\cO_X$-Module of finite rank $\cF$ with a connection $\na$. If $k$ is algebraically closed, $X$ is strict, and $x \in X(k)$ is a $k$-rational point of $X$,  we define the \emph{normalized} radius of convergence $\cR(x) = \cR_X(x,(\cF,\na))$ as the radius of convergence of $\cF^{\na}$ around $x$ in the corresponding maximal open disk neighborhood $\cD_x(X) \cong D(0,1^-)$. If $k$, $X$ and $x$ are arbitrary, we consider an algebraically closed non-archimedean field $k^\p$ over $k$ such that $X \wt_k k^\p$ is strict and so large that $\sH(x)$ admits an isometric $k$-embedding $\varphi$ in $k^\p$. Then, there is a canonical point $x^\p \in X \wt_k k^\p$
 and we set 
 $\cR(x) = \cR(x^\p)$, a quantity independent of the choice of $k^\p$ and of $\varphi$. The function $\cR : X \to \R_{>0}$ does not depend on the embedding of $X$ in $\A^1$, either, and one has $R(x) = \cR(x) \cdot r(\tau_X(x))$ for all $x \in X$. The function $r(\tau_X(x)) = \delta_{\what{\P^1}}(x,X)$ was called the \emph{diameter of $X$ at $x$} (with respect to the embedding $X \hookrightarrow \A^1 \hookrightarrow \P^1 = (\what{\P^1})_{\eta}$) in \cite{Lucia}; its continuity (at least under the present assumptions) was proven in \lc [3.3]. 
 By definition, the function $\cR(x)$ is preserved by extensions of the ground field, and so it suffices to study its behavior in the case when $X$ is strictly $k$-affinoid.
 \smallskip
 \par
 We now notice that, independently of the characteristic, if $k$ is algebraically closed and non-trivially valued, 
 a strictly $k$-affinoid, as well as a projective, rig-smooth curve $X$ is the generic fiber $\fX_{\eta}$ of a strictly semistable\footnote{When $\kc$ is not noetherian, this definition is not completely standard \cf  (\ref{STRss}) below.} formal scheme 
 $\fX$ over $\kc$.
 We say that $\fX$ is a  \emph{strictly semistable model}  
 of $X$.  
 \par
 We will assume, for simplicity, that $k$ is  non-trivially  valued and algebraically closed. 
To any  semistable formal  scheme $\fX$ of relative dimension one over $\kc$, one  associates (\cf \cite[Chap. IV]{Berkovich}, modulo the equivalence between such formal schemes and formal coverings of smooth compact curves with semistable reduction),
 a triple of the previous type  ${\bf S}(\fX) =(S(\fX),\cV(\fX),\cE(\fX))$ supported on a closed subset $S(\fX)  \subset X := \fX_{\eta}$, called the \emph{skeleton} of $\fX$. Unless $\fX$ is \emph{strictly} semistable (\ref{STRss}), the subpolygon  ${\bf S}(\fX)$ may have loops, \ie open edges with a single boundary vertex in $X$. 
 
%
\par
For $\fX$ and $\fY$  semistable models of $X$, we set $\fX \leq \fY$ if there exists a morphism $\fY \to \fX$ inducing the identity on the generic fiber (identified with $X$ for both). 
  Then, if $\fX \leq \fY$,  ${\bf S}(\fX)$ is a ``subpolygon" of  ${\bf S}(\fY)$, in the sense that $S(\fX) \subset S(\fY)$, $\cV(\fX) \subset \cV(\fY)$, and an open edge of ${\bf S}(\fX)$ is a union of vertices and open edges of ${\bf S}(\fY)$. There is a natural continuous retraction  $\tau_{\fX,\fY}: S(\fY) \to S(\fX)$. Now, any strictly $k$-affinoid curve $X$, as much as any projective smooth curve which is neither rational nor a  Tate curve,  admits a \emph{minimum} semistable model $\fX_{0}$ and  the skeleton ${\bf S}(X)$ of $X$ coincides with the skeleton ${\bf S}(\fX_{0})$ of the formal scheme $\fX_{0}$\footnote{For a rational curve $X$, $S(X)$ is empty, while for any minimal semistable model $\fX$ of $X$, ${\bf S}(\fX) = (\eta, \{ \eta \},\emptyset)$, for a point $\eta \in X$ 
  with $\wtilde{\sH(\eta)}$ purely transcendental of degree one over $\kt$. For a Tate curve $X$, a choice of $\eta \in S(X)$ with $\wtilde{\sH(\eta)}$ of transcendence degree one over $\kt$, determines  a homeomorphism   
 $S(X) \xrightarrow{\ \sim\ } (\R_{>0}, \cdot)/r^{\Z} \approx S^{1}$, $\eta \mapsto 1$, up to orientation, 
 for $r  \in |k| \cap (0,1)$. In this case,  ${\bf S}(X) = (S(X),\emptyset,\{ S(X)\})$. The choice of a  minimal semistable model $\fX$ of $X$ corresponds to the choice of $\eta \in S(X)$~: ${\bf S}(\fX) = (S(X), \{ \eta \},\{S(X) \setminus \{\eta\}\})$.}. Moreover, as a topological space, $\ds X = \limpr_{\fY \geq \fZ}(S(\fY), \tau_{\fZ, \fY})$  where $\fY$ and $\fZ$ run over a cofinal system of  semistable models of $X$, is a quasi-polyhedron (\cf \cite[IV.1]{Berkovich}) and, if the minimum semistable model $\fX_{0}$ of $X$ exists, the retraction $\tau_X: X \to S(X)$ is the natural projection to $S(\fX_{0}) = S(X)$. 
 \par \noindent 
  A semistable formal scheme $\fX$ over $\kc$ is
  an example of a nondegenerate polystable  formal scheme over $\kc$ (see \cite{BerkContr}). In the case of dimension one over $\kc$, the class of such formal schemes coincides with the class of semistable formal schemes.  Recall  \cite[5.2]{BerkContr} that every nondegenerate polystable formal scheme $\fX$ has a skeleton $S(\fX) \subset \fX_{\eta}$ and a retraction map $\tau_{\fX}: \fX_{\eta} \to S(\fX)$ which are preserved under any ground field extension functor. Furthermore, if $x$ is a $k$-rational point of $\fX_{\eta}$, then there is a well-defined maximum open neighborhood $D_{\fX}(x,1^-)$ of $x$ in $\fX_{\eta} \setminus S(\fX)$, which is isomorphic to the standard open unit polydisk $D_{k}(0,1^-)$ (with $x \mapsto 0$), so that the open or closed polydisks $D_{\fX}(x,r^{\pm})$, $r < 1$ with center at any $k$-rational point of $\fX_{\eta}$ are well-defined (we are talking about polydisks with equal radii). 
   \smallskip
 \par
Let $k$ be any non-archimedean field.  \emph{All $k$-analytic spaces considered in this paper are supposed to be separated.} We generalize  a definition of  \cite[1.1]{BerkInt}, as follows.
\begin{defn} \label{rigsmooth} A $k$-analytic space $X$ is said to be \emph{rig-smooth} (resp. \emph{of pure dimension $n$}) if, for any non-archimedean field $k^\p$ over $k$, and any connected strictly affinoid domain $V \subset X\what{\otimes}k^\p$, $\Omega_V^1$ is a locally free $\cO_V$-Module (resp. of rank $n$).
\end{defn} 
 Assume now $k$ is of characteristic zero, let $X$ be a rig-smooth $k$-analytic space, and let $X_G$ be the associated $G$-analytic space.  We denote by ${\bf MIC}(X/k)$ the category of pairs consisting of  a locally free $\cO_{X_G}$-Module
 $\cF$ of finite type and  of an integrable $k$-linear connection $\na $ on $\cF$, with the usual (horizontal) morphisms. 
Notice that, unlike in the classical case, for an object $(\cF,\na)$ of 
 ${\bf MIC}(X/k)$, the abelian sheaf 
$$
\cE^{\na} = \ker (\na : \cE \to \cE \otimes \Omega^1_{X/k}) 
$$
 for the $G$-topology of $X$ is not  in general locally constant.
 \begin{defn} \label{Robbaconn} Let $k$ be a non-archimedean field  of characteristic zero,  and  $X$ be a rig-smooth $k$-analytic space. We say that an object $(\cF,\na)$ of 
 ${\bf MIC}(X/k)$ is a \emph{Robba connection on $X$} if for any  algebraically closed complete extension $k^\p$ of $k$, and any open polydisk $D^\p \subset X^\p = X \what{\otimes} k^\p$, if we denote by $(\cE^\p, \na^\p)$ the $k^\p$-linear extension of $(\cE, \na)$ on $X^\p$, 
the sheaf $(\cE^{\p \, \na^\p})_{|D^\p} = (\cE^\p_{\,|D^\p})^{\na^\p_{|D^\p}}$ is  constant.   We denote by ${\bf MIC}^{Robba}(X/k)$ the full abelian subcategory of ${\bf MIC}(X/k)$, consisting of Robba connections.
\end{defn}

\begin{exa} \label{Turrittin} Let $\C$ be the complex field, equipped with the trivial valuation. On the $\C$-analytic space $\A^1_{\C}$, we consider a system of linear differential equations $\Sigma$ of the form (\ref{diffCHDWintro}), with $G$ a matrix of rational functions in $\C(T)$. To analyze the singularity  of $\Sigma$ at $T=0$, we consider the open analytic domain 
$$X = \{x \in \A^1_{\C} \, |\,  0 < |T(x)| < 1 \, \} \subset \A^1_{\C} \; .$$ 
Notice that $X = S(X) \cong (0,1)$ via $\rho \mapsto t_{0,\rho}$. The field $\sH(t_{0,\rho})$ is the non-archimedean field over $\C$, $\cK_{\rho}=  \C((T))$, equipped with the absolute value $||~||_{\rho}$, such that $||0||_{\rho} =0$ and
$|| \sum_n a_n T^n||_{\rho}= \rho^m$, if $a_n =0$, for $n <m$ and  $a_m \not=0$.
Let $X^\p := X \what{\otimes} \cK_{\rho}$, and let $t^\p_{0,\rho}$ be the canonical $\cK_{\rho}$-rational point in $X^\p$ above $t_{0,\rho}$. Then, writing $t = T( t_{0,\rho})$,
 the maximal open disk containing $t^\p_{0,\rho}$ in $X^\p$ is $D_{\cK_{\rho}}(t^\p_{0,\rho},\rho^-)$, with normalized coordinate
$\ds \frac{T-t}{t}$,
and the fact that $\Sigma$ is  a Robba connection is equivalent to the property that the singularity at $T=0$ is regular. In fact, by the formal Turrittin theory \cite{Advances}, it is enough to check this statement for the rank one differential equations
$$ \frac {d\,  y}{dT}=\frac{\alpha}{T}\, y \; ,\; \alpha \in \C \; ,
$$
and
$$ \frac {d\,  y}{dT}=\frac{\alpha}{T^N}\, y \; ,\; \alpha \in \C^{\times} \; , \; N =2,3,\dots \; .
$$
In the former case  the solution at $t^\p_{0,\rho}$ is $\ds y = \sum_i {\alpha \choose i} (\frac{T-t}{t})^i$, converging for $\ds |\frac{T-t}{t}|~<~1$. In the latter case a simple calculation shows that the solution at $t^\p_{0,\rho}$ is a power series $\ds \sum a_i(t) (\frac{T-t}{t})^i \in \cK_{\rho}[[\frac{T-t}{t}]]$, having the same radius of convergence as the series $\ds \exp (t^{-N}\frac{T-t}{t})$, that is converging for 
$\ds |\frac{T-t}{t}|~<~|t|^N$. We recall from \lc that $N$ is the Poincar\'e-Katz rank of irregularity of the corresponding differential equation at $T=0$.
\end{exa}

Notice that if $(\cF,\na)$ is any connection, 
the natural map 
 \beq
\label{eq:RHsheafintro}
\cF^{\na}    \otimes_k \cO_{X} \longrightarrow  \cF \;  ,
\eeq
is  an isomorphism. So,  if  $D$ is a connected  analytic domain in  $X$ and $\cF^{\na}$ is constant over $D$, then 
$(\cF,\na)_{|D}$  is isomorphic to the trivial connection $(\cO_D,d_{D/k})^{\mu}$, for $\mu = \rk \, \cF_{|D}$. In particular, if $\cF^{\na}$ is locally constant, then it is locally isomorphic to the locally constant sheaf $k^{\rk \cF}$. 
\par

Assume now that the  $k$-analytic space $X$ is  isomorphic to the generic fiber $\fX_{\eta}$ of a   nondegenerate polystable formal scheme $\fX$ over $\kc$. 
For any object  $(\cF,\na)$ of ${\bf MIC}(X/k)$, and for 
a $k$-rational point $x \in \fX_{\eta}$, we may define the (\emph{$\fX$-normalized}) \emph{radius of convergence $\cR(x) = \cR_{\fX}(x,(\cF,\na))$} (\ref{limrad}). For $x \in X(k)$, it  is the supremum of $r \in (0,1]$ such that the restriction of $\cF^{\na}$ to $D_{\fX}(x,r^-)$ is constant. Notice that the restriction of $\cF^{\na}$ to $D_{\fX}(x,\cR(x)^-)$ is then locally constant \emph{for the usual topology}\footnote{This means for the topology of a Berkovich analytic space.}, and is therefore constant, since an open disk is simply connected for the usual topology. 
If $x$ is arbitrary,  we set $\cR(x) = \cR(x^\p)$, where $x^\p$ is the canonical $\sH(x)$-rational point in $
 \fX_{\eta} \what{\otimes} \sH(x) = (\fX \what{\otimes}_{\kc}  \sH(x)^{\circ} )_{\eta}$ over $x$. 
Let $\varphi: \fY \to \fX$ be an \'etale morphism of  nondegenerate  polystable formal schemes over $\kc$, and let $\varphi_{\eta}: Y \to X$ be the generic fiber of $\varphi$. For any object  $(\cF,\na)$ of ${\bf MIC}(X/k)$, $\varphi_{\eta}^{\ast}(\cF,\na)$ is an object of ${\bf MIC}(Y/k)$ and for any point $y \in Y$
we have
\beq \label{etaleinv}
 \cR_{\fY}(y,\varphi_{\eta}^{\ast}(\cF,\na)) =  \cR_{\fX}(\varphi_{\eta}(y),(\cF,\na)) \; .
\eeq
Under the same assumptions, if $k^{\p}/k$ is any extension of non-archimedean fields, the 
$k^{\p \, \circ}$-formal scheme $\fX^{\p} := \fX \what{\otimes} k^{\p \, \circ}$ is nondegenerate polystable, and, $\psi: \fX^{\p}  \to \fX$ is the natural projection
\beq \label{basechange}
 \cR_{\fX^{\p}}(x,\psi_{\eta}^{\ast}(\cF,\na)) =  \cR_{\fX}(\psi_{\eta}(x),(\cF,\na)) \; .
\eeq
  \begin{conj} \label{conjcont} Let  $X = \fX_{\eta}$ be the generic fiber $\fX_{\eta}$ of a  nondegenerate  polystable formal scheme $\fX$ over $\kc$ and $(\cF,\na)$ be an object of ${\bf MIC}(X/k)$.   The   function $X \to \R_{>0}$~, $x \mapsto \cR_{\fX}(x,(\cF,\na))$ is continuous. 
 \end{conj}
 Notice that, by formula \ref{etaleinv}, the conjecture holds if it holds under the restrictive condition that $\fX$ is  strictly polystable over $\kc$ (base change to a finite extension field $k^\p/k$ is irrelevant). 
 Let ${\bf MIC}_{\fX}(X/k)$ be the full subcategory of ${\bf MIC}(X/k)$ consisting of the objects $(\cF,\na)$, such that  $\cF$ is the $\cO_X$-Module associated to some locally free $\cO_{\fX}$-Module of finite type $\fF$, called a \emph{formal model} of $\cF$ \emph{over $\fX$}.  
 In case $\cF$ admits a formal model $\fF$, the calculation of $\cR(x)$ may be performed as follows. We  pick  an open affine neighborhood $\fY$ of the specialization of $x$ in $\fX$  such that $\fF$ is free on $\fY$. Then $x \in \fY_{\eta}$, $D_{\fX}(x,1^-) = 
 D_{\fY}(x,1^-)$,  and $\cR_{\fX}(x,(\cF,\na)) = \cR_{\fY}(x,(\cF,\na)_{|\fY_{\eta}})$ is the maximum common radius of convergence of all sections of $\cF^{\na}$ in $D_{\fY}(x,1^-)$ around $x$, when expressed in terms of a basis of global sections of $\fF$ on $\fY$. In dimension one, this amounts to solving a system of differential equations as in (\ref{diffCHDWintro}), \emph{over an affinoid domain in a smooth projective curve}.  
 \par
  \begin{conj} \label{conjcont2} Let  $X = \fX_{\eta}$ be the generic fiber $\fX_{\eta}$ of a  nondegenerate  polystable formal scheme $\fX$ over $\kc$ and $(\cF,\na)$ be an object of 
  ${\bf MIC}_{\fX}(X/k)$.   Then  $(\cF,\na)$
  is a Robba connection on $X$ if and only if the function $x \mapsto \cR_{\fX}(x,(\cF,\na))$ is identically 1 on $X$. 
 \end{conj}

 \par \medskip
 And here is our main result.
 \begin{thm} \label{mainthm} Let $k$ be a non-archimedean field extension of $\Q_p$ and $X$ be a rig-smooth compact strictly $k$-analytic curve. Conjectures \ref{conjcont}  and \ref{conjcont2} hold for objects of ${\bf MIC}_{\fX}(X/k)$, where $\fX$ is any semistable model of  $X$.
 \end{thm}
 \par
Surprisingly enough,  before our joint paper with Di Vizio \cite{Lucia}, conjecture \ref{conjcont}  seemed to be open even in the case when  $\fX = \Spf \kc\{T\}$, hence $\fX_{\eta} = X$ is the closed unit disk $D_{k}(0,1^{+})$, a case  extensively discussed in the literature (\cf \cite{Ke}, \cite{DGS} and \cite{ChMeAST} for reference,  and \cite[10.4.3]{RoCh} for a partial result  in our direction). A direct proof, in the case of an affinoid domain $X$ of $\A^1$,  was given in \cite[5.3]{Lucia}. Notice that, always by (\ref{etaleinv}), it suffices to prove theorem 
\ref{mainthm} in the special case when $\fX$ is either a ``formal disk"  or a ``formal annulus", 
\ie is affine connected and admits a dominant  \'etale morphism $\varphi$ to  $\Spf \kc \{X\}$ or to $\Spf \kc \{X,Y\}/(XY-a)$, for $a \in \kcc \setminus \{0\}$, respectively. The previously mentioned result of \cite{Lucia}, implies the theorem when $\varphi$ is an open immersion. In the general case however, we must appeal to the more powerful statement in section \ref{ContLemma} below. 
\par
Our detailed description of the function $x \mapsto \cR_{\fX}(x,(\cF,\na))$ will in fact show the following
 \begin{thm} \label{conjRobba} Let  $k$, $X$ and $\fX$  be as in  theorem \ref{mainthm}, and let $(\cF,\na)$ be an object of ${\bf MIC}_{\fX}(X/k)$.  If $\cR_{\fX}(\xi,(\cF,\na)) =1$ at any vertex  $\xi$ of the skeleton ${\bf S}(\fX)$ of $\fX$, then $x \mapsto \cR_{\fX}(x,(\cF,\na))$ is identically 1 on  $X$. 
 Moreover, 
 $(\cF,\na)$ is a Robba connection on $X$. 
   \end{thm}
\par
\begin{defn} \label{crystal} Let  $k$, $X$ be as in  \ref{mainthm} and let $\fX$ be a semistable formal model of $X$.  An object $(\cF,\na)$  of ${\bf MIC}(X/k)$ is said to be \emph{$\fX$-convergent} if it is an object of 
${\bf MIC}_{\fX}(X/k)$ and $\cR_{\fX}(x,(\cF,\na)) =1$ identically on $X$.  We denote by 
${\bf MIC}^{{\fX}-conv}(X/k)$ the full subcategory of ${\bf MIC}(X/k)$ consisting of $\fX$-convergent objects. 
\end{defn}

 Let $\varphi: \fY \to \fX$ be an \'etale morphism of  semistable $\kc$-formal schemes and let $\varphi_{\eta}: Y \to X$ be its generic fiber.  
 It follows from formula \ref{etaleinv} that an object $(\cF,\na)$ of   ${\bf MIC}_{\fX}(X/k)$ is  $\fX$-convergent  if and only if  the object $\varphi_{\eta}^{\ast}(\cF,\na)$ of   ${\bf MIC}_{\fY}(Y/k)$ is $\fY$-convergent.
\begin{cor}  \label{colimit} The category 
${\bf MIC}^{Robba}(X/k)$ is the 2-colimit category of its full subcategories 
${\bf MIC}^{{\fX}-conv}(X/k)$, where $\fX$ runs over the semistable models of $X$.
\end{cor}
Notice that, despite the previous result,  the introduction of formal models is necessary for several reasons:
 \begin{enumerate}
 \item The analytic definition  of $\cD_x(X)$ given above for $x$ a $k$-rational point of the smooth affinoid $X$, as the maximum open neighborhood of $x$ in $X$ isomorphic to the polydisk $D(0,1^-)$, does not work in higher dimensions,  where such a \emph{maximum} neighborhood need not exist. 
 \item
 Even in the case of an affinoid $X \subset \A^1$, and $x \in X(k)$, the above analytic definition  of  $\cD_x(X)$, does not globalize well. If $X = X_1 \cup X_2$ is a union of two affinoids, and $x$ is a $k$-rational point of $X_1$, one can only say that  
 $\cD_x(X_1)$ and  $\cD_x(X_2)$ are contained in $\cD_x(X)$. The assignment of a semistable model $\fX$ of $X$, remedies to this problem, \ie  $D_{\fX_1}(x,1^-) = D_{\fX}(x,1^-)$, provided the decomposition $X = X_1 \cup X_2$ originates from a Zariski open covering $\fX = \fX_1 \cup \fX_2$ of formal schemes. Notice that if $X$ is the analytic projective line, there is no maximum disk containing a given point. 
 \item According to conjecture \ref{conjcont2}, it seems likely that consideration of a formal model $\fX$ of $X$ should be of lesser  importance in the study of the directional logarithmic derivative of $\cR(x)$ along any tangential vector to $X$. In other words, a generalized  theory of \emph{slopes} in the style of \cite{ChMe3}
should only depend upon the connection on the analytic curve $X$. 
 \end{enumerate}

 We finally propose, for future use,  a more general definition, where we consider various possible topologies  on categories of analytic spaces over $X$. We have in mind, in particular, the natural (Berkovich) topology on $X$, the $G$-topology on $X$ described in  \cite[1.3]{BerkovichEtale}, the \'etale topology of  \cite[4.1]{BerkovichEtale}, the quasi-\'etale topology of  \cite[\S 3]{BerkovichCycles}.
 \begin{defn} \label{CONV} Let $k$  and $X$ be as in definition \ref{Robbaconn} and let $\tau$ be a topology on $X$. An object $(\cF,\na)$ of ${\bf MIC}(X/k)$ will be said to be \emph{$\tau$-convergent} if, $\tau$-locally, 
  $(\cF,\na)$ is a Robba connection. We will denote by ${\bf MIC}^{\tau-conv}(X/k)$
   the full subcategory of ${\bf MIC}(X/k)$ consisting of $\tau$-convergent objects. For $\tau$ the natural (resp. the $G$-, resp. the \'etale, resp. the quasi-\'etale) topology, ${\bf MIC}^{\tau-conv}(X/k)$ will be denoted by  ${\bf MIC}^{conv}(X/k)$ (resp.  ${\bf MIC}^{G-conv}(X/k)$,  resp.  ${\bf MIC}^{et-conv}(X/k)$, resp.  ${\bf MIC}^{qet-conv}(X/k)$).
 \end{defn}
\subsection{Basic definitions and description of contents}
\label{basicdefn}
\medskip
The organization of the paper is as follows. 
In section \ref{Graph} we review the 
structure of a rig-smooth compact connected strictly $k$-analytic curve $X$ over a non-archimedean field $k$ with non-trivial valuation and the parallel theory of semistable reduction of $k$-algebraic curves.  \emph{For simplicity, we assume that $k$  is algebraically closed.} 
Most, but not all, of these results may be extracted from \cite[Chap. IV]{Berkovich} or \cite{BL0}.  
We describe systematically  the equivalence between the notions of  semistable model  of  $X$, of  formal affinoid covering of $X$ with semistable reduction, of \emph{semistable partition} 
 of $X$ (\ref{formalpartitions}), and of  \emph{complete subpolygon}  of $X$  (\ref{subsect:graphs}). A well-known consequence of these equivalences is that all results on semistable reduction of curves admit in principle both an algebraic and an analytic proof.   
\par
The algebraic theory of semistable models  of families of algebraic curves \cite{CA} has  recently made important progress  thanks to Temkin \cite{Te}, who has eliminated in particular the properness assumption, and has clarified minimality of the stable model. It should be possible to apply Temkin's method to completely describe the reticular structure of the category $\cF\cS(X)$ of semistable models of $X$, and in particular to showing that, unless $X$ is either rational or a Tate curve, a minimum semistable model $\fX_{0}$ of $X$ exists, and its skeleton coincides with the analytically defined skeleton of $X$.  
We have not seriously attempted to follow this algebraic strategy. Instead, we approach the problem via  
an analytic discussion of intersections and unions of disks and annuli in the style of  \cite[Prop. 5.4]{BL0}: we content ourselves with a precise account of the proofs. Details, which would lead us astray of the main path, have been collected in \cite{FormalCurves}.

\par  
Over an algebraically closed field $k$, the notions of open or closed disk or annulus, present no difficulty\footnote{This is one of the main reasons why we make such assumption on $k$.}. Open or closed disks or annuli in $\A^{1}_{k}$ are called \emph{$k$-rational}, if their center in a $k$-rational point and their radii belong to $|k|$. Similarly, for open or closed disks or annuli in general.  
A \emph{punctured disk} is a disk punctured at a rigid point, and it is \emph{$k$-rational}, if the full disk is so. 

\par
Open annuli, viewed as simply connected quasi-polyhedra,  have two endpoints. An  \emph{open segment}  (resp.  \emph{open half-line}) in $X$ is the skeleton of an open annulus (resp. of an open punctured disk) in $X$:   the annulus (resp. the punctured disk) is then uniquely defined.  An open segment  may have one or two boundary points called \emph{ends} in the compact curve $X$; they are points of type (2) or (3). The segment is \emph{$k$-rational} iff the corresponding annulus is, \ie  if and only if its ends are points of type $(2)$. 
In general, an open segment $E$ is  a locally closed subset of $X$  and it admits a  continuous $(1:1)$ parametrization $\rho \mapsto \eta_{\rho}$ by an open interval $(r,1) \subset (0,1)$  (resp. $[r,1) \subset (0,1)$), with $r$ the \emph{height} of the annulus, which is  canonical up to the  inversion $\rho \mapsto r/\rho$. Similarly, an open half-line is canonically parametrized by $(0,1)$, where $\eta_{\rho}$ tends  to a rigid point as $\rho \to 0$. A half-line is \emph{$k$-rational} iff the corresponding punctured disk is. 
\par
We define the category $\cG\cP(X)$ (resp. $\cG\cP^c(X)$) of (resp. complete) subpolygons  $\Gamma = (|\Gamma|,\cV_{\Gamma},\cE_{\Gamma})$
of $X$. Here $|\Gamma|$ is a closed connected subset of $X$, $\cV_{\Gamma}$ is a finite subset of points of type $(2)$ of $X$ 
(``vertices" of $\Gamma$) and $\cE_{\Gamma}$ is a finite set of open $k$-rational segments of $X$ (``open edges" of $\Gamma$), such that $|\Gamma|$ is the disjoint union of all vertices and open edges of $\Gamma$.  

A subpolygon $\Gamma$ of $X$ is ``complete" if $X \setminus |\Gamma|$ is a union of open disks. The category $\cG\cP^c(X)$ is really a directed partially ordered set, where 
$\Gamma \leq \Gamma^\p$ if $|\Gamma| \subset |\Gamma^\p|$,  $\cV_{\Gamma}\subset \cV_{\Gamma^\p}$, so that an open edge of $\Gamma$ is a union of open edges and vertices of $\Gamma^\p$.  
We  prove (more details in  \cite{FormalCurves}) that $\cG\cP^c(X)$ is  isomorphic to the category $\cF\cS(X)$ of semistable formal models of $X$ and morphisms inducing the identity on the generic fiber (they turn out to be precisely  the ``admissible blow-ups" of  \cite{BL1},
\cite{bosch}). The isomorphism ${\bf S}: \cF\cS(X) \xrightarrow{\ \sim\ } \cG\cP^c(X)$ is the \emph{skeleton functor}. 
Unless $X$ is a rational projective curve or  a Tate curve, the category   $\cF\cS(X)$  admits a minimum $\fX_0$. The skeleton ${\bf S}(\fX_0) = (S(\fX_0),\cV(\fX_0), \cE(\fX_0))$ is the skeleton ${\bf S}(X) = (S(X),\cV(X),\cE(X))$ of the analytic curve $X$, in the sense of the previous subsection. 
 \par
 
We define below \cf (\ref{basicannuli}) a ``basic formal disk" (resp. a  ``basic formal annulus") $(\fB,T)$ (resp. or $(\fB,S = a/T)$, for $a \in \kcc \setminus \{0\}$). 
We write  $(\cB,T)$ for the corresponding affinoid disk (resp. annulus). The coordinate $T$ in $\fB$ defines, for any $z \in \cB(k)$, a normalized coordinate $T_z :  D_{\fB}(z,1^-) \xrightarrow{\ \sim\ } D_k(0,1^-)$, with $z \mapsto 0$,  
 by 
 \beq
  \label{coorcann}
T_z(y) =  \left\{ \begin{array}{lll} T(y) - T(z)\; , \;  & \mbox{if $(\cB,T)$ is a disk\, ,}
\\
&
\\
(T(y) - T(z))/T(z)\; , \;  & \mbox{if $(\cB,T)$ is an annulus \, ,}
 \end{array} \right .
\eeq
$\forall y \in D_{\fB}(z,1^-)$. 
Notice that the basic formal annulus $\fB$ is the minimum  semistable model of $\cB$, so that ${\bf S}(\fB) = {\bf S}(\cB)$, $D_{\fB}(z,1^-) = D_{\cB}(z,1^-)$, $\tau_{\cB} = \tau_{\fB}$ and $|T(x)| = |T(\tau_{\cB}(x))| = r(T(\tau_{\cB}(x)))$, for any $x \in \cB$. 
\par \noindent
The main point for us is the fact that  any semistable model $\fX$ of  the curve $X$  admits an \'etale covering by formal disks or annuli.  
This allows us to recover and globalize most  results of the classical  non-archimedean theory of linear differential systems \cite{DGS}.
\par

In section \ref{ContLemma} we prove a general criterion due to Berkovich (private communication)  to test whether a  function $X \to \R_{>0}$, on a  smooth $k$-affinoid curve $X$ is continuous:
it represents an abstract and more efficient version of the method of proof used in \cite[5.3]{Lucia}.
In section \ref{raddef}, we start assuming that $k$ has characteristic zero. We pick a semistable model $\fX$ of $X$ and introduce the function ``normalized radius of convergence at $x$" $\cR_{\fX}(x,(\cF,\na))$, for an object $(\cF,\na)$ of ${\bf MIC}(X/k)$. 
Actually, in view of the most common applications, we also consider the following more general situation: $X$ is the complement of a finite reduced divisor $\cZ = \{z_1,\dots,z_r \}$ in a compact rig-smooth connected $k$-analytic curve $\ol{X}$. So, if $\ol{X} = \cM(\sA)$, $z_1,\dots,z_r $ correspond to $r$ distinct maximal ideals of the $k$-affinoid algebra $\sA$. We assume  that  $\ol{X}$ is the generic fiber of a semistable $\kc$-formal scheme  $\ol{\fX}$. It is well known that  $\cZ$   determines a finite flat generically \'etale closed reduced  subscheme $\fZ$ of $\ol{\fX}$. 
We will assume that $\fZ$ is \'etale over $\kc$. 
For the limited purposes of this paper, however, consideration of $\fZ$ will not be necessary, and the previous assumption will be concealed under the requirement  that  the disks  $D_{\ol{\fX}}(z_1,1^-), \dots, D_{\ol{\fX}}(z_r,1^-)$ are distinct and that their boundary points $\tau_{{\ol{\fX}}}(z_1), \dots, \tau_{{\ol{\fX}}}(z_r)$ are vertices of ${\bf S}(\ol{\fX})$.  We keep the notational distinction mainly for future use. 
\par
For any object $(\cF,\na)$ of ${\bf MIC} (X/k)$,
we define  the function $x \mapsto  \cR_{\ol{\fX}, \fZ}(x, (\cF,\na))$,  $X \to (0,1]$. 
We restrict our attention to the abelian tannakian subcategory ${\bf MIC} (\ol{X}(\ast \cZ)/k)$ of ${\bf MIC}(X/k)$, consisting of the objects $(\cF,\na)$ \emph{having  meromorphic singularities at $\cZ$}. This means that 
$\cF$ is the restriction to $X$ of a coherent and locally free $\cO_{\ol{X}}$-Module  $\ol{\cF}$ 
and for any open $U \subset \ol{X}$ and any section $e \in \Gamma(U,  \ol{\cF})$, there is a non-zero $f \in \cO_{\ol{X}}(U)$, such that $ f \na(e) \in \Gamma(U, \ol{\cF} \otimes  \Omega^1_{\ol{X}})$. Morphisms are horizontal $\cO_{X}$-linear maps, with meromorphic poles at $\{z_1,\dots,z_r\}$. 
\par
We are mostly concerned with the full abelian tannakian subcategory ${\bf MIC}_{\ol{\fX}}(\ol{X}(\ast \cZ)/k)$ of 
$ {\bf MIC} (\ol{X}(\ast \cZ)/k)$.  It consists of objects 
$(\cF,\na)$ of $ {\bf MIC} (\ol{X}(\ast \cZ)/k)$ such that 
 the $\cO_X$-Module $\cF$ is the restriction to $X$ of   a coherent and locally free  
 $\cO_{\ol{\fX}}$-Module  $\ol{\fF}$.
The category 
${\bf MIC}_{\ol{\fX}}(\ol{X}(\ast \cZ)/k)$ admits  natural internal tensor product and internal ${\cH}om$ and ${\cE}nd$. 
\par
Let us assume, from now on in this section, that $k$ is an extension of $\Q_p$.  
We say that an object $(\cF,\na)$ of ${\bf MIC}_{\ol{\fX}}(\ol{X}(\ast \cZ)/k)$ \emph{satisfies condition $\bf NL$ at $z \in \cZ$}, if  the formal Fuchs exponents of  (the regular part of) ${\cE}nd ((\cF,\na))$ at  $z$ are $p$-adic non-Liouville numbers. An old result of the author \cite[Prop. 4]{Advances}, which applies under  condition $\bf NL$, relates the  asymptotic behavior of $x \mapsto \cR_{\ol{\fX}, \fZ}(x, (\cF,\na))$ for $x \to z_i$ to the algebraic irregularity 
$\rho_{z_i}(\cF,\na)$ of the connection at $z_i$, namely
\beq \label{advances}                                                                                                                                                                                                                                                                                                                                                                                                                                                                
\cR_{\ol{\fX}, \fZ}(x, (\cF,\na)) \sim |T_{z_i}(x)|^{\rho_{z_i}(\cF,\na)}  \;\;  \; \; , \; \; {\rm as} \; x \to z_i \; ,
\eeq 
where $T_{z_i}$ is  the normalized coordinate  on $D_{\ol{\fX}}(z_i,1^-)$ with $T_{z_i}(z_i)=0$.
Our main result is  that, for any object $ (\cF,\na)$ of ${\bf MIC}_{\ol{\fX}}(\ol{X}(\ast \cZ)/k)$, the function $x \mapsto  \cR(x) = \cR_{\ol{\fX}, \fZ}(x, (\cF,\na))$ is continuous on $X$. We first reduce the problem of continuity of this function to  the standard situation of a connection $(\cF,\na)$ with $\cF$ free of finite rank $\mu$ over a basic affinoid disk or annulus.
\par
 For the convenience of the reader, we recall in section \ref{review} the classical theory of differential systems on an annulus, largely due to Dwork, Robba, Christol and Mebkhout. One should also refer to the elegant account  \cite{Ke}, for more recent results.  These authors treat the case of a system of ordinary differential equations (\ref{diffCHDWintro}) defined on the open  annulus $B(r_1,r_2)$ (\ref{openann}), with coefficients in the Banach $k$-algebra $\sH(r_1,r_2)$ of {\it analytic elements} on 
 $B(r_1,r_2)$: $\sH(r_1,r_2)$ is the completion of the $k$-subalgebra of $k(T)$, consisting of rational functions with no pole in 
 $B(r_1,r_2)$, under the supnorm on $B(r_1,r_2)$. 
\par Our main technical tool is explained in section \ref{Dwork-Robba}: it is  a  generalization of the  Dwork-Robba theorem
\cite[IV.3.1]{DGS} on effective bounds  for the growth of local
solutions (theorem (\ref{cor:DworkRobba1}) and its corollaries).  Essentially, we must replace in the original formulation analytic elements in the previous sense on the  open annulus $B(r_1,r_2) \subset \A^1$, with functions which extend to analytic functions on a basic affinoid annulus containing $B(r_1,r_2)$ as an open analytic subspace. 
The Dwork-Robba theorem is the essential step in
our proof of the upper semicontinuity of the radius of convergence
(\cf \S\ref{subsection:USC}) and in generalizing the continuity
result of Christol-Dwork \cite[2.5]{ChristolDwork} to the present  situation \S \ref{CHDWHigherGenus}. Section 4 of the joint work with Di Vizio \cite{Lucia} contains a several variable generalization of this result. We point out that the upper semicontinuity is precisely the non-trivial part of  \cite[2.5]{ChristolDwork}, and that our proof (as given here and in \cite{Lucia}) differs from the one of  Christol-Dwork. 
We conclude the proof of  continuity of $x \mapsto \cR(x)$ in section  \ref{Conclusion}.  As in the classical case,  we obtain a more precise description of $x \mapsto \cR(x)$. 
 Namely, let $k_{0}$ be a closed subfield of $k$, and assume $\ol{\fX} = \ol{\fX}_{0} \wt \kc$, where  $\ol{\fX}_{0}$ is a $k_{0}^{\circ}$-semistable formal scheme.
  Assume the points $z_{1}, \dots,z_{r}$ come from $k_{0}$-rational points of  $(\ol{\fX}_{0})_{\eta}$ (identified with the $z_{i}$'s) and 
 the object $(\cF,\na)$ of ${\bf MIC}_{\ol{\fX}}(\ol{X}(\ast \cZ)/k)$ from an object $(\cF_{0},\na_{0})$ of ${\bf MIC}((\ol{\fX}_{0})_{\eta}(\ast \cZ)/k_{0})$ such that $\cF_{0}$ is the inverse image on $(\ol{\fX}_{0})_{\eta}$ of a coherent and locally free $\cO_{\ol{\fX}_{0}}$-module $\fF_{0}$.
Let $\fY_{0} \geq \ol{\fX}_{0}$ be any $k_{0}^{\circ}$-semistable
model of $(\ol{\fX}_{0})_{\eta}$, and let $E$  be 
any oriented open edge  of 
${\bf S}(\fY_{0})$, with $|k_{0}|$-rational ends  in $(\ol{\fX}_{0})_{\eta}$. 
Let 
$(r_E,1) \xrightarrow{\ \sim\ } E$, $\rho \mapsto \eta_{\rho}$, be the canonical parametrization, with $r_E \in |k_{0}| \cap (0,1)$. Then  the restriction to $(r_E,1)$ of $x \mapsto \cR(x)$, namely $\rho \mapsto \cR(\rho) = \cR(\eta_{\rho})$,   is the infimum of the constant 1 and a finite set of functions of the form
\beq
\label{rslinearity1}
 |p|^{1/(p-1)p^h}|b|^{1/jp^h}\rho^{s/j}\; ,
\eeq
 where $j \in \{ 1, 2, \dots,\mu - 1\}$, $s \in \Z$,  $h \in \Z \cup \{\infty\}$, $b \in k_{0}^{\times}$.  Arbitrarily high $h$ can appear even in the simplest rank 1 case of the equation killing $x^{\a}$ \cf \cite[IV.7.3 $(iv)$]{DGS} and our account below (\ref{xalpha}). 
 \par
 
 This statement in the classical situation appeared in \cite{Pons}: we provide here a hopefully more convincing proof\footnote{The treatment of  \cite{ChMe3} is unfortunately restricted to the solvable case, where negative slopes do not appear.}.   We point out the novelty of using Dwork's technique of descent by Frobenius on basic affinoid annuli which are not necessarily affinoid subdomains of $\A^{1}$, \cf \S \ref{Frobenius}.
 \par
 If $(\cF,\na)$  has meromorphic singularity at a $z_i \in \cZ$,  and the previous $p$-adic non-Liouvilleness assumption holds at $z_i$, then a similar result holds on the half-line  $E$ connecting $z_i$ with the  point $\tau_{\ol{\fX}}(z_i)$ at the boundary of $D_{\ol{\fX}}(z_i,1^-)$. In terms of the canonical parametrization  via the function ``radius of a point"  
$\rho: \ol{E} = E \cup \{ z_i, \tau_{\ol{\fX}}(z_i) \} \xrightarrow{\ \sim\ } [0,1]$, 
 the restriction of $x \mapsto \cR(x)$ to $\ol{E}$  is precisely of the form
(\ref{rslinearity1}), with moreover $\max \{s/j \} = \rho_{z_i}(\cF,\na)$. 
A consequence  is
\begin{thm} \label{Robbacond3} Let $(\cF,\na)$  be an object of ${\bf MIC}_{\ol{\fX}}(\ol{X}(\ast \cZ)/k)$ satisfying condition $\bf NL$ at each $z \in \cZ$. Then $\cR_{\ol{\fX}, \fZ}(x,(\cF,\na)) = 1$ identically for $x \in X$
if and only if both
\begin{enumerate}
\item
$\cR_{\ol{\fX}, \fZ}(\xi,(\cF,\na)) = 1$, for each vertex $\xi$ of ${\bf S}(\ol{\fX})$,
\item
$(\cF,\na)$  has regular singularities along $\cZ$. 
\end{enumerate}
Under the previous assumptions 
$(\cF,\na)$ is a Robba connection on $X$.
\end{thm}
\par
It follows from \cite[4.5.3]{Berkovich} that
\begin{prop} If $f:X \to Y$ is a non-constant holomorphic map of rig-smooth $k$-analytic curves, and   $(\cF,\na)$ is a Robba connection on $Y$, then $f^{\ast}(\cF,\na)$ is a Robba connection on $X$. 
\end{prop}
\medskip
{\bf Acknowledgements.} 
Continuity of the radius of convergence was the main target of a collaboration with Lucia Di Vizio on the long delayed, and still unpublished,  manuscript \cite{Lucia} dedicated to the higher dimensional generalization of the theory of Dwork-Robba and Christol-Dwork. 
While we plan to pursue that joint higher-dimensional program with Di Vizio in the next future, the present author 
had the good fortune of arousing the interest of Vladimir Berkovich in these problems. The influence of Berkovich and a series of letters he generously addressed to the author, modified the original plan. Namely, Berkovich insisted that the global 1-dimensional case should be treated first, using the theory of semistable reduction, and provided a detailed plan of the results to prove, and some of the proofs. To the author's regret, he did not accept to be a co-author. 
 \par
 The collaboration with Lucia over the past few years  has greatly contributed to shape the ideas appearing in this paper.
\par
We are
also grateful to Yves Andr\'e and Kiran Kedlaya for showing so much
interest in our results, to Pierre Berthelot for a question he asked on Robba connections. It is a pleasure to 
acknowledge the generous help we received from Michel Raynaud on the topic of formal geometry. 
We thank  Ahmed Abbes, Maurizio Cailotto, Gilles Christol, Antoine Ducros,  Adrian Iovita, Qing Liu,   Michael Temkin and especially Lorenzo Ramero for many useful discussions.

\section{The structure of rig-smooth compact strictly $k$-analytic curves}
\label{Graph}
\begin{notation}
All over this section,  the valuation of $k$ is assumed to be non-trivial. In fact, except in subsection \ref{semistability}, which might be of independent interest, $k$ will be assumed to be algebraically closed, in order to minimize technical difficulties. All $k$-analytic spaces are assumed to be separated. As in \cite[\S 3]{dejong}, a $k$-analytic curve is a $k$-analytic space pure of dimension 1. All over this and the next section,  $X$ denotes  a compact connected rig-smooth strictly $k$-analytic curve. For any field $L$  an \emph{$L$-algebraic curve}  is a  separated $L$-scheme of finite type, of pure dimension 1. 
\end{notation}   
Any $k$-analytic curve is a good analytic space, \ie every point of it has an affinoid neighborhood, and is paracompact  \cite[\S 3]{dejong}. 
An irreducible  compact $k$-analytic curve is either the analytification of a projective curve or it is  affinoid  \cite{FM}, \cite[Prop. 3.2]{dejong}. 
So, in our case, $X$ is either the analytification $\cX^{\an}$ of a smooth projective $k$-algebraic curve $\cX$ or it is strictly $k$-affinoid; its underlying topological space  is a quasipolyhedron \cite[4.1.1]{Berkovich}.
\subsection{Semistability}
\label{semistability}
Classically, one gives the following definition. 
\begin{notation} \label{forsch} Let $L$ be a non-archimedean field over $k$. An  $L^{\circ}$-scheme (resp. $L^{\circ}$-formal scheme) is \emph{admissible}  if it is reduced, quasi-compact, separated,  of  (resp. topologically)  finite presentation, and flat over $L^{\circ}$.  The algebra of an affine admissible $L^{\circ}$-scheme (resp. $L^{\circ}$-formal scheme) is also said to be \emph{admissible}.
\end{notation} 
In the present article, all $\kc$-schemes and $\kc$-formal schemes will be assumed to be  admissible. 
\par 
For an admissible $\kc$-scheme $\sZ$, we define the \emph{completion along the closed fiber} or simply \emph{completion} of $\sZ$ as the admissible $\kc$-formal scheme 
\beq \what{\sZ} = \what{\sZ}_{/\sZ_{s}} :=  \limind_{n} \sZ \otimes K^{ \circ}/(\pi^{n}) \; 
\eeq
where $\pi \in \kcc \setminus \{0\}$. The definition is clearly independent of the choice of $\pi$ and is functorial. 

\begin{defn} \label{defspec}   
  For an admissible $\kc$-formal scheme $\fY$, we denote by
 $\sp : \fY_{\eta} \to \fY_s$ the \emph{set-theoretic specialization map} and by 
  $\sp_{\fY}  : ( \fY_{\eta})_G \to \fY$ the \emph{specialization map} viewed as a morphism of ringed $G$-topological spaces. 
 \end{defn}
 \begin{defn} A morphism $\varphi : \fX \to \fY$ of $\kc$-formal schemes, is  \emph{proper} if for any $\pi \in \kcc \setminus \{0\}$ the morphism of $\kc/(\pi)$-schemes $\varphi_{0} : \fX \otimes  \kc/(\pi) \to \fY \otimes  \kc/(\pi)$ is proper.
\end{defn}
This definition agrees with \cite[III.1, 3.4.1]{EGA} when $\kc$ is noetherian and $\fX$, $\fY$ are admissible.
We recall \cite[\S 2]{BerkovichCycles} that a morphism $\varphi : \fX \to \fY$ of admissible $\kc$-formal schemes, is  said to be \emph{\'etale} if $\varphi_{0} $ is \'etale, for any choice of $\pi$.

\par
\smallskip
 We further abuse the classical terminology as follows. 
\begin{defn}  \label{STRss}
An  admissible formal scheme $\fX$  over $\kc$ 
is \emph{strictly semistable} if, locally for the Zariski topology, it is of the form 
\beq
\fX \xrightarrow{\ \phi\ } \Spf \kc  \{S,T\}/(ST-a) \to  \Spf \kc
\eeq
 with $a \in \kcc \setminus \{0\}$, where  $\phi$ is  \'etale.  Similarly, a reduced separated   $\kt$-scheme of finite type $\cY$ is \emph{strictly semistable} if, locally for the Zariski topology, it is of the form  
 \beq    \label{STRss1}
 \cY \xrightarrow{\ \phi\ } \Spec \kt  [x,y]/(xy) \to  \Spec \kt  \; ,
   \eeq
with  $\phi$  \'etale.
\hfill \break
The $\kc$-formal scheme $\fX$   (resp. the $\kt$-scheme $\cY$) is    \emph{semistable}  if there is a surjective \'etale morphism $\fX^\p \to \fX$ (resp. $\cY^\p \to \cY$)  with  $\fX^\p$ (resp. $\cY^\p$)   strictly semistable. 
\end{defn} The generic fiber of a semistable $\kc$-formal scheme  is rig-smooth. 
 It follows from \cite[Prop. 1.4 and step (1) in its proof]{BerkContr}  that a semistable $\kc$-formal scheme  is normal. 
 \par
The notion of  (resp. strict) semistability we use  is the
one-dimensional case of (resp. strict) nondegenerate polystability \cite{BerkContr}.  
The following result is proven in \cite{FormalCurves}. 
\begin{thm}\label{evviva}
A normal admissible formal scheme $\fX$ over $\kc$ is  (resp. strictly) semistable  if and only if  its  closed fiber $\fX_s$ is a (resp. strictly) semistable scheme over $\kt$. 
\end{thm}

\begin{defn} A normal admissible $\kc$-scheme $\sX$ is said to be (resp. \emph{strictly}) \emph{semistable} if  its special fiber is (resp. strictly) semistable.
\end{defn}
It follows from theorem \ref{evviva} that a normal admissible $\kc$-scheme $\sX$ is (resp. strictly) semistable if and only if its completion is a (resp. strictly) semistable $\kc$-formal scheme.
\par 

\subsection{Semistable models}
\emph{We assume from now on in this section that $k$ is algebraically closed.}
The following strong form of the semistable reduction theorem for  curves over $k$ (in fact over a general basis) has recently been made available  by Temkin \cite{Te}.

\begin{defn} \emph{\cite{Te}} Let $\sX$, $\sY$ (resp. $\fX$, $\fY$)  be  admissible   $\kc$-schemes (resp. $\kc$-formal schemes) of pure relative dimension 1, with smooth (resp. rig-smooth) generic fibers.
A  proper dominant morphism 
$\varphi~: ~\sX~\to~\sY $ (resp. $\varphi~: ~\fX~\to~\fY $) is called an \emph{$\eta$-modification of $\fY$} (resp. \emph{of $\fY$}) 
if the generic fiber  $\varphi_{\eta} : \sX_{\eta} \to   \sY_{\eta} $ 
 (resp. $\varphi_{\eta} : \fX_{\eta}  \to \fY_{\eta} $) is an isomorphism of $k$-algebraic curves (resp. of $k$-analytic curves).  
\end{defn}
\begin{thm} \label{strongsemistability} (Strong Semistable Reduction Theorem) Let $\sY$ be an admissible $\kc$-scheme of pure relative dimension one over $\kc$. Let us assume that the generic fiber $\cY := \sY_{\eta}$ of $\sY$ is a smooth $k$-algebraic curve. Then,  there is a 
  semistable (and even a \emph{strictly} semistable) $\kc$-scheme  $\sY^\p$ and an $\eta$-modification  
$\varphi^{\p}~: ~\sY^\p~\to~\sY$  of  $\sY$. Moreover,   the pair $(\sY^\p,\varphi^{\p} )$, with $\sY^\p$ semistable,  may be chosen to be \emph{minimal} in the following sense.  If   $\sY^{\p \p}$ is a semistable 
$\kc$-scheme and   $\varphi^{\p \p}: \sY^{\p \p} \to \sY$ is an $\eta$-modification of $\sY$,  then 
 there is a unique $\kc$-morphism $\chi: \sY^{\p\p} \to \sY^\p$, such that $\varphi^{\p \p}  = \varphi^{\p}   \circ \chi$. For any such minimal choice of $(\sY^\p,\varphi^{\p} )$, the morphism $\varphi^{\p}$ is projective and is an isomorphism over the maximal semistable (open) subscheme of $\sY$. 
\end{thm}

\begin{defn} The minimal pair $(\sY^\p,\varphi^{\p} )$ appearing in theorem \ref{strongsemistability} is unique up to unique isomorphism inducing the identity on $\sY$. 
 It will be called \emph{ the minimum semistable $\eta$-modification of $\sY$} and will be denoted by $(\sY_{\st},\varphi_{\st})$.
 \end{defn}

\par

The previous theorem has a formal counterpart as follows.

\begin{thm} \label{modelaffintro} Let $\fY$ be an admissible $\kc$-formal scheme of pure relative dimension 1. Assume the generic fiber $\fY_{\eta}$ is rig-smooth. Then,  there is  a  semistable (and even a \emph{strictly} semistable) $\kc$-formal scheme $\fY^\p$   and an $\eta$-modification 
$\phi^{\p}~: ~\fY^\p~\to~\fY $. 
Moreover,   the $\kc$-semistable formal scheme $\fY^\p$ may be chosen to be 
 \emph{minimal} in the following sense.  If  $\fY^{\p \p}$ is a  $\kc$-semistable formal scheme and a morphism 
 $\varphi^{\p \p}: \fY^{\p \p} \to \fY$ is an $\eta$-modification of $\fY$,  then
 there is a unique $\kc$-morphism $\chi: \fY^{\p\p} \to \fY^\p$, such that $\varphi^{\p \p} 
= \varphi^{\p}  \circ \chi$. For any such minimal choice of $(\fY^\p,\varphi^{\p} )$, the morphism $\varphi^{\p}$ is projective and is an isomorphism over the maximal semistable (open) formal subscheme of $\fY$. 
\end{thm}
\begin{proof} 
By uniqueness, it is enough to prove the statement for  $\fY = \Spf A$  affine.  
Under this assumption, by  Elkik's theorem \cite[Thm. 7 and Rmk. 2 on p. 587]{elkik}, one can find a finitely generated $\kc$-algebra $B$ such that $\what{B} \xrightarrow{\ \sim\ } A$ and that $B \otimes_{\kc} k $ is smooth over $k$. 
We now apply theorem \ref{strongsemistability} to the admissible affine $\kc$-scheme $\sY = \Spec B$. We consider the minimal semistable $\eta$-modification $\varphi^{\p}: \sY^{\p} \to \sY$. Then $(a)$-adic  completion $\what{\varphi^{\p}}: \fY^{\p} \to \fY$ of $\varphi^{\p}$, for any $a \in \kcc \setminus \{0\}$, satisfies the  requirements of the theorem. It is in fact clearly semistable, projective and is an isomorphism over the maximal semistable (open) subscheme of $\fY$.   It is also minimal, because  if $\phi^{\p \p}: \fY^{\p \p} \to \fY$ has the properties of the statement, we may again assume that 
$\fY^{\p \p}$ is affine. We 
 apply again Elkik's result, to show that 
$\phi^{\p \p}$ is the formal $(a)$-adic completion of a morphism  $\varphi^{\p \p}: \sY^{\p \p} \to \sY$ of admissible $\kc$-schemes. On replacing  $\sY^{\p \p}$ by an open neighborhood of its special fiber, we may assume that $\sY^{\p \p}$ is semistable, as in   theorem 
\ref{strongsemistability}. So,   there is a unique $\kc$-morphism $\chi: \sY^{\p \p}  
\to \sY^\p$, such that $\varphi^{\p \p}  =  \varphi^{\p} \circ \chi$.
Then the $(a)$-adic  completion of $\chi$ satisfies the requirements in our statement.
\end{proof}
\par 
The following result is precisely what we need to start our arguments.
\begin{prop} (Berkovich) \label{projemb} Let  $V = \sM (\sA)$  be a  rig-smooth strictly $k$-affinoid curve. There exists  a  projective strictly semistable $\kc$-scheme $\sY$   and an embedding of $V$ in the generic fiber $\fY_{\eta}$ of the formal completion $\fY$ of $\sY$ along its closed fiber, which identifies $V$ to the affinoid  $\sp^{-1}(\cZ)$ for an affine open subset $\cZ$ of $\fY_{s}$. In particular, $V$ is the generic fiber of the strictly semistable  $\kc$-formal scheme $\what{\sY}_{/ \cZ}$,  formal completion of $\sY$ along $\cZ$. 
\end{prop}
\begin{proof} 
\par  (Steps 1 and 2 hold for a higher dimensional affinoid $V$).
\par
Step 1. \emph{$V$ is isomorphic to a strictly affinoid domain in the analytification $\sV^\an$ of an affine scheme $\sV$ over $k$}. 
By the theory of Raynaud \cite[\S 2.8]{bosch}, the closed $\kc$-subalgebra $\sA^{\circ}$ of $\sA$ consisting of elements of spectral norm $\leq 1$, contains an admissible $\kc$-subalgebra $A$,  such that $A \otimes_{\kc} k  \xrightarrow{\ \sim\ }  \sA$.  
Applying Elkik's theorem \cite[Thm. 7 and Rmk. 2 on p. 587]{elkik}, one can find a finitely generated $\kc$-algebra $B$ such that $\what{B} \xrightarrow{\ \sim\ } A$ and that $B \otimes_{\kc} k $ is smooth over $k$. Then the claim holds for $\sV = \Spec (B \otimes_{\kc} k )$. 
\par
Step 2. By \cite[Lemma 9.4]{BerkContr}, there is an open embedding of $\sV$ in $\sY_{\eta}$, where $\sY$ is an integral scheme proper finitely presented and flat over $\kc$, and an open subscheme $\cZ \subset \sY_s$ such that $V = \sp_{\sY}^{-1}(\cZ)$, where $\sp: (\sY_{\eta})^{\an} \to \sY_s$ underlies  the specialization map $\sp_{\what{\sY}}: (\what{\sY})_{\eta} \to \what{\sY}$ of the formal completion 
$\what{\sY}$ of $\sY$ along $\sY_s$, and $(\what{\sY})_{\eta}$ is identified with $(\sY_{\eta})^{\an}$, by properness.  It follows that $V$ coincides with the generic fiber of the admissible $\kc$-formal scheme $\what{\sY}_{/ \cZ}$,  formal completion of $\sY$ along $\cZ$.
\par
Step 3. 
 Assume now that $V$ is one-dimensional. Then  both $\sV$ and $\sY$ are one-dimensional. By the  semistable reduction theorem for curves \ref{strongsemistability}, there is  a projective strictly semistable $\eta$-modification $\sY^\p \to \sY$. In particular, $\sY^\p$ is proper, strictly semistable, and has projective smooth generic fiber $\cY$. The existence of a very ample line bundle on $\sY^\p$ relative to $\kc$ may be proven as in \cite[\S 7]{BL0}. 
 If $\cZ^\p$ is the preimage of $\cZ$ in $\sY_s^\p$, then the lemma holds for $\fV^{\p} = \what{\sY^\p}_{/ \cZ^\p}$,  formal completion of $\sY^\p$ along $\cZ^\p$.
\end{proof}
\par
We deduce as above from theorem \ref{modelaffintro} the notion of 
the \emph{minimum  semistable $\eta$-modification of $\fY$}, which is a pair $(\fY_{\st},\varphi_{\st})$, consisting of a semistable $\kc$-formal scheme $\fY_{\st}$,  and of a morphism $\varphi_{\st} : \fY_{\st} \to \fY$: it  is  again unique up to unique isomorphism inducing the identity on $\fY$.
\par
\smallskip

\begin{defn} \label{ssmodels} Let the  $k$-analytic curve $X$ and, in case, the projective curve $\cX$ such that $X = \cX^{\an}$, be as before. 
Let $\cF\cA(X)$ be the category whose objects are pairs $(\fY,\phi)$, where $\fY$ is an admissible normal formal scheme and $\phi$ is an isomorphism $\phi: X  \xrightarrow{\ \sim\ } \fY_{\eta}$ of $k$-analytic curves. We call the pair $(\fY,\phi)$, and abusively also the formal scheme $\fY$, an \emph{admissible formal model} of $X$. A \emph{morphism} $f: (\fY,\phi) \to (\fZ,\psi)$ is a morphism of $\kc$-formal schemes $f : \fY \to \fZ$ such that $f_{\eta} \circ \phi= \psi$. We usually identify $X$ with $\fY_{\eta}$ via $\phi$, and regard the specialisation map $\sp_{\fY} : (\fY_{\eta})_{G} \to \fY$, via the composition with the morphism locally ringed $G$-topological spaces $\phi_{G}$ induced by $\phi$, as a morphism $\sp_{\fY} : X_{G} \to \fY$. The category $\cF\cA(X)$ is \emph{the category of admissible formal models\footnote{We do not specify \emph{normal}, as we should,  in the rest of this paper.} of $X$}. \hfill\break
We similarly define the \emph{category $\cS\cA(\cX)$ of admissible models of $\cX$}, namely  pairs $(\sY,\phi)$, where $\sY$ is an admissible normal $\kc$-scheme and $\phi$ is an isomorphism $\phi: \cX  \xrightarrow{\ \sim\ } \sY_{\eta}$ of $k$-algebraic curves. A \emph{morphism} $f: (\sY,\phi) \to (\sZ,\psi)$ in 
$\cS\cA(X)$ is a morphism of $\kc$-schemes $f:\sY \to \sZ$ such that $f_{\eta} \circ \phi = \psi$.
\hfill\break
We define $\cF\cS(X)$ (resp. $\cS\cS(\cX)$) to be the full subcategory of 
$\cF\cA(X)$ (resp. of $\cS\cA(\cX)$) consisting of pairs $(\fY,\phi)$ (resp. $(\sY,\phi)$) in which $\fY$ (resp. $\sY$) is semistable. Similarly for the subcategory  $\cF\cS^{\st}(X)$ (resp. $\cS\cS^{\st}(\cX)$)  of \emph{strictly} semistable models. \hfill\break
We will denote by $\cP\cS\cS(\cX)$ (resp. $\cP\cS\cS^{\st}(\cX)$) the full subcategory of $\cS\cS(\cX)$ consisting of proper (resp. strictly) semistable $\kc$-models of $\cX$.
\end{defn}
\smallskip

\begin{defn} A \emph{Tate curve over $k$} is a smooth projective $k$-algebraic curve of genus 1, which does not admit a smooth projective model over the ring of integers $\kc$. By a \emph{rational}   curve we simply mean the $k$-projective line $\P^{1}$.  The analytification of a rational (resp. of a Tate) curve is also called \emph{rational} (resp. a \emph{Tate curve}).
\end{defn}
\begin{rmk} 
A Tate curve $X$ over $k$  is $k$-analytically isomorphic to $\G_{m,k}/q^{\Z}$, where $\G_{m,k}$ is the $k$-analytic multiplicative group, while $q \in k$, $0 < |q| <1$. 
\end{rmk}

The following result completely describes the category of formal semistable models of a given $k$-analytic curve. 
\begin{thm} \label{Q-analss} (Semistable reduction of compact $k$-analytic curves) Let   $Y$ be a compact rig-smooth strictly $k$-analytic curve.  The category $\cF\cS(Y)$ of semistable models of $Y$ and morphisms of $\kc$-formal schemes inducing the identity on the generic fiber, is a non-empty partially ordered set, where $x \to y \Leftrightarrow x \geq y$. For any objects $x,y$ of $\cF\cS(Y)$, $x \vee y = \sup \{x,y\}$ exists.   Every morphism $\fY^{\p\p} \to \fY^{\p}$ in $\cF\cS(Y)$ is an admissible blow-up of an open ideal $\fA$ of finite presentation  of $\cO_{\fY^{\p}}$, invertible outside of a finite set of points
 (in particular, it is proper).  Unless $Y$ is  a rational or a Tate  projective curve, for any object $x$ of $\cF\cS(Y)$, there is a unique minimal object $m_{x}$ of $\cF\cS(Y)$, such that $x \geq m_{x}$ and, 
 for any objects $x,y$ of $\cF\cS(Y)$, $x \wedge y = \inf \{x,y\}$ exists. In particular, $\cF\cS(Y)$ admits a minimum.
\end{thm}
The proof we propose is purely analytic. It will be sketched in the next subsections. It is based on an equivalence of categories between $\cF\cS(Y)$ and a category $\cP\cS(Y)$  whose objects are  partitions of $Y$ into a finite family of open annuli and of affinoids with good canonical reduction.  The statement of the theorem, translated for $\cP\cS(Y)$, becomes an amusing discussion of intersections of disks and annuli in $Y$ in the style of \cite[Prop. 5.4]{BL0}.  For a complete proof, see \cite{FormalCurves}. 
\par
We admit here the following relatively standard result discussed in \cite{FormalCurves} (\cf \cite[III.1, Th. 5.4.5]{EGA}, \cite[Cor. 2.7]{BL0}, \cite[Thm. 3.1]{WL}.)

\begin{lemma} \label{completion} Let $\cY$ be a  smooth projective $k$-algebraic curve and let $Y = \cY^{\an}$ be the associated compact $k$-analytic curve. The completion functor  $\sZ \mapsto  \what{\sZ}$ induces  equivalences of categories 
\beq \cP\cS\cS(\cY) \xrightarrow{\ \sim\ } \cF\cS(Y) \; ,
\eeq
and
\beq \cP\cS\cS^{\st}(\cY) \xrightarrow{\ \sim\ } \cF\cS^{\st}(Y) \; .
\eeq
 
\end{lemma}
%

The following corollary of theorem \ref{Q-analss} and lemma \ref{completion} gives a complete description of the category of proper semistable models of a given smooth projective curve. We are not aware of a direct algebraic proof in the style of Temkin's theorem \ref{strongsemistability}.
\begin{thm} \label{Q-semistability} (Semistable Reduction of Projective Curves) Let   $\cY$ be a  smooth projective $k$-algebraic curve.   A proper semistable model of $\cY$ is necessarily projective. The category $\cP\cS\cS(\cY)$ of projective semistable models of $\cY$,  and morphisms of $\kc$-schemes inducing the identity on the generic fiber, is a non-empty partially ordered set, where $x \to y \Leftrightarrow x \geq y$.  For any objects $x,y$ of $\cP\cS\cS(\cY)$, $x \vee y = \sup \{x,y\}$ exists. Every morphism $\sY^{\p\p} \to \sY^{\p}$ in $\cP\cS\cS(\cY)$ is a blow-up of a closed finite $a$-torsion subscheme  of $\cO_{\sY^{\p}}$, for some non-zero $a \in \kcc$ (in particular, it is projective).  Unless $Y$ is  a rational or a Tate  projective curve, for any object $x$ of $\cP\cS\cS(\cY)$, there is a unique minimal object $m_{x}$ of $\cP\cS(\cY)$, such that $x \geq m_{x}$ and, for any objects $x,y$ of $\cP\cS\cS(\cY)$, $x \wedge y = \inf \{x,y\}$ exists. In particular, $\cP\cS\cS(\cY)$ admits a minimum.
\end{thm}

 \subsection{Formal coverings and formal models} \label{formalmodels}
  \par
 For the definition and basic properties of formal affinoid coverings (always assumed to be \emph{strictly affinoid} in this paper) we refer to the first part of \cite[\S 4.3]{Berkovich}. We want to describe a functorial equivalence between (resp. strictly) semistable models and the (resp. strictly) semistable coverings of $X$ described below. 
 \medskip   
 \par
The set of  formal coverings  of $X$ forms a 
category ${\cC}ov(X)$ in which an arrow  $\cU \to \cV$ is a {\it refinement} of coverings, \ie a map of sets $\Phi : \cU \to \cV$ such that $ U \subset \Phi(U)$, $\forall \, U \in \cU$. 
If  $\Phi: \cU \to \cV$ is a refinement of formal coverings of $X$, the inclusions $U \hookrightarrow \Phi(U) = V$ induce   morphisms of canonical reductions $\wtilde{U} \to \wtilde{V}$ which patch together to a morphism of $\kt$-curves  
$\wtilde{\varphi}_{\Phi}: \wtilde{X}_{\cU} \to \wtilde{X}_{\cV}$. We obtain a natural functor $\cU \mapsto  \wtilde{X}_{\cU}$, $\Phi \mapsto \wtilde{\varphi}_{\Phi}$,  from ${\cC}ov(X)$  to the category ${\cS}ch_{\kt}$ of $\kt$-schemes. 
\par
For any formal covering $\cU$ of $X$ there is a canonical  {\it specialization  map} $\sp_{\cU} : X  \to  \wtilde{X}_{\cU}$,  obtained by patching together the canonical reduction maps for all affinoids in the covering (their compatibility being precisely the definition of a {\it formal} affinoid covering). Specialization maps also have a  functorial behaviour: if $\Phi : \cU \to \cV$ is a refinement, 
$\sp_{\cV} = \wtilde{\varphi}_{\Phi} \circ \sp_{\cU}$.
\par
Two formal coverings $\cU$ and $\cV$ of $X$ are {\it equivalent} if  $\cU \cup \cV$ is a formal covering of $X$. Equivalence of formal coverings is an equivalence relation: we then write  $\cU \sim \cV$, and denote by $[\cU]$ the equivalence class of $\cU$. 
\begin{defn} A refinement of equivalent coverings $\Phi: \cU \to \cV$ is called a
\emph{quasi-isomorphism}.
\end{defn}
It is easy to check that the class of quasi-isomorphisms in ${\cC}ov(X)$ admits calculus of right fractions. The corresponding localized category  is denoted ${\cC}ov(X)/\sim$ and is called \emph{the category of formal affinoid coverings of $X$ up to equivalence}. Its objects are equivalence classes $[\cU]$ of  formal coverings of $X$, and, for any pair $\cU$, $\cV$ of  formal coverings of $X$, 
\beq
{\rm Hom}_{{\cC}ov(X)/\sim}([\cU],[\cV]) = \lim_{\cU^\p \to \cU} {\rm Hom}_{{\cC}ov(X)}(\cU^\p,\cV) \; ,
\eeq
where $\cU^\p \to \cU$ is a quasi-isomorphism of  formal coverings. 
If $\Phi$ is a quasi-isomorphism, 
$\wtilde{\varphi}_{\Phi}$ canonically identifies $\wtilde{X}_{\cU}$ to $\wtilde{X}_{\cV}$.
In particular, if $\cU$ and $\cV$ are any pair of equivalent formal coverings  of $X$, the inclusions $\Phi:\cU \hookrightarrow \cU \cup \cV$ and  $\Psi:\cV \hookrightarrow \cU \cup \cV$ are quasi-isomorphisms and 
$\wtilde{\varphi}_{\Phi}$ and $\wtilde{\varphi}_{\Psi}$ are isomorphisms which canonically identify   $\sp_{\cU}$ with $\sp_{\cV}$. In other words, the natural  functor $\cU \mapsto  \wtilde{X}_{\cU}$ factors through the \emph{reduction functor} 
$ {\cC}ov(X)/ \sim \, \to {\cS}ch_{\kt}$.

\par Let us now assume that the formal covering $\cU$ is distinguished. One then defines a sheaf of topological rings, flat and topologically of finite type \cite[\S 6.4.1 Cor. 5]{BGR}, hence topologically of finite presentation \cite[\S 2.3 Cor. 5]{bosch} over $\kc$, $\cO_{\cU}$ on $\wtilde{X}_{\cU}$ by setting 
 $\cO_{\cU}(V) = \l((\sp_{\ast}\cO_{X_G})(V)\r)^{\circ}$, for any open subset $V$ of $\wtilde{X}_{\cU}$. 
The topologically ringed space $\fX_{\cU} = (\wtilde{X}_{\cU}, \cO_{\cU})$ is a normal formal model of $X$  whose specialization map of $G$-ringed spaces $\sp_{\fX_{\cU}} : X_G \to \fX_{\cU}$ set-theoretically identifies with $\sp_{\cU}$. 
It is easy to check that  if  $\Phi: \cU \to \cV$ is a refinement of formal coverings of $X$, there is a natural morphism of sheaves of topological rings $ \varphi_{\Phi}^{\flat}: \cO_{\cV} \to (\wtilde{\varphi}_{\Phi})_{\ast}\cO_{\cU}$ lifting $\wtilde{\varphi}_{\Phi}$ to a  morphism of formal schemes $\varphi_{\Phi} = (\wtilde{\varphi}_{\Phi} , \varphi_{\Phi}^{\flat}): \fX_{\cU} \to \fX_{\cV}$ whose generic fiber is  the identity of the $k$-analytic curve $X$. Moreover, if $\Phi$ is a quasi-isomorphism of distinguished coverings, $\varphi_{\Phi}: \fX_{\cU} \to \fX_{\cV} =  \fX_{\cU}$ is the identity of the formal scheme $ \fX_{\cU}$.
\par
\begin{defn} We denote by ${\cC}ov^{\dis}(X)$ the full subcategory of ${\cC}ov(X)$ whose objects are distinguished formal coverings of $X$ . 
\end{defn}
The class of quasi-isomorphisms in ${\cC}ov^{\dis}(X)$ admits calculus of right fractions. The corresponding localized category  ${\cC}ov^{\dis}(X)/\sim$ is a full subcategory of ${\cC}ov^{\dis}(X)/\sim$, and is called \emph{the category of distinguished formal coverings of $X$ up to equivalence}. Its objects are equivalence classes $[\cU]$ of distinguished formal coverings of $X$. 
The following result will be detailed in \cite{FormalCurves}. 
\begin{thm}
The
construction $\cU \mapsto \fX_{\cU}$, $\Phi \mapsto \varphi_{\Phi}$ extends to a functor  ${\cC}ov^{\dis}(X) \to {\cF\cA}(X)$, which induces an equivalence of categories ${\cC}ov^{\dis}(X)/\sim \,  \xrightarrow{\ \sim\ } {\cF\cA}(X)$. 
\end{thm}
As explained in  \cite[\S 2.8, Step (b) in the proof of Thm. 3]{bosch}, for two admissible formal models $\fX$ and $\fY$ of $X$, there is at most one morphism $\varphi: \fX \to \fY$ inducing the identity on generic fibers. Therefore, the category ${\cF\cA}(X)$ is really a partially ordered directed set. 
It then follows that,  for any pair $\cU$, $\cV$ of distinguished formal coverings of $X$, 
\beq
{\rm Hom}_{{\cC}ov^{\dis}(X)/\sim}([\cU],[\cV]) =  \left\{ \begin{array}{ll}
         \{ \geq \} & \mbox{if $\exists \; \Phi : \cU^\p \to \cV^\p\, , \; \cU^\p \sim \cU \, , \; \cV^\p \sim \cV$};\\
        \; \; \emptyset \; & \mbox{otherwise}.\end{array} \right.  \; ,
\eeq
where $\Phi : \cU^\p \to \cV^\p$ is a refinement of  formal coverings. 

\begin{defn} Let $X$ be as before, and let  $\cU$ be a distinguished  formal strictly affinoid covering of $X$. We say that 
$\cU$ is a (resp. \emph{strictly}) \emph{semistable affinoid covering} of $X$, if $\wtilde{X}_{\cU}$ is a (strictly) semistable curve over $\kt$. 
 It is clear that these notions are compatible with  equivalence  of formal coverings. We define $\cC\cS(X)$ (resp. $\cC\cS^{\st}(X)$) as  the full subcategory of ${\cC}ov^{\dis}(X)/\sim$ whose objects are  equivalence classes of (resp. strictly) semistable affinoid coverings of $X$. 
\end{defn}
\begin{rmk}
A semistable covering $\cU$ of $X$ is distinguished by assumption and  the reduction $\wtilde{X}_{\cU}$ is geometrically reduced.  We deduce from \cite[Prop. 1.2]{BL0}  that for any non-archimedean extension field $K$ of $k$ and for $\cU 
\what{\otimes} K = \{U\what{\otimes} K\}_{U \in \cU}$, one has 
 $$\wtilde{X\what{\otimes} K}_{\cU\what{\otimes} K} = \wtilde{X}_{\cU} \otimes \wtilde{K} \; .$$
In particular,  $\cU 
\what{\otimes} K $ is a semistable covering of $X\what{\otimes} K$. 
\end{rmk}

 \par
It follows from Theorem \ref{evviva} that if the covering $\cU$ is (strictly) semistable, then $\fX_{\cU}$ is a (strictly) semistable model of $X$. 
\begin{cor} The construction $\cU \mapsto \fX_{\cU}$  induces an equivalence of the category $\cC\cS(X)$ (resp. $\cC\cS^{\st}(X)$)  of (resp. strictly)  semistable coverings  up to equivalence, and the category 
${\cF}{\cS}(X)$ (resp. ${\cF}{\cS}^{\st}(X)$). 
It has the property that ${\rm Hom}(\fX_{\cU}, \fX_{\cV})$ is nonempty if and only if there are formal coverings $\cU^\p$ and $\cV^\p$ respectively equivalent to $\cU$ and $\cV$, and a refinement 
 $\Phi : \cU^\p \to \cV^\p$. In that case ${\rm Hom}(\fX_{\cU}, \fX_{\cV}) = \{\varphi_{\Phi}\}$, via the canonical identifictions of $\fX_{\cU^\p}$  and $\fX_{\cV^\p}$ with $\fX_{\cU}$  and $\fX_{\cV}$, respectively. 
\end{cor}

 \par
 \subsection{Disks and annuli} \label{annuli} 
For $a \in \A^1(k)$ and $r \in \R_{>0}$ 
we denote by 
 $$D(a,r^-) = \{x \in \A^1 :  |T(x) - a| < r \} \subset \A^1  
$$ 
(resp. 
$$D(a,r^+) = \{x \in \A^1 : |T(x) - a|  \leq r\} \subset \A^1 \text{) ,}$$
 the \emph{ standard open} (resp. \emph{closed}) \emph{$k$-disk} of radius $r>0$ centered at $a$. 
The maximal point point  $t_{a,r}$ of $D(a,r^+)$ is defined by the multiplicative norm 
$$ |\sum_{i=0}^m a_i (T-a)^i (t_{a,r})| = \max_{i = 0, \dots,m} |a_i|r^i $$ on $k[T]$. 
We also denote by 
$$B(a;r_1,r_2) = \{x \in \A^1 : r_1  < |T(x) - a| < r_2   \} \subset \A^1 $$ 
(resp. 
$$B[a;r_1,r_2]  = \{x \in \A^1 :  r_1  \leq  |T(x) - a| \leq  r_2   \} \subset \A^1\text{) ,}$$
the  \emph{standard open} (resp. \emph{closed}) \emph{$k$-annulus} of radii $0 < r_1 < r_2$
(resp. $0 < r_1 \leq r_2$), centered  at $a \in \A^1(k)$.  We write $B(r_1,r_2)$ (resp. $B[r_1,r_2]$) for the standard annulus  $B(0;r_1,r_2)$ (resp. $B[0;r_1,r_2]$).  The ratio $r(V) = r_1/r_2 \in (0,1)$ is the \emph{height} of   $V = B(a;r_1,r_2)$ or $B[a;r_1,r_2]$. 
In general, an open or closed $k$-disk (resp. $k$-annulus) is a $k$-analytic space isomorphic to a standard open or closed $k$-disk (resp. $k$-annulus). It is $k$-rational if, moreover, the radii of the corresponding disk or annulus are in $|k|$. Two $k$-rational $k$-annuli, both open or both closed, are isomorphic iff they have the same height.

Moreover, if a $k$-analytic curve $V$  is isomorphic  via the coordinate $T$ to 
$D(0,r^{-})$ (resp. to $B(r(V),1)$),  for any automorphism $\varphi \in \Aut (V)$, and any $x \in V$, $|T(\varphi(x))|$ equals $|T(x)|$ (resp. either $|T(x)|$ or $ r(V)/ |T(x)|$: $\varphi$ is \emph{direct} in the former case and \emph{inverse} in the latter). 
An open  $k$-disk  (resp. $k$-annulus)  is a simply-connected quasi-polyhedron; it  has precisely one (resp. two) endpoint(s). 
\par   
A $k$-disk or $k$-annulus in  a $k$-analytic curve  $X$ will simply be called a \emph{disk} or an \emph{annulus} in $X$. 
\par 
\begin{defn} \label{basicannuli} The \emph{standard formal disk} over $\kc$ is the formal affine line $\what{\A}^1 := \what{\A^1_{\kc}} = \Spf   \kc\{T\}$. The \emph{standard formal annulus of height $r \in |k| \cap (0,1)$} over $\kc$ is the $\kc$-formal scheme $\fB[r,1] := \Spf \kc\{S,T\}/(ST -a)$, where $a \in \kc$, $|a| =r$. Notice that $\fB[r,1]$ is well-defined. 
A \emph{basic formal disk} (resp. a  \emph{basic formal annulus}) over $\kc$ is an affine connected $\kc$-formal scheme $\fB = \Spf A$ equipped with a dominant \'etale morphism $T$ to the standard formal disk (resp. to a standard formal annulus) over $\kc$, such that the inverse image of the point $0 \in \A^{1}(\kt)$ (resp. $(S,T) \in \Spec \kt[S,T]/(ST)$) consists of a single point, with residue field $\kt$.
A  \emph{basic affinoid disk}  (resp. \emph{basic affinoid annulus}) over $k$ is the generic fiber of a basic formal disk (resp. annulus) over $\kc$. We write $(\fB,T)$ 
for a formal disk (resp. annulus) and $(\cB,T)$ for the corresponding affinoid disk (resp. annulus). 
\end{defn}
Notice that if $(\fB,T)$ is a basic formal disk (resp. annulus) as in the definition, the generic fiber of $T$ induces an isomorphism of $k$-analytic spaces  $T^{-1}(D(0,1^{-})) \xrightarrow{\ \sim\ } D(0,1^{-})$
 (resp. $T^{-1}(B(|a|,1)) \xrightarrow{\ \sim\ } B(|a|,1)$)

\par \noindent

We will make essential use of the following  immediate consequence of definition \ref{STRss}.
\begin{lemma} \label{cover}
Every semistable $\kc$-formal scheme has a finite covering by \'etale neighborhoods which are basic formal annuli or disks over $\kc$. 
\end{lemma}

In our discussion below, it will be more convenient to use a different normalization for the coordinate function $T$ on a formal annulus. Namely, 
for $r_{1},r_{2}\in |k| \cap (0,1]$, we define, 
\beq \label{genbasannuli}
A[r_{1},r_{2}] = \kc\{S,T,U\}/(a_{2}S - T, TU -a_{1}) \; \; , \;\;  \fB[r_{1},r_{2}] = \Spf  A[r_{1},r_{2}] \; ,
\eeq
where $a_{i} \in \kc$, $|a_{i}| =   r_{i}$, $i=1,2$. So, $\fB[r_{1},r_{2}]$ is isomorphic to the formal annulus $\fB[r_{1}/r_{2},1]$, via the coordinate $T/a_{2}$.  Let $\fB$ be a $\kc$-formal scheme and $T:\fB \to \fB[r_{1},r_{2}]$ be an \'etale morphism, such that the composite 
$\fX  \xrightarrow{\ T \ }  \fB[r_{1},r_{2}]  \xrightarrow{\ T/a_{2}\ }  \fB[r_{1}/r_{2},1]$ is a basic formal annulus of height $r_{1}/r_{2}$ over $\kc$. We still say that $(\fB,T)$ (resp. its generic fiber $(\cB,T)$) is a \emph{basic formal (resp. affinoid) annulus of radii $r_{1}$, $r_{2}$ over $\kc$.} 

\subsection{Semistable partitions}
\label{formalpartitions} 

\par 
We introduce here a very simple notion, equivalent to the ones discussed in the previous subsections of semistable formal model  and of semistable affinoid covering of $X$. It is the notion of  a ``semistable partition" of $X$. It has the virtue of making completely obvious two well-known statements for which there is no canonical reference in the literature. Namely, the existence of a minimum semistable model of $X$ (unless $X$ is the analytification of either a rational or a Tate curve), and the existence of a common minimum semistable refinement of two semistable formal coverings. Again, we need such explicit descriptions, in view of application to differential systems on $X$. \par
We refer to \cite{FormalCurves} for full proofs in a more general setting.
\begin{defn} Let $X$ be a connected compact rig-smooth strictly $k$-analytic curve over an algebraically closed  non-trivially valued non-archimedean field $k$. A partition of $X$ into a disjoint union of a finite set of open annuli $\cB = \{ B_1, \dots, B_N\}$ and a finite set $\cC = \{C_1, \dots,C_r\}$  of connected distinguished strictly affinoid domains with good  canonical reduction will be called a  \emph{semistable partition} of $X$, and will be denoted by $\sP(\cB,\cC)$. The semistable partition $\sP(\cB,\cC)$ is a \emph{strictly semistable partition} if every annulus $B \in \cB$ has two distinct boundary points in $X \setminus B$. 
If $X$ is the analytification of a smooth projective curve with good reduction, we also regard $\sP(\emptyset,\{X\})$ as a  (strictly) semistable partition. 
For two semistable partitions $\sP = \sP(\cB,\cC)$ and $\sP^\p = \sP(\cB^\p,\cC^\p)$ of $X$, we say that 
$\sP^\p \geq \sP$ if for every $B \in \cB$ (resp. $C \in \cC$), the elements of  $\cB^\p$ and $\cC^\p$ contained in $B$ (resp. $C$) give a partition of $B$ (resp. $C$). We denote by  $\cP\cS(X)$ (resp. $\cP\cS^{\st}(X)$) the category, in fact a partially ordered set,  of semistable (resp.  strictly semistable) partitions of $X$. 
\end{defn} 
\begin{rmk} The condition that the affinoids of $\cC$ be distinguished is equivalent to the requirement that the boundary points of the annuli of $\cB$ in $X$ be points of type (2) \cite[6.4.3]{BGR}.
\end{rmk}
\begin{rmk} If $\sP(\cB,\cC)$ is a strictly semistable partition of $X$,  and $B \in \cB$, the two boundary points of $B$  in $X \setminus B$ are the maximal points of two distinct affinoids $C_{B,1}$ and $C_{B,2} \in \cC$.
\end{rmk}
 Let $\fX$ be   a semistable (resp. a strictly semistable) formal model of $X$ and let $\sp: X \to \fX_s$ be the corresponding specialization  map to a curve over $\kt$. Let ${\rm Sing}(\fX_s)$ be the singular locus of $\fX_s$, and  let $\fX_s^{\sm} = \fX_s \setminus {\rm Sing}(\fX_s)$ be the smooth part of $\fX_s$. Then ${\rm Sing}(\fX_s) = \{z_1,\dots,z_N\}$ where each $z_i$ is a closed point of $\fX_s$, and $B_i := \sp^{-1}(z_i)$ is an open annulus in $X$. Similarly, let ${\fc_1},\dots,{\fc_r} $ be the connected (hence irreducible) components  of the smooth locus $\fX_s^{\sm}$, and let  
 $\wtilde{\eta}_{\fc_j}$ be the generic point of ${\fc}_j$. Then either $\fX$ is smooth  in which case $r=1$ and $\fc_{1} = \fX_{s}$, 
 or  $\sp^{-1}(\fc_j)$ is a connected affinoid domain $C_j$ in $X$ with good canonical reduction, whose Shilov boundary consists of the single point $\eta_{C_j}$ of type (2), unique inverse image of $\wtilde{\eta}_{\fc_j}$.
Let $\cB_{\fX} = \{B_1,\dots,B_N\}$ and $\cC_{\fX} = \{C_1,\dots,C_r\}$.  We conclude that $\sP_{\fX} = \sP(\cB_{\fX},\cC_{\fX})$ is a  semistable (resp. strictly semistable) partition of $X$. If $\varphi: \fY \to \fX$ is a morphism of semistable (resp.  strictly semistable) formal models of $X$, then $\sP_{\fY} \geq \sP_{\fX}$:
we have defined a functor $\fX \mapsto \sP_{\fX}$ from the category of   semistable (resp. strictly semistable) formal models of $X$ to the category of  semistable (resp. strictly semistable)  partitions of $X$. 
\par
\medskip
 
\begin{thm}
The functor $\fX \mapsto \sP_{\fX}$ induces an equivalence of categories between $\cF\cS(X)$      (resp. $\cF\cS^{\st}(X)$) and 
 $\cP\cS(X)$  (resp. $\cP\cS^{\st}(X)$).
\end{thm}
\begin{proof} 
In order to construct a quasi-inverse of the functor $\fX \mapsto \sP_{\fX}$, we first associate to a semistable (resp. strictly semistable)  partition $\sP = \sP(\cB,\cC)$ of $X$,  a (resp. strictly) semistable covering $\cU_{\sP}$ of $X$, as follows.   For any $B \in \cB$, let $\{ C,C^\p \} \subset \cC$ be the set of $C \in \cC$  such that the closure $\ol{B}$ of $B$ in $X$ contains the maximal point $\eta_C$ of $C$. (We do not exclude that $C = C^{\p}$). Notice that $B \cup C \cup C^\p = \ol{B} \cup C \cup C^\p$ is a compact connected analytic domain in $X$, hence it is either an affinoid or a projective curve. In both cases, cutting off a finite number of  open disks contained in either $C$ or $C^{\p}$, we get an affinoid domain $U_{B}$ such that 
$B \subset U_B \subset B \cup C \cup C^\p$.
\begin{lemma} \label{GLUE} Let $B,C,C^\p$ be as before and let $U_{B}$ be any affinoid domain of $X$ such that $B \subset U_B \subset B \cup C \cup C^\p$ and such that $\l( B \cup C \cup C^\p \r) \setminus U_{B}$ is the union of a finite number of \emph{maximal} open disks contained in either $C$ or $C^{\p}$. Then, 
the canonical reduction of $U_B$ is a reduced curve which is either
irreducible  with a single ordinary double point, if $C=C^\p$, or  
the union of two smooth irreducible curves crossing normally, if $C \not= C^\p$.
\end{lemma} 
\begin{proof} 
We apply Prop. \ref{projemb} to realize  $U_{B}$  as an affinoid domain in $Y = \sY_{\eta}^{\an} = \fY_{\eta}$, where $\sY$ is a strictly semistable projective $\kc$-scheme of formal completion $\fY$, and  $U_{B}$ is of the form $\sp_{\fY}^{-1}(\cZ)$, for $\cZ$ an affine open subset of $\fY_{s}$.  Notice that this \emph{does not} imply that the canonical reduction map of $U_{B}$ be just the restriction of  $\sp_{\fY}$ to $U_{B} \to \cZ$. All we know is that there is a morphism  compatible with  specialization maps  $\cZ \to \widetilde{U}_B$. We now construct a semistable partition $\sP_{B} = (\fB_{B},\fC_{B}) \leq \sP_{\fY}$ of $Y$ as follows. On $\sp_{\fY}^{-1}(\fY_{s} \setminus \cZ)$, we let our partition be defined as $\sP_{\fY^{\p}}$, where  $\fY^{\p}$ is the open formal subscheme of $\fY$ supported on $\fY^{\p}_{s} := \fY_{s} \setminus \cZ$. On $U_{B} \subset Y$, the partition $\sP_{B}$ is given by the single open annulus $B$ and the (one or two) affinoids $\{C\cap U_{B} ,C^{\p}\cap U_{B} \}$ (the $\cap$ being taken in the original curve $X$).
Let $B_1, \dots, B_m$ (resp. $B^\p_1, \dots, B^\p_{m^\p}$) be the elements of $\fB_{B}$, different from $B$, whose closure contains  $\eta_C$ (resp. $\eta_{C^\p}$). As in the proof of \cite[3.6.1]{BerkovichEtale}, we cut thin subannuli, for small positive  $\veps$, $B^{(\veps)}_1, \dots, B^{(\veps)}_m$ (resp. $B^{\p (\veps)}_1, \dots, 
B^{\p (\veps)}_{m^\p}$), $B^{(\veps)}_i \subset B_i$, $i = 1, \dots, m$ (resp.  $B_{j}^{\p (\veps)} \subset B^\p_j$, $j = 1, \dots, m^\p$) having $\eta_C$ (resp. $\eta_{C^\p}$) as a cluster point. We assume that, for every $i$, $B_i$ is identified with $B(r_{i},1)$ in such a way that $t_{0,r} \to \eta_{C}$ as $r \to 1$, and $B^{(\veps)}_i$ is identified with $B(1 -\veps,1) \subset B(r_{i},1)$, and similarly for $B_{j}^{\p (\veps)} \subset B^\p_j$, for every $j$.
We now 
glue open disks to $C$ (resp. $C^\p$), via the natural embeddings $B^{(\veps)}_i = B(1 -\veps,1) \subset D(0,1^-)$  (resp.
$B_j^{\p (\veps)} = B(1 -\veps,1) \subset D(0,1^-)$) as $r \to 1$, to obtain ${\bf C}$ (resp. ${\bf C}^\p$).  The resulting $k$-analytic  curve $T = U_B \cup {\bf C} \cup {\bf C}^\p$  is smooth and proper, and is therefore algebraizable $T = \sT^{\an}$, where $\sT$ is a smooth projective curve over $k$. Moreover, the formal covering $\{ U_B , {\bf C} , {\bf C}^\p \}$ determines a formal model $\fT$ of $T$ whose special fiber is the union of the canonical reductions $\{ \widetilde{U}_B ,\widetilde{\bf C} , \widetilde{\bf C}^\p \}$ of the three affinoids in the covering. As for $\widetilde{\bf C}$ and $\widetilde{\bf C}^\p$, they are disjoint affine open smooth $\kt$-subvarieties of $\fT_{s}$. 
Moreover, $X \setminus \sp_{\fT}^{-1}(\widetilde{\bf C} \cup \widetilde{\bf C}^\p) = B$  implies that 
$\fT_{s} \setminus (\widetilde{\bf C} \cup \widetilde{\bf C}^\p)$ consists of a single $\kt$-rational point. 
  By \cite[2.3]{BL0}, that point is an ordinary double point of $\fT_{s}$. This concludes the proof. 
\end{proof}
We now define the affinoid  covering $\cU_{\sP}$ of $X$, as the union of $\cC$ and of the set of all  affinoinds $U_B$ constructed as above, for all $B \in \cB$. It is clear that, for any choice of $B,B^\p \in \cB$ and of affinoids 
$U_B$ and $U^\p_{B^\p}$, as above, the intersection $U_B \cap U^\p_{B^\p}$ is a formal affinoid in both $U_B$ and $U^\p_{B^\p}$. In other words, 
we have shown that   $\cU_{\sP}$ is a formal semistable covering of $X$. 
We  set $\fX_{\sP}$ for the formal scheme  $\fX_{\cU_{\sP}}$. 
If $\sP^\p \geq \sP$ there  a canonical refinement $\Phi: \cU_{\sP^\p}  \to  \cU_{\sP}$, and  therefore  a morphism $\fX_{\sP^\p} \to \fX_{\sP}$.
\par The composite functor $\fX \mapsto \sP := \sP_{\fX} \mapsto \cU_{\sP} \mapsto\fX_{\cU_{\sP}}$ is easily seen to be the identity  of $\cF\cS(X)$ (resp. of $\cF\cS^{\st}(X)$). Similarly, 
$\sP \mapsto \cU_{\sP} \mapsto \fX := \fX_{\cU_{\sP}} \mapsto \sP_{\fX}$ is the identity of   $\cP\cS(X)$  (resp.  $\cP\cS^{\st}(X)$). 
\end{proof}
\par
The proof  of the following results will be detailed in \cite{FormalCurves}.

\begin{lemma}
\label{uniondiskannuliCLOSED} Let $X$ be a compact rig-smooth strictly $k$-analytic curve. Then
\begin{enumerate}
\item The  union of two closed disks  with nonempty intersection in $X$ is either 
one of the two disks, or else it is all of $X$ and then $X$ is the analytification of a projective smooth curve of genus 0. In the former case the intersection of the two disks is one of them; in the latter, it is a closed annulus. 
\item
A non-empty  intersection of a closed  disk  and a closed annulus in $X$  is either the disk or the  annulus  or a strictly smaller closed annulus. Their union is, respectively, the annulus, the disk or a strictly bigger closed disk. 
\item
A non-empty intersection  of two closed annuli in $X$ consists of either one or two disjoint  closed annuli. Their union is a closed annulus in the former case and a Tate curve in the latter.
\end{enumerate}
\end{lemma}

\begin{cor}\label{uniondiskannuli} The lemma holds by replacing ``closed'' by ``open'' everywhere in the statement.
\end{cor}
\begin{cor}\label{anninterCLOSED} Let $C$ and $C^{\p}$ be connected strictly affinoid domains of $X$, with good canonical reduction.
Then
\begin{enumerate}
\item  Assume $X$ is not rational. Then $C \cap C^{\p}$ equals either $\emptyset$, or $C$ or $C^{\p}$, or else it is a formal affinoid domain in both $C$ and $C^{\p}$.
\item If the union $U = C \cup C^{\p}$ is an affinoid domain in $X$, properly containing both $C$ and $C^{\p}$, then $U$ has good canonical reduction, and its specialization map is induced by those of $C$ and $C^{\p}$.
\item The intersection of a closed annulus $B$ of $X$ with $C$, equals either $\emptyset$, or $B$, or $C$, or the complement in $C$ of a finite number of open disks contained in distinct residue classes, and coinciding with the full residue class except for at most one of them. 
\end{enumerate} 
\end{cor}

\begin{cor}\label{anninter}  \hfill
\begin{enumerate}
\item The intersection of an open disk $D$ of $X$ with a connected strict affinoid domain $C$ of $X$, with good canonical reduction, equals either $\emptyset$  or $D$ or $C$, or the complement of a closed disk contained in a single residue class of $C$. 
\item The intersection of an open annulus $B$ of $X$ with a connected strict  affinoid domain $C$ of $X$, with good canonical reduction, equals either $\emptyset$  or $B$ or $C$.
\end{enumerate} 
\end{cor}

Straightforward consequences of the previous lemmas are the following propositions \cite{FormalCurves}.
\begin{prop}  \label{directed} The partially ordered set of semistable partitions of $X$ is directed: if $\sP = \sP(\cB,\cC)$ and $\sP^\p = \sP(\cB^\p,\cC^\p)$ are semistable partitions of $X$, there exists a minimum semistable partition $\sP^{\p \p}$ of $X$ with $\sP^{\p \p} \geq \sP$ and 
$\sP^{\p \p} \geq \sP^\p$.
\end{prop}
\begin{prop}  \label{minimumpart} Assume $X$ is neither the analytification of  a rational nor of a Tate curve. Then the partially ordered set of semistable partitions of $X$ has  a minimum $\sP_0 = \sP(\cB_0,\cC_0)$.
\end{prop}
\par

\subsection{Subpolygons of an analytic curve} \label{subsect:graphs}
There is a fourth equivalent viewpoint we need to master for dealing with differential equations on $X$. This is the notion of a ``subpolygon of $X$". Actually, this notion has been used implicitly by Dwork, Robba, Christol,... since the infancy of the theory of $p$-adic differential equations. The framework of Berkovich analytic spaces gives substance to what has been for a long time a somewhat artificial tool to describe the ``generic" behaviour of radii of convergence of power series solutions to differential equations. The global description of radii of convergence on the full Berkovich analytic space is the main novelty of this paper. 
This section is entirely due to Berkovich, but we are unable to give precise references for most of the material we have collected here. Full details are given in \cite{FormalCurves}.
\par 
We recall that the  non-archimedean field $k$  is supposed to be non-trivially valued and algebraically closed.
\begin{defn} Let $Y$ be any rig-smooth strictly  $k$-analytic curve. The \emph{analytic skeleton} $S(Y)$ of $Y$ is the minimum closed subset $T$ of $Y$, such that $Y \setminus T$ is a union of open disks. A \emph{vertex} of $Y$ is a point contained in no  open annulus.  We denote by $\cV(Y) \subset Y$ the set of vertices of $Y$.
\end{defn}
Notice that a vertex of $Y$ necessarily belongs to $S(Y)$ and that, if $Y$ is affinoid, any point in the Shilov boundary of $Y$ is a vertex. 
Points of $S(Y)$ are of type (2) or (3). We denote by 
$\H(Y) \subset Y$ the subset consisting of points of type (2) or (3), and call it the \emph{Berkovich hyperbolic subspace} of the curve $Y$. For any algebraically closed non-archimedean extension field $k^\p/k$,  the canonical projection $Y \wt k^\p  \to Y$  induces a homeomorphism $S(Y \wt k^\p)  \xrightarrow{\ \sim\ } S(Y)$ (resp. $\cV(Y \wt k^\p)  \xrightarrow{\ \sim\ } \cV(Y)$).
\begin{defn}
An \emph{open segment}   $E$ in $Y$ is  the skeleton of an open annulus  in $Y$. Its (\emph{multiplicative}) \emph{length} $\ell(E)$ is the inverse of the height of the corresponding annulus. An \emph{open half-line}   in $Y$ is  the skeleton of an  open disk  punctured at a $k$-rational point  in $Y$. We say that it has  \emph{infinite length}. 
\end{defn} 
\begin{exa} For the standard $k$-annulus $B(r_1,r_2)$, the skeleton is the subset $\{ t_{0,\rho} | r_1 < \rho < r_2\}$, homeomorphic to $(r_1,r_2)$ via $\rho \mapsto t_{0,\rho}$. The elements of $\Aut(B(r_1,r_2))$ act on $S(B(r_1,r_2))$ via the identity, if they are direct, and via  $\rho \mapsto r_1r_2/\rho$ otherwise. Similarly for  punctured open disks.
\end{exa}
For an open segment (resp. an open half-line) $E \subset Y$, the open annulus  (resp. the open punctured disk)  $A \subset Y$ such that $S(A) = E$ is uniquely determined as the union of open disks $D$ in $Y \setminus E$ such that the closure of $D$ in $Y$ intersects $E$. We denote it by $A_E$.    An open segment (resp. half-line) $E$ is  \emph{$k$-rational}  if $A_E$ is a $k$-rational  open annulus (resp. disk) in $Y$, \ie if the ends of $E$ in  $Y$ are points of type (2) (resp. one point of type (2) and one $k$-rational point). 
\par
Let $E$ be an open $k$-rational segment of $Y$ of length $1/r(E)$. An isomorphism  $T: A_{E} \xrightarrow{\ \sim\ } B(r(E),1)$ induces a homeomorphism  $\rho = |T(-)|: E \xrightarrow{\ \sim\ }  (r(E),1)$, which is canonical up to the inversion $\rho \mapsto r(E)/ \rho$.  
\par
Notice that the previous construction associates canonically, on any simply connected  sub-quasi-polyhedron  $U$ of $\H(X)$, a (multiplicative) \emph{length} $\ell_U([x,y]) \in (1,\infty)$ where the  \emph{link} $[x,y]$, with $x, y \in U$ is the unique path in $U$ joining $x$ to $y$.  In fact, $[x,y]$ is a disjoint union of a finite set $E_1, \dots,E_N$ of open segments of $X$ contained in $U$, and $\ds \ell([x,y]) = \prod_{i=1}^N \ell(E_i)$. The function $(x,y) \mapsto \ell([x,y])$ is clearly continuous on $U \times U$. 
\par

\begin{defn} A \emph{subpolygon} of $X$ is a triple  ${\bf S} = (S,\cV, \cE)$, where $S$ is a compact connected subset of $\H(X)$ and $\cV$ is a finite subset of points of type $(2)$ of $S$  called the \emph{vertices} of $\bf S$.  We assume that the connected components of $S \setminus \cV$ are a finite set of open (necessarily $k$-rational) segments $E_1, \dots,E_N$ of $X$. We call them \emph{open edges} of $\bf S$ and write $\cE = \{E_1, \dots,E_N\}$. 
For two subpolygons ${\bf S} = (S,\cV, \cE)$  and  ${\bf S}^\p = (S^\p,\cV^\p,\cE^\p)$ of $X$, we write 
${\bf S} \leq {\bf S}^\p$, and say that ${\bf S}$ is a \emph{subpolygon of} ${\bf S}^\p$ if $S \subset S^\p$,       $\cV \subset \cV^\p$, and every open edge $E \in \cE$ is a union of vertices and open edges of ${\bf S}^\p$. A \emph{morphism} $\varphi: {\bf S}^\p  \to {\bf S}$ of subpolygons of $X$ is a continuous retraction $\varphi: S^\p \to S$ such that for any open edge $E^\p \in \cE^\p$, either $\varphi(E^\p)$ is a vertex of  $\bf S$ or $\varphi$ induces the inclusion of $E^{\p}$ in an open edge of  $\bf S$. 
A subpolygon ${\bf S} = (S,\cV, \cE)$ of $X$ is \emph{complete} if $X\setminus S$ is a union of open disks. 
We denote by $\cG\cP(X)$ the category of  subpolygons of $X$, and by $\cG\cP^c(X)$ the full subcategory consisting of complete subpolygons of $X$.
\end{defn}
Notice that, if a morphism $\varphi: {\bf S}^\p  \to {\bf S}$  in $\cG\cP(X)$ exists, $\bf S$ is a subpolygon of ${\bf S}^\p$, and the fibers of $\varphi$ over a vertex are simply connected trees. We also denote by $(\cG\cP(X), \leq)$ the category associated to the order relation $\leq$: a morphism ${\bf S} \leq {\bf S}^\p$ in this category is called an \emph{inclusion} ${\bf S} \to {\bf S}^\p$ of subpolygons of $X$. 
Let  $\iota: {\bf S} \to {\bf S}^\p$ be an inclusion of subpolygons of $X$. We say that $\iota$ is a \emph{subdivision} (and that \emph{ ${\bf S}^\p$ is obtained by subdivision of 
${\bf S}$}) if $S = S^\p$.
A subdivision is visibly the composition of a finite number of subdivisions obtained by the addition of one vertex inside an open edge. We call  \emph{1-step subdivision} such a subdivision. 
In general, a subdivision ${\bf S} \leq {\bf S}^\p$ in which ${\bf S}^\p$ has $N$ vertices more than ${\bf S}$, \ie an \emph{$N$-step subdivision}, is a product of $N$ 1-step subdivisions.
An inclusion ${\bf S} \leq {\bf S}^\p$ is called \emph{exact}   if every vertex (resp. open edge) of ${\bf S}^\p$  contained in $S$ is a vertex (resp. open edge) of ${\bf S}$. An inclusion 
$\iota : {\bf S} \to {\bf S}^\p$ canonically decomposes into a composition $\iota = \veps \circ \sigma$, where $\veps : {\bf S}^s \to {\bf S}^\p$  is an exact inclusion and $\sigma :  {\bf S} \to {\bf S}^s$ is a subdivision.
\par
For a morphism $\varphi: {\bf S}^\p  \to {\bf S}$  in $\cG\cP(X)$, we say that it is a \emph{trivial retraction} if ${\bf S} \leq {\bf S}^\p$ is a subdivision, and a \emph{neat retraction} if ${\bf S} \leq {\bf S}^\p$ 
 is an exact inclusion. 
 Any morphism $\varphi: {\bf S}^\p  \to {\bf S}$  in $\cG\cP(X)$
  canonically decomposes as 
$\varphi^{\rm triv} \circ \varphi^{\rm neat}$, where $\varphi^{\rm neat}: {\bf S}^\p \to {\bf S}^s$ is neat and  $\varphi^{\rm triv}: {\bf S}^s \to {\bf S}$ is trivial. Notice that a neat retraction sends vertices to vertices. 
 A neat retraction $\varphi: {\bf S}^\p  \to {\bf S}$ is \emph{simple} if there are precisely  one open edge and  one vertex of ${\bf S}^\p$ which are not, respectively,  an open edge and  a vertex  of ${\bf S}$. 
Every neat retraction $\varphi: {\bf S}^\p  \to {\bf S}$  is a composition of a finite number $N$ of simple retractions. The number $N$ is the number of open edges of ${\bf S}^\p$ which are not open edges of $\bf S$.  Similarly, every trivial retraction is a product of a finite number  of trivial 1-step retractions, each corresponding to a 1-step subdivision. 
\par
 To any semistable partition $\sP = \sP(\cB=\{ B_1, \dots,B_N \},\cC= 
\{C_1,\dots,C_r \} )$ of $X$, we naturally associate a complete  subpolygon ${\bf S}(\sP)$, whose vertices  are the maximal points of the affinoids $C_1,\dots,C_r$ and whose open edges  are the skeleta of the open annuli $B_1, \dots,B_N$. If $\sP^\p \geq \sP$ in $\cP\cS(X)$, it is clear that 
${\bf S}(\sP^\p) \geq {\bf S}(\sP)$ as  subpolygons of $X$, and that there is a morphism 
${\bf S}(\sP^\p) \to {\bf S}(\sP)$  in $\cG\cP^c$. 
\par
The following results are proven in \cite{FormalCurves}.
\begin{thm} \label{skelfunctor}
The functor $\bf S$ establishes an equivalence of categories between $\cP\cS(X)$ and 
$\cG\cP^c(X)$. Moreover, $\cG\cP^c(X)$ is the opposite category to $(\cG\cP^c(X),\leq)$. We also denote by $\bf S$ the composite functor $\fX \mapsto \sP_{\fX} \mapsto {\bf S}(\sP_{\fX} ) =: {\bf S}(\fX)$, which induces an equivalence of categories between $\cF\cS(X)$ and $\cG\cP^c(X)$.
\end{thm}
Notice that \cite[1.7]{BerkContr} the generic point of a component $\fc_i$ of $\fX_s$ has a unique inverse image $\eta_{i}$ in $X$ under $\sp_{\fX}$, and the set of vertices $\cV(\fX)$ of ${\bf S}(\fX)$ is precisely the set of those inverse images. The set of open edges $\cE(\fX)$ is the set of analytic skeleta of the inverse images under $\sp_{\fX}$ of the singular points of $\fX_{s}$. 
For a morphism $\varphi: \fY \to \fX$ in $\cF\cS(X)$, we denote by $\tau_{\fX,\fY}: {\bf S}(\fY) \to {\bf S}(\fX)$ the corresponding morphism in $\cG\cP^c(X)$. Then, as a topological space,
\beq \label{limin}
X = \limpr_{\fY \geq \fZ}(S(\fY), \tau_{\fZ, \fY}) \; ,
\eeq
while the subset $\H(X) \subset X$ of points of type (2) or (3) is
\beq
\label{hyperb}  \H(X) = \bigcup S(\fY) \; ,
\eeq
where $\fY$ varies among  semistable models of $X$. In particular, we get, for any  semistable model $\fX$ of $X$, a retraction $\tau_{\fX}: X \to S(\fX)$, as the natural map 
\beq
\label{retr}
\tau_{\fX} :=  \limpr_{\fY \geq \fX}(S(\fY), \tau_{\fX, \fY}) \;, 
\eeq 
to $S(\fX)$. It has the property  that the inverse image under $\tau_{\fX}$ of any vertex or closed edge of ${\bf S}(\fX)$, is a strict  affinoid domain in $X$. It follows that the same is true for any closed $k$-rational segment contained in a closed edge of 
${\bf S}(\fX)$ (because it is a closed edge of ${\bf S}(\fX)$, for some $\fY \geq \fX$). 
If $\fX = \fX_{0}$ is the minimum semistable model of $X$, then ${\bf S}(\fX_{0}) = {\bf S}(X)$ and 
$\tau_{\fX_{0}} = \tau_X$. 
\begin{defn} \label{defcoord} For any  semistable model $\fX$ of $X$, and  any rational point $y \in X$ (necessarily, $y \in X \setminus S(\fX)$), we denote by $D_{\fX}(y,1^-)$ the maximal open disk neighborhood of $y$ contained in $X \setminus S(\fX)$. 
\end{defn}
\begin{defn} Assumptions as in definition \ref{defcoord}.
An isomorphism 
\beq D_{\fX}(y,1^-)  \xrightarrow{\ \sim\ }  D_k(0,1^-)
\eeq
 is called an \emph{$\fX$-normalized coordinate at $y$}.
\end{defn}
Notice that the map $\tau_{\fX}$ takes the disk  $D_{\fX}(y,1^-)$ to the unique boundary point of 
$D_{\fX}(y,1^-)$ in $X$,  not in $D_{\fX}(y,1^-)$,   $\tau_{\fX}(y) \in S(\fX)$. The fiber  $\tau_{\fX}^{{-1}}(\tau_{\fX}(y)) \setminus \{\tau_{\fX}(y)\}$ is the disjoint union of a family of disks having the same limit point $\tau_{\fX}(y) \in S(\fX)$. It follows from theorem \ref{Q-analss} that for any compact rig-smooth strictly $k$-analytic curve $X$, which is not rational, and any rational point $y \in X$ (necessarily, $y \in X \setminus S(X)$), there exists a maximal open disk neighborhood $D_{X}(y,1^{-})$ of $y$ in $X$. It generalizes the neighborhood $\cD_{y}(X)$ defined in the introduction for an affinoid $X$ in $\A^{1}$. 

 \subsection{Admissible blow-ups} \label{blow-ups}

  Let us recall that for any admissible $\kc$-formal scheme  $\fX$, any $a \in \kcc \setminus \{0\}$, and any open ideal of finite presentation $\fA \subset \cO_{\fX}$, one defines the blow-up of $\fX$ along $\fA$ as the morphism  $\varphi: \fY \to \fX$ of formal schemes, inductive limit  as $n \to \infty$, of  the blow-up $\varphi_n: \fY_n \to \fX_n$ of the scheme $\fX_n = (\fX, \cO_{\fX}/(a)^{n+1})$ along the ideal $\fA \otimes \cO_{\fX}/(a)^{n+1}$. 
Such a morphism  is called  an \emph{admissible blow-up}, and is independent of the choice of $a$. 
\par
We now show that a morphism $\varphi: \fX^{\p} \to \fX$  in $\cF\cS(X)$ is an admissible blow-up. We use  the category $\cP\cS(X)$ for our description. Let us  consider two strictly semistable partitions $\sP = \sP(\cB,\cC)$ and $\sP^\p = \sP(\cB^\p,\cC^\p)$ of $X$, with 
$\sP^\p \geq \sP$, and let $\varphi: \fX^{\p} := \fX_{\sP^{\p}} \to \fX := \fX_{\sP}$ be the corresponding morphism of  $\cF\cS(X)$.
Consider an affinoid with good canonical reduction $C^\p \in \cC^\p$. Then either $C^{\p}$ is contained in an open annulus $B \in \cB$ or it is contained in an affinoid $C \in \cC$. In the second case,  either $C^\p$ is contained in a single residue class $D$ of $C$, or  the maximal point $\eta_{C^{\p}}$ of $C^{\p}$ coincides with  the maximal point $\eta_{C}$ of $C$. In the latter case,   $C^{\p}$ is the complement in $C$ of a finite number of residue classes $D_{1},\dots,D_{r}$ of $C$. Notice that for each $i = 1, \dots,r$, $D_{i}$ must contain a finite number of disjoint affinoids with good canonical reduction belonging to $\cC^{\p}$.
The  partition $\sP^\p$ may be reconstructed from the family $\sF = \sF(\sP^\p,\sP)$ of the elements of $\cC^\p$ contained in either an open annulus $B \in \cB$, or in a maximal disk $D$ of some affinoid with good canonical reduction  $C \in \cC$. In fact, let $C^\p_{0} \in \cC^\p$ be an affinoid with good canonical reduction which is not in $\sF$. Then, $C^\p_{0}$ is obtained from one affinoid $C \in \cC$, by subtracting those residue classes which do contain some $C^{\p} \in \sF $. On the other hand, the annuli in $\cB^\p$ are simply the connected components of the complement in $X$ of the union of all affinoids in $\cC^\p$.
Now  if $B$ (resp. $D$) is an open $k$-rational annulus (resp disk),  the description of connected strictly affinoid domains with good canonical reduction in $B$ or $D$ is elementary. They are the complement of a finite number of maximal open disks in a closed strictly affinoid disk. We define $C_{\varphi}$ as the union of all affinoids $V \in \sF$ and $\fA_{\varphi}$ as the sheaf of ideals of  $\cO_{\fX}$ consisting of sections $f$ of $\cO_{\fX}$ whose pull-back under $\sp_{\fX}$ is $<1$ on  $C_{\varphi}$. It is clear (by explicit description of generators) that  $\fA_{\varphi}$ is an open ideal of finite presentation of $\cO_{\fX}$.
\par
It is easily seen  (\cf  \cite{FormalCurves} for more details) that 
\begin{thm} 
The morphism $\varphi$ is the admissible blow-up of $\fX$ along $\fA_{\varphi}$. 
\end{thm}
\begin{defn}
We will say that the morphism $\varphi: \fY \to \fX$  in $\cF\cS(X)$ is a \emph{trivial} (resp. \emph{simple}) \emph{1-step blow-up}, if the corresponding morphism ${\bf S}(\varphi): 
{\bf S}(\fY) \to {\bf S}(\fX)$ in  $\cG\cP^c(X)$ is a trivial (resp. simple) 1-step retraction.  
\end{defn}
\begin{cor} Any  morphism $\varphi: \fY \to \fX$  in $\cF\cS(X)$ is  a composition  
\beq
\fY =: \fY^{(N)} \to  \fY^{(N-1)}\dots \to   \dots \to \fY^{(1)} \to \fY^{(0)}:=  \fX^{(M)} \to \fX^{(M-1)}\dots   
\to \fX^{(1)} \to \fX^{(0)}:= \fX \;,
\eeq 
where $\fY^{(i+1)} \to \fY^{(i)}$  is the blow-up of a single closed disk contained in an open disk $D_{\fY^{(i)}}(x,1^{-})$, for some $x \in X \setminus S(\fY^{(i)})$ (resp. 
 $\fX^{(i+1)} \to \fX^{(i)}$  is the blow-up of a single closed annulus of height 1 contained in an open annulus $B_{i} \in \cB_{i}$ and intersecting its skeleton, where  $\sP_{i} = (\cB_{i},\cC_{i})$ is the semistable partition associated to $\fX_{i}$).
 \end{cor}
In particular, every morphism $\varphi:\fX^{\p} \to \fX$ in $\cF\cS(X)$ is an admissible blow-up, it is a product of 1-step trivial or simple blow-ups,   and each non-empty fiber of $\fX^{\p}_{s} \to \fX_{s}$ is a connected union of rational smooth projective curves whose graph is a tree.

\subsection{\'Etale morphisms of formal schemes} \label{etalechange}
For any morphism $\varphi : \fY \to \fX$ of  semistable $\kc$-formal schemes, with generic fiber  $\varphi_{\eta}: Y \to X$, the following diagram is commutative \cite[4.4.2]{BerkDwork}
\beq
\label{diagram1}
\begin{array}{ccc} S(\fY)&\mathop{\hbox to 80pt{\rightarrowfill}}
\limits^{\tau_{{\fX}} \circ \varphi_{\eta}}&S(\fX) \\
\tau_{\fY}\Bigg\uparrow&&\Bigg\uparrow \tau_{\fX}\\
Y&\mathop{\hbox to 80 pt{\rightarrowfill}}
\limits^{\varphi_{\eta}}&X
\end{array}
\eeq

If, moreover, $\varphi : \fY \to \fX$ is   \'etale,  then  $\varphi_{\eta}(S(\fY)) \subset 
S(\fX)$, and 
$\varphi_{\eta}^{-1}(S(\fX)) = 
S(\fY)$, 
so that the previous diagram becomes

\beq\label{diagram2}
\begin{array}{ccc} S(\fY)&\mathop{\hbox to 80pt{\rightarrowfill}}
\limits^{(\varphi_{\eta})_{|S(\fY)}}&S(\fX) \\
\tau_{\fY}\Bigg\uparrow&&\Bigg\uparrow \tau_{\fX}\\
Y&\mathop{\hbox to 80 pt{\rightarrowfill}}
\limits^{\varphi_{\eta}}&X
\end{array}
\eeq
and 
$\varphi_{\eta}$ induces isomorphisms $D_{\fY}(y,r^{-}) \xrightarrow{\ \sim\ } D_{\fX}(x,r^{-})$, for any $y \in Y \setminus S(\fY)$, $\varphi_{\eta}(y) =x$, $r \in (0,1]$ [4.3.2 \lc]. 
\par
Suppose the point $y \in Y$ is a vertex of ${\bf S}(\fY)$. This means that it is the unique inverse image under $\sp_{\fY}$ of the generic point of a component $\fc$ of $\fY_{s}$. But $\varphi$ induces an \'etale morphism from the smooth part $\fc^{\sm}$ to a component $\fc^{\p}$ of $\fX_{s}$ \cite[2.2 $(i)$]{BerkContr}. Therefore, $\varphi_{\eta}(\cV(\fY)) \subset \cV(\fX)$. 

\section{Continuity of real valued functions on rig-smooth $k$-analytic curves} \label{ContLemma}

\begin{lemma}\label{lemma:fieldext} Let $k$ be any non-archimedean field, $Y$ be any $k$-analytic space, $L$ a non-archimedean field over $k$ and $Y_L = Y \what\otimes_k L$ be the extension of $Y$ to $L$. Then the natural topology of $Y$ is the quotient topology of the natural topology of $Y_{L}$ via the projection map  $\psi_L = \psi_{Y,L/k}:Y_L \to Y$.
\end{lemma}
\begin{proof} We first prove that the map $\psi_L$ is closed. Let $C$ be a closed subset of $Y_L$. Let $y$ be a point of $Y \setminus \psi_L(C)$, and let $D_2$ be a compact neighborhood of $y$ in $Y$. Then $D_1 = \psi_L^{-1}(D_2)$ is a compact subset of $Y_L$. The intersection $C \cap D_1$ is then compact; its image $\psi_L(C \cap D_1)$ is then closed, so that $D_2 \setminus \psi_L(C \cap D_1)$ is a neighborhood of $y$ in $Y$ not intersecting $\psi_L(C)$. 
\par \noindent The conclusion follows from  \cite[2.4]{GenTop}.
\end{proof} 
We assume from now on in this section  that the  non-archimedean field $k$ is non-trivially valued.
\begin{defn}\label{etalefct}
Let  $\fX$ be a semistable $\kc$-formal scheme, $\Gamma$ a topological space, 
and let $X = \fX_{\eta}$ be the generic fiber of $\fX$. Let $f:X \to \Gamma$ be any function and, for any non-archimedean field extension $k^\p/k$ and  \'etale morphism of semistable $k^{\p \circ}$-formal schemes $\psi:\fY \to \fX \wt k^{\p \circ}$, let $f_{\psi}: \fY_{\eta} \to \Gamma$ be the composite 
$$ f_{\psi} =  f  \circ \psi_{\fX_{\eta},k^\p/k} \circ \psi_{\eta} \; .
$$
We say that $(f_{\psi})_{\psi}$ is the \emph{\'etale-local system of functions on $\fX$ with values in $\Gamma$} associated to the function $f$. We identify $f =  f_{{\rm id}_{\fX}}$ with the system $(f_{\psi})_{\psi}$.
\end{defn}
Notice that if $f : X \to \Gamma$ is continuous, every component $f_{\psi}$ of the \'etale-local system of functions associated with $f$ is also continuous.   We present some basic results in the opposite direction. The first lemma is a consequence of lemma \ref{lemma:fieldext} and  \cite[Lemma 5.11]{BerkContr}.
\begin{lemma}\label{etaleloc} Let $\fX$ be a semistable $\kc$-formal scheme and $f = (f_{\psi})_{\psi}$ be an \'etale-local system of functions on $\fX$ with values in the topological space $\Gamma$, associated to the function $f: X = \fX_{\eta} \to \Gamma$. 
Assume there is a non-archimedean  extension field $k^\p/k$ and an \'etale covering  $\{\psi_{\alpha} : \fY \to \fX \wt k^{\p \circ} \}_{\alpha} $  of $\fX \wt k^{\p \circ}$, such that $\forall \alpha$, $f_{\psi_{\alpha}}$  is continuous. 
Then $f$ is continuous.
\end{lemma}
\begin{notation} \label{gamma}Let $[a,b]$ be an interval in $\R$, and let $\Gamma$ be a topological subspace of the metric space of all continuous functions   $[a,b] \to \R$, where $d(h,g) = \sup_{a \leq x \leq b}|h(x) - g(x)|$. We consider from now on  \'etale-local systems of functions on $\fX$ with values in a topological space $\Gamma$ of this type. Notice that $\Gamma$ has a partial ordering $\leq$, where $h \leq g$ if $h(x) \leq g(x)$ $\forall x \in [a,b]$, and that if $g - \veps < h < g + \veps$, with $\veps >0$ a constant function on $[a,b]$, then $d(g,h) < 2 \, \veps $.
\end{notation}
We recall  for completeness that a function $\varphi:T \to \Gamma$, where $T$ is
any topological space is \emph{upper semicontinuous} or USC
(resp. \emph{lower semicontinuous}  or LSC) if $\forall t_0 \in
T$ and $\veps >0$, there exists a neighborhood $U_{t_0,\veps}$ of
$t_0$ in $T$ such that
$$
\varphi(t) <  \varphi(t_0)  + \veps \;\; {\rm ( resp.}\;\; \varphi(t) >  \varphi(t_0)  - \veps \; {\text )}\;
$$
$\forall t \in U_{t_0,\veps}$. If $\forall \a \in I$, $\varphi_{\a}$
is USC (resp. LSC), then
$$
\varphi = \inf_{\a \in I} \varphi_{\a} \;\; {\rm ( resp.}\;\; \varphi = \sup_{\a \in I} \varphi_{\a}  \; {\text )}\;
$$
is USC (resp. LSC). Notice that if $\varphi$ is both USC and LSC at $t_0 \in T$, then it is continuous at $t_0$.
From now on in this section  $k$ is algebraically closed and $X$ be a rig-smooth strictly $k$-affinoid curve.
\begin{defn}  A strictly affinoid domain $V \subset X$ is said to be \emph{special} if it is either isomorphic to a closed disk in $\tau_{X}^{-1}(x)$ for a  point $x \in S(X)$ of type $(2)$, or it is $\tau_{X}^{-1}(L)$, where $L$ is a $k$-rational closed interval in a closed edge of ${\bf S}(X)$, such that at most one of the ends of $L$ is a vertex (the case of $L$ consisting of  a single  point is not excluded). In the former case, $S(V) = \Gamma (V)$ is the maximal point of the disk and, in the latter case, $S(V) = L$ and $\Gamma(V) = \cV(V)$ consists of the ends of $L$.
\end{defn}
In the following theorem, for a function $f:X \to \Gamma$ and for any algebraically closed non-archimedean field $k^\p$ over $k$, we let $X_{k^\p} := X \wt k^\p$, $\psi_{X,k^\p/k} : X_{k^\p}  \to X$ be the  canonical projection, and 
$f_{k^\p} := f \circ \psi_{X,k^\p/k}$ be the pull-back of $f$ to $X_{k^\p}$.

\begin{thm}\label{contthm} Any function $f:X \to \Gamma$ with the following five properties is continuous:
\begin{enumerate}
\item  $f$ is continuous at every $k$-rational point of $X$; 
\item the restriction of $f$ to every special affinoid subdomain $V$ of $X$ is such that $\ds f(x) \geq \min_{y \in \Gamma(V)} f(y)$ for all points $x \in V$;
\item \label{contannuluszz}
for any algebraically closed non-archimedean field $k^\p$ over $k$, the restriction of $f_{k^\p}$ to any 
open segment of $X_{k^\p}$ is continuous;
\item
the restriction of $f$ to $S(X)$ is continuous at all vertices of ${\bf S}(X)$;
\item
$f$ is \emph{USC}.
\end{enumerate}
\end{thm} 
Notice that the properties $(3)$ and $(4)$ of the theorem imply that the restriction of $f$ to $S(X)$ is continuous.
\begin{proof} 
First of all, we observe  that, given an algebraically closed non-archimedean field $k^\p$ over $k$, the function $f_{k^\p} : X_{k^\p} \to \Gamma$ possesses the properties  $(1) - (5)$. 
Indeed, this is trivial for the property $(3)$. Property $(4)$ follows from the isomorphism ${\bf S}(X_{k^\p})  \xrightarrow{\ \sim\ }  {\bf S}(X)$. Property $(5)$ is a consequence of lemma \ref{lemma:fieldext}. Furthermore, the validity of $(2)$ for $X_{k^\p}$ follows from its validity for $X$and the property $(3)$. Finally, let $x^{\p}$ be a $k^\p$-rational point of 
$X_{k^\p}$, which is not the preimage of a $k$-rational point of $X$. If $x$ is the image of $x^{\p}$ in $X$, then the preimage of $x$ in $X_{k^\p}$ contains an open neighborhood of $x^{\p}$ in $X_{k^\p}$ (isomorphic to an open disk), and the property $(1)$ follows.
\par
By lemma \ref{lemma:fieldext} we may increase the field $k$ and assume that it is maximally complete (and algebraically closed). By the property $(5)$, it suffices to verify that, for every point $x_0 \in X$ and every $\veps >0$, the set $\{ x \in X \,|\, f (x) > f(x_0) - \veps \}$ contains a neighborhood of the point $x_0$. According to $(1)$, we may assume that $x_{0} \notin X(k)$.
\par
\medskip
Consider first the case when $x_0 \notin S(X)$. We may then assume that $x_0 \in D_{X}(y,1^-)$, for some $y \in X(k)$. Then $D_{X}(y,1^-) \cong D_k(0,1^-)$, $y \mapsto 0$, and $x_0 \mapsto  t_{0,r} \in D_k(0,1^-)$ for some $0 < r <1$. Take now a number  $r < R < 1$. It follows from the property $(3)$ that the restriction of $f$ to the interval $\{t_{0,r^\p} | r \leq r^\p \leq R \}$ is continuous and, therefore, we can find $R$ sufficiently close to $r$ with $f ( t_{0,R}) > f (x_0) - \veps$. Let $V$ be the closed disk $D(0,R^+)$. It is an affinoid neighborhood of $x_0 = t_{0,r}$ in $D(0,1^-)$ with $\Gamma(V) = \{ t_{0,R}\}$ and, by the property $(2)$, we get $f(x) \geq f( t_{0,R}) > f(x_0) - \veps$ for all points $x \in V$.
\par
 Suppose now that $x_0 \in S(X)$. First of all, if $x$ lies in an open edge of ${\bf S}(X)$, then the preimage of that edge under the retraction $\tau_{X} : X \to S(X)$ is isomorphic to an open  $k$-annulus isomorphic to $B(r_1,r_2)$, where $x_0 \mapsto t_{0,r}$, with $r_1 < r < r_2$. By the property $(3)$, the function $(r_1,r_2) \to \Gamma$~: $r^\p \mapsto f(t_{0,r^\p})$ is continuous, and so we can find numbers $r_1 <R_1<r<R_2 <r_2$ sufficiently close to $r$ such that $f(y_1)$, $f(y_2) > f(x_0) - \veps$, where $y_1 =  t_{0,R_1}$ and $y_2 =  t_{0,R_2}$. Let $V$ be the closed annulus $\tau_{\fX}^{-1}([R_1,R_2]) \cong D(0,R_2^+) \setminus D(0,R_1^-)$.
Then $V$ is a neighborhood of $x_0$, and $\Gamma(V) = \{y_1,y_2\}$. Property $(2)$ implies that $f(x) \geq \min (f(y_1),f(y_2)) >  f(x_0) - \veps$ for all  $x \in V$.
\par \medskip
Furthermore, suppose that $x_0$ is a vertex of ${\bf S}(X)$. If $S(X) = \{ x_0 \}$, then $\{ x_0 \}$ is the Shilov boundary of $X$ and, by property $(2)$, one has $f(x) \geq f(x_{0})$ for all $x \in X$. 
Assume therefore that $S(X) \not= \{ x_0\}$.
Consider a connected closed neighborhood $L$ of $x_0$ in $S(X)$ which does not contain vertices of ${\bf S}(X)$ other than $x_0$. One has $L = \cup_{i=1}^n L_i$, where each $L_i$ is homeomorphic to $[r_i,1]$ with $r_i \in |k^{\times}|$ and $x_0$ corresponds to the point 1 of $[r_i,1]$. Let $V_i$ be the affinoid domain $\tau_{X}^{-1}(L_i) \subset X$. 
The union $ \bigcup_{i=1}^n V_i$  is a compact neighborhood of $x_0$ in $X$, so that it suffices to verify that the set $\{x \in V_{i} | f(x) > f(x_{0}) -\veps \}$ is open in $V_{i}$.  
Since the restriction of $f$ to $S(X)$, hence to $L$, is continuous at $x_0$, for any given $\veps >0$, we may find $r_i \in |k^{\times}| \cap (0,1)$ so close to 1 that $f(y_i)>  f(x_0) - \veps$. Property $(2)$ then implies that $f(x) > f(x_0) - \veps $  for all $x \in V_{i}$, and the theorem follows.
\end{proof}

\section{Normalized radius of convergence} \label{raddef}

\subsection{Formal structures} \label{formalstructures}
\begin{notation} \label{merosing}  
In this section,   $k$ is supposed to be non trivially valued, algebraically closed and  \emph{of characteristic zero}. 
Here $X$ is the complement 
\beq
 X = \ol{X} \setminus \cZ \;  ,
\eeq
of a finite set (possibly empty)   $\cZ = \{z_1,\dots,z_r \}$,  consisting of $r$ distinct $k$-rational points  in a compact rig-smooth connected $k$-analytic curve $\ol{X}$.  We assume  that  $\ol{X}$ is the generic fiber of a semistable $\kc$-formal scheme  $\ol{\fX}$. 
It is well known that  the reduced positive divisor $\cZ$  determines a finite flat closed reduced  subscheme $\fZ$ of $\ol{\fX}$.  We call $\fZ$ the \emph{schematic closure} of $\cZ$ in $\ol{\fX}$.
\par
It is known \cite{Te} that there is an admissible blow-up 
$\fY$ of  $\ol{\fX}$ such that the inverse image of $\fZ^{\p}$ of $\fZ$ in $\fY$ is an effective  relative Cartier divisor  of $\fY$ finite \'etale over $\Spf \kc$ which does not intersect  the singular locus of $\fY$.
\par
So, \emph{we will assume in the following that  the schematic closure $\fZ$ of $\cZ$ in $\ol{\fX}$  is  \'etale and is contained in the maximal smooth open formal subscheme of $\ol{\fX}$}. We point out that this fact is equivalent to 
 the requirement  that  the disks  $D_{\ol{\fX}}(z_1,1^-), \dots, D_{\ol{\fX}}(z_r,1^-)$ are distinct and that their boundary points $\tau_{\ol{\fX}}(z_1), \dots, \tau_{\ol{\fX}}(z_r)$ are vertices of ${\bf S}(\ol{\fX})$.  
 \begin{rmk}
 We keep the notational distinction between $\cZ$ and $\fZ$ mainly for future use. 
\end{rmk}

\end{notation}

We generalize the skeleton  functor 
\beq \begin{array} {cccc}
 {\bf S} ~: & \cF\cS(X)& \longrightarrow & \cG\cP^{c}(X)
\\
  & \fX & \longmapsto & {\bf S}(\fX) \;  ,
\end{array}
\eeq 
described in  theorem \ref{skelfunctor} for compact $X$, to the present non compact case, as follows. 
\begin{defn}
A \emph{subpolygon} of $X$ is a triple ${\bf S} =(S, \cV, \cE)$ consisting of a connected closed subset $S$ of $X$, of a finite set $\cV \subset X$ of points $v$ of type $(2)$ (the \emph{vertices of $\bf S$}), and of a finite set $\cE$ of \emph{open edges} of $\bf S$. It is required that the elements of $\cE$ are either open  $k$-rational segments (the \emph{open edges of finite length of $\bf S$}) or  open $k$-rational half-lines of $\ol{X}$ (the \emph{open edges of infinite length of $\bf S$}), and that $S$ is the disjoint union of the elements of $\cE$ and of $\cV$.  A  subpolygon ${\bf S} =(S, \cV, \cE)$ of $X$ is \emph{complete} if $X \setminus S$ is a union of open disks. 
\end{defn}
\begin{defn} \label{openskeleton} Let $\fY \to \ol{\fX}$ be a semistable model of $\ol{X}$ above $\ol{\fX}$. We define the \emph{$\fY$-skeleton of $X = \ol{X} \setminus \cZ$} as the subpolygon of $X$,
${\bf S}_{\fZ}(\fY) = (S_{\fZ}(\fY),\cV(\fY), \cE(\fY) \cup \{\ell_{1}, \dots,\ell_{r}\})$, where  $S_{\fZ}(\fY) =S(\fY) \cup  \bigcup_{i=1}^{r}\ell_{i}$, and $\ell_{i}$ is the half-line
 $S(D_{\fY}(z_i, 1^-) \setminus \{z_i\})$, for $i = 1, \dots, r$.
 A \emph{$\fY$-normalized coordinate at $z_{i}$} is an isomorphism $T_{z_i}: D_{\fY}(z_i, 1^-)  \xrightarrow{\ \sim\ }  D(0,1^-)$ with $T_{z_i}(z_i ) = 0$. 
 We will say that the open half-line  $\ell_{i}$ \emph{connects} $z_i$ to $\tau_{\fY}(z_i)$.  For any rigid point $x \in X_{0}$, we define the \emph{$\fY$-maximal open disk  $D_{\fY,\fZ}(x, 1^-)$ in $X$, centered at $x$} as the maximal open disk contained in $D_{\fY}(x, 1^-) \cap X$. 
\end{defn} 
 The boundary point 
   $\tau_{\fY,\fZ} (x)$ of  $D_{\fY,\fZ}(x, 1^-)$ in $\ol{X}$ belongs to $S_{\fZ}(\fY)$ and, as
   in (\ref{limin}) we have a canonical continuous retraction 
  \beq \label{eq:openskeleton}
  \tau_{\fY,\fZ} : X \to S_{\fZ}(\fY) \; .
  \eeq

\par

 Let  $\cE$ be  a locally free $\cO_X$-Module of finite type equipped with an $X/k$-connection $\na$.  Notice that the abelian sheaf 
$$
\cE^{\na} = \ker (\na : \cE \to \cE \otimes \Omega^1_{X/k}) 
$$
 for the $G$-topology of $X$, is not in general locally constant. We use a canonical coordinate $T: D_{\fY,\fZ}(x, 1^-) \xrightarrow{\ \sim\ } D(0,1^-)$, with $T(x) =0$, on $D_{\fY,\fZ}(x, 1^-)$, and define, for $r \in (0,1)$, $D_{\fY,\fZ}(x, r^\pm)$ as  $T^{-1}(D(0,r^\pm))$. 
 For any $r \in (0,1) \cap |k|$, the restriction of $\cE$ to $D = D_{\fY,\fZ}(x, r^+) \subset 
 D_{\fY,\fZ}(x, 1^-)$  is a free $\cO_D$-Module of finite type. Let us choose a basis $\ul e : \cE_{|D} \xrightarrow{\ \sim\ } \cO_D^{\mu}$. The coordinate column vector $\vec y$  with respect to the basis $\ul e$ of a horizontal section of $(\cE_{|D},\na_{|D})$ satisfies a system of differential equations of the form (\ref{diffCHDWintro})
where $G \in M_{\mu \times \mu}(\cO_X(D))$ depends on $r$ and on the choice of the basis $\ul e$. 
By iteration of  the system (\ref{diffCHDWintro}) we obtain,
for any $i \in \N$, the equations
\beq\label{eq:stratintro}
(i!)^{-1}(\frac{d}{dT})^i \, \vec y   = G_{[i]}\, \vec y \;\;\;\;,
\eeq
with $G_{[i]} \in M_{\mu \times \mu}(\cO_X(D))$.
\par
The  power series $Y(T) = \sum_iG_{[i]}(x)(T-x)^i$ 
converges in a neighborhood $U$ of $x \in D$ to the unique solution matrix of the system (\ref{diffCHDWintro})  such that $Y(x)$ is the $\mu \times \mu$ identity matrix. So, $\ul e \,Y \in \cE^{\na}(U)^{\mu}$ is a row of horizontal sections of $\cE (U)$. 
For any $r \in (0,1) \cap |k|$, the value
\beq 
\cR_r (x)  := \min (r, \liminf_{i \to \infty}|G_{[i]}(x)|^{-1/i}) \in (0,r] \; ,
\eeq
is the radius of the maximal open disk $D(0,\rho^-)$, contained in $D(0,r^+)$, such that the power series $Y(T)$ converges in $D(0,\rho^-)$. It is independent of the choice of $T$ and of the basis $\ul e$ of $\cE$ restricted to $D_{\fY,\fZ}(x,r^+)$. 
The  matrix $Y(T)$ is in fact an {\it invertible}  solution matrix of the system (\ref{diffCHDWintro}) holomorphic in $D(0,\cR_r(x)^-)$, since the differential system for the wronskian $w := \det Y$, namely
\beq\label{eq:wronskian}
 \frac {d \, w}{dT} = ({\rm Tr} \, G) \,  w \;\; ,
\eeq
has no singularity in $D(0,1^-)$. We conclude that $\cE^{\na}_{| D_{\fY,\fZ}(x, \cR_r(x)^-)}$ is 
 the constant sheaf  $k^{\mu}$ on $D_{\fY,\fZ}(x, \cR_r(x)^-)$. 
 
\par
For $r_1 < r_2 < 1$ we have $\cR_{r_1}(x) \leq \cR_{r_2}(x) \leq 1$,  and  therefore the 
quantity 
\beq \label{limrad}
\cR(x)   := \lim_{r \to 1} \cR_r (x)  \;\;\;  {\rm (} \; r \in |k| \cap (0,1) \; {\rm )} \; ,  
\eeq
is well-defined and belongs to $(0,1]$. We conclude that $\cE^{\na}_{| D_{\fY,\fZ}(x, \cR(x)^-)}$ is a locally constant sheaf with fiber $k^{\mu}$ on $D_{\fY,\fZ}(x, \cR(x)^-)$ (for both the natural and the $G$-topology). Since 
$D_{\fY,\fZ}(x, \cR(x)^-)$ equipped with the natural topology  is contractible, $\cE^{\na}_{| D_{\fY,\fZ}(x, \cR(x)^-)}$ is in fact 
 the constant sheaf  $k^{\mu}$ on $D_{\fY,\fZ}(x, \cR(x)^-)$, and $\cE_{| D_{\fY,\fZ}(x,\cR(x)^-)}$ is free of rank $\mu$.
 
\par
 From the previous discussion, we conclude the following
 \begin{prop} If for some  analytic domain $V \subset X$,   $\cE^{\na}_{|V}$ is  locally constant, then it is necessarily a local system of $k$-vector spaces of rank $\mu$ on $V$ and the canonical 
monomorphism
\beq
\label{eq:RHsheafintro2}
\cE^{\na}_{|V}   \otimes_k \cO_V \hookrightarrow  \cE_{|V} \;  ,
\eeq
is  an isomorphism.  For any $x \in X(k)$, $D_{\fY,\fZ}(x,\cR(x)^-)$ is the maximal open disk $\cD$ centered at $x$, not intersecting 
$S_{\fZ}(\fY)$, and 
such that $\cE^{\na}$ is a locally constant sheaf on $\cD$, or, equivalently, such that the restriction of $(\cE,\na)$ to $\cD$ is isomorphic to the trivial connection on $\cO_{\cD}^{\mu}$.
\end{prop}
\par

We may then give the 
\begin{defn} \label{def:altradius} Let $X = \ol{X} \setminus \fZ$ be a rig-smooth connected $k$-analytic curve  as in \ref{merosing}. Let $(\cE,\na)$ be an object of 
${\bf MIC}(X/k)$, with $\cE$ locally free of rank $\mu$ for the $G$-topology. We fix a semistable model $\fY$ of $\ol{X}$ over $\ol{\fX}$ as in (\ref{merosing}). For any $k$-rational point $x \in X$, we define the \emph{$\fY$-normalized radius of convergence 
$\cR_{\fY,\fZ} (x,(\cE,\na))$ (or $\cR_{\fY,z_{1},\dots,z_{r}} (x,(\cE,\na))$) of $(\cE,\na)$ at $x$} as the radius $\rho \in (0,1]$ of the  maximal open disk $\cD$ centered at $x$, contained in $X$ and not intersecting the closed subset   $S_{\fZ}(\fY) \subset X$, such that  $(\cE,\na)_{|\cD}$ is isomorphic to the trivial connection $(\cO_{\cD},d_{\cD})^{\mu}$. 
\end{defn}

\begin{lemma} \label{etalechange2} Let $X$, $\ol{X}$, $\fY$, $(\cE,\na)$ be as in the previous definition, and let $\psi: \fY^{\p} \to \fY$ be an \'etale morphism of semistable $\kc$-formal schemes. Let 
$\ol{X}^{\p} = \fY^{\p}_{\eta}$, $\cZ^{\p} = \psi_{\eta}^{-1}(\cZ)$, $\fZ^{\p} = \psi^{\ast}(\fZ)$, $X^{\p} = \ol{X}^{\p} \setminus \cZ^{\p}$, and $g: X^{\p} \to X$ be the morphism induced by $\psi_{\eta}$. Then 
$g^{\ast}(\cE,\na)$ is an
  object of 
${\bf MIC}(X^{\p} /k)$ and
 \beq
\label{etalechange3} 
\cR_{\fY^{\p},\fZ^{\p}} (x,g^{\ast}(\cE,\na)) = \cR_{\fY,\fZ} (g(x),(\cE,\na)) \;, \; \forall x \in X^{\p} \;.
\eeq
\end{lemma}
\begin{proof}
This follows from the discussion in (\ref{etalechange}) and the definitions.
\end{proof}
Notice that if $\cE$ is free over $\cD = D_{\fY,\fZ}(x, 1^-)$, $x \in X(k)$, and we use a  global basis $\ul e$ of $\cE$ on $\cD$ to transform $\na$ into the system (\ref{diffCHDWintro}),  where the coordinate $T$ is normalized so that $\cD \xrightarrow{\ \sim\ }  D(0,1^-)$, $x \mapsto 0$, formula \ref{limrad} collapses to 
\beq \label{eq:raddefsmooth}
\cR(x, \Sigma) = \min (1, \liminf_{n \to \infty}|G_{[n]}(x)|^{-1/n})\; .
\eeq 
\begin{defn} \label{mersing} Let $\fY$ be a semistable model of $\ol{X}$ above $\ol{\fX}$ and $X = \ol{X} \setminus \cZ$, as before. We will say that the $\cO_{\ol{X}_G}$-Module $\ol{\cE}$ is \emph{coherent} (resp \emph{locally free}) \emph{over} $\fY$ if it is of the form $\sp_{\fY}^{\ast} (\ol{\fE})$, for a coherent (resp. locally free) $\cO_{\fY}$-Module $\ol{\fE}$, where $\sp_{\fY}: \ol{X}_G \to \fY$ is viewed as a morphism of $G$-ringed spaces. We denote by ${\bf MIC}_{\fY}(X/k)$
(resp. ${\bf MIC}_{\fY}(\ol{X}(\ast \cZ)/k)$) the full subcategory of $ {\bf MIC}(X/k)$ consisting  of pairs $(\cE,\na)$, where $\ol{\cE}$ is  an $\cO_{\ol{X}_G}$-Module  coherent and locally free over $\fY$ (resp. and $\nabla$ extends to a connection on $\ol{\cE}$ with meromorphic singularities at $\cZ$
$$\ol{\na}: \ol{\cE} \to \ol{\cE} \otimes \Omega^1_{\ol{X}_G/k}(\ast \cZ) \; \; \text{).}
$$ 
\end{defn}
\par
We will  prove the following 
\begin{thm} \label{context} Let $(\cE,\na)$ be an object of ${\bf MIC}_{\fY}(X/k)$. 
The function $X(k) \to \R_{>0}$, $x \mapsto \cR_{\fY,\fZ} (x,(\cE,\na)) $,   extends (uniquely) to a continuous function $X \to \R_{>0}$. 
\end{thm}
The proof of theorem \ref{context} will be based in fact on the construction of a function  $x \mapsto \cR_{\fY,\fZ} (x,(\cE,\na))$, at all $x \in X$, extending the definition given in this section for $k$-rational points $x \in X(k)$.  Namely, we set the following 
\begin{defn} \label{nonrigrad} Let $(\cE,\na)$ be an object of ${\bf MIC}_{\fY}(X/k)$. 
For  any $x \in X$, we consider the base-field extension 
$$\psi_{\ol{X},\sH(x)/k} : \ol{X}^\p := \ol{X} \wt \sH(x) \to \ol{X} \; ,
$$
inducing
$$\psi_{X,\sH(x)/k} : X^\p := X \wt \sH(x) \to X \; ,
$$
and  the canonical $\sH(x)$-rational point $x^\p$ of $X^\p$ above $x$ (\cf the comments following proposition 1.4.1 of \cite{BerkovichEtale}). For any  semistable model $\fY$ of $\ol{X}$, $\fY^\p = 
\fY \hat{\otimes} \sH(x)^{\circ}$ is a semistable model of $\ol{X}^\p$; let $z^\p_i$ be the canonical inverse image of $z_i \in \ol{X}$ in $\ol{X}^\p$.  
We set 
$$ (\cE^\p,\na^\p) = \psi_{X,\sH(x)/k}^{\ast} (\cE,\na) \;, 
$$
and 
\beq \label{nonrigrad2}
 \cR_{\fY,z_1,\dots,z_r} (x,(\cE,\na)) := \cR_{\fY^\p,z^\p_1,\dots,z^\p_r} (x^\p,(\cE^\p,\na^\p)) \; .
\eeq
\end{defn} 

Let $\veps \in |k| \cap (0,1)$, and let 
$$X^{(\veps)} = \ol{X} \setminus \bigcup_{i=1}^r D_{\fY}(z_i,\veps^-) \; .
$$
Recall that we have assumed that all the points $\tau_{\fY}(z_{i})$ are vertices of ${\bf S}(\fY)$. Therefore, if $\sP_{\fY} = \sP(\cB,\cC)$ is the semistable partition of $\ol{X}$ associated to $\fY$, each of the open disks $D_{\fY}(z_i,1^-)$ is a maximal disk in an affinoid $C_{i} \in \cC$, for  $i =1,\dots,r$. 
We consider the semistable partition $\sP^{(\veps)} = \sP(\cB^{(\veps)},\cC^{(\veps)})$, where, for each $i$, we have replaced the affinoid $C_{i}$ by the two affinoids $C_{i} \setminus D_{\fY}(z_i,1^-)$ and $D_{\fY}(z_i,\veps^+)$, and have added the open annulus $D_{\fY}(z_i,1^-) \setminus D_{\fY}(z_i,\veps^+)$.
So, $\sP^{(\veps)} \geq \sP_{\fY}$ corresponds to a morphism $\fY^{(\veps)} \to \fY$ in $\cF\cS(\ol{X})$, which is the admissible blow-up of the $\cO_{\fY}$-ideal of sections of $\cO_{\fY}$ whose pull-back to $\ol{X}$ is $<1$ 
 on the affinoid disks $D_{\fY}(z_i,\veps^+)$, 
  $i =1,\dots,r$. The closed fiber $\fY^{(\veps)}_{s}$ is the union of $\fX_{s}$ and of $r$ projective lines  $\ell_{i}^{(\veps)} \cong \P^{1}_{\kt}$. The line  $\ell_{i}^{(\veps)}$ is at the same time the canonical reduction of $D_{\fY}(z_i,\veps^+)$, and the image of $D_{\fY}(z_i,\veps^+)$ via $\sp_{\fY^{(\veps)} }$. 
Let $\fX^{(\veps)} \subset \fY^{(\veps)}$ be the open formal subscheme $\sp_{\fY^{(\veps)}}^{-1}(\fY^{(\veps)}_{s} \setminus \bigcup_{i=1}^{r} \ell_{i}^{(\veps)} )$. \par
We summarize the previous discussion:
 \begin{lemma} \label{compact} There is a morphism of  semistable formal schemes $\fX^{(\veps)} \to \fY$, composite of an open immersion $\fX^{(\veps)} \to \fY^{(\veps)}$ and of an admissible blow-up $\fY^{(\veps)} \to \fY$, whose generic fiber identifies with the embedding $X^{(\veps)} \subset \ol{X}$. 
 \end{lemma}
 The restriction 
$(\cE,\na)_{|X^{(\veps)}}$ of  $(\cE,\na)$ to $X^{(\veps)}$, is an object of ${\bf MIC}_{\fX^{(\veps)} }(X^{(\veps)}/k)$, and, since $\forall x \in X^{(\veps)}(k)$, $D_{\fY,\fZ}(x,1^{-}) = D_{\fX^{(\veps)} }(x,1^{-})$, and, according to our definition, we are allowed to replace the field $k$ by its extension $\sH(x)$, we conclude that
$\cR_{\fX^{(\veps)} } (x,(\cE,\na)_{|X^{(\veps)}}) = \cR_{\fY,z_1,\dots,z_r} (x,(\cE,\na))$, for any $x \in X^{(\veps)}$.
So, we may restrict our attention to the case where there are no $z_i$'s and $\ol{X} = X = \fY_{\eta}$, for a semistable formal scheme $\fY$. 

Let $(\fB,T)$ be an \'etale neighborhood of $\fY$ which is a basic formal annulus or disk. Suppose that 
the image $B$ of $(\cB,T) = (\fB,T)_{\eta}$ in $\ol{X}$ is contained in $X$, and contains the point $x \in X$. 
Now, 
the canonical point $x^{\p} \in X^{\p}$ above $x$ is an interior point of the image  $B^{\p}$ in $\ol{X}^{\p}$, of 
$\cB^{\p} :=  \cB \wt \sH(x) = \fB^{\p}_{\eta}$, where $\fB^{\p} := \fB \wt \sH(x)^{\circ}$. More precisely, $D_{\fY^\p,z^\p_1,\dots,z^\p_r}(x^{\p},1^{-})$ is contained in $B^{\p}$, and, for any choice of an inverse image $y^{\p}\in \cB^{\p}$ of  $x^{\p}$, it  is canonically isomorphic to $D_{\fB^{\p} }(y^{\p},1^{-})$. If $\fB$ is a disk, the coordinate $T$ of $\fB$ induces the normalized coordinate $T_x = T-T(x)$ on $D_{\fY^\p,z^\p_1,\dots,z^\p_r}(x^\p,1^-)$. If $\fB$ is an annulus, 
the \'etale coordinate $T$ on $\fB$ (which is also an \'etale coordinate on $\fB^{\p}$)  induces a normalized coordinate $T_x$ on $D_{\fY^\p,z^\p_1,\dots,z^\p_r}(x^\p,1^-)$, with $T_x(x^\p) =0$, via (\cf (\ref{coorcann}))
\beq \label{coorcann2} 
T_x(y) = (T(y) - T(x))/T(x)\; , \;  \forall y \in D_{\fY^\p,z^\p_1,\dots,z^\p_r}(x^\p,1^-) \; .
\eeq
If the pull-back of $\cE$ on $\cB$ is free, there exists a global basis $\ul e$ of $\cE$ on the basic affinoid annulus $(\cB,T)$, and we use such a global basis to transform $\na$ into the system (\ref{diffCHDWintro}),  formulas \ref{limrad}  and \ref {eq:raddefsmooth} become
\beq \label{eq:raddefsmooth21}
\cR (x,\Sigma)  = \min (1, \liminf_{n \to \infty}|G_{[n]}(x)|^{-1/n})  \; ,  
\eeq
for all  $x \in B$, if $B$ is a disk, and  
\begin{multline} \label{eq:raddefsmooth2}
\cR (x,\Sigma)  = \min (1, |T(x)|^{-1}\liminf_{n \to \infty}|G_{[n]}(x)|^{-1/n})  \\
= \min (|T(x)|, \liminf_{n \to \infty}|G_{[n]}(x)|^{-1/n}) |T(x)|^{-1}
\; ,  
\end{multline} 
for all  $x \in B$, if $B$ is an annulus. 
Notice that the previous formulas are independent of the choice of the global basis $\ul e$.
\par
The previous formulas are especially useful if $(\cE,\na)$ is an object of 
${\bf MIC}_{\fY}(\ol{X}(\ast \cZ)/k)$ as in \ref{mersing} and $(\cB,T)$ is a basic affinoid  annulus or disk, generic fiber of a basic formal annulus or disk $(\fB,T)$ varying in a covering of the type described in lemma \ref{cover}. Then, $\cR_{\fY,\fZ} (x,(\cE,\na))$ may be calculated locally on $\cB$, that is 
\beq \label{eq:raddefsmooth3}
\cR_{\fY,\fZ} (x,(\cE,\na))  = \cR_{\fB}(x,(\cE,\na)_{|\cB}) \; , \; \forall x \in \cB \; ,
\eeq
and, if $\cE_{|\cB}$ is free of finite type and the connection is described via (\ref{diffCHDWintro})  in terms of an $\cO(\cB)$-basis $\ul e$ of $\cE(\cB)$, this in turn is computed as in (\ref{eq:raddefsmooth2}).

\subsection{Apparent singularities and change of formal model} \label{apparent}
We  describe here how the function $x \mapsto \cR_{\fY,\fZ} (x,(\cE,\na))$ gets modified, 
by a change of formal model $\fY$ of $\ol{X}$ or by addition of  an extra point $z = z_{r+1} \in X(k)$ to $\cZ = \{z_1, \dots, z_r\}$. 
Notice that, if $\fY_1 \to \fY$ is a morphism of   semistable models of $\ol{X}$ inducing the identity on $\ol{X}$, there is a well-defined function $\rho_{\fY_1/\fY} : \ol{X}(k) \to (0,1) \cap |k|$, such that 
\beq \label{eq:dilatation1} D_{\fY_1}(x,1^-) = D_{\fY}(x,\rho_{\fY_1/\fY}(x)^-) \; , \; \forall \, x \in \ol{X}(k) \; .
\eeq
The function $\rho_{\fY_1/\fY} $ is extended to all points  $x \in \ol{X}$, by using the canonical $\sH(x)$-rational point $x^\p \in  \ol{X}^\p := \ol{X} \what{\otimes} \sH(x)$ above $x$. Then $\fY^\p := \fY \what{\otimes} \sH(x)^{\circ}$  (resp. $\fY_1^\p := \fY_1 \what{\otimes} \sH(x)^{\circ}$) is a semistable model of $ \ol{X}^\p$, and 
\beq \label{eq:dilatation2} \rho_{\fY_1/\fY}(x) := \rho_{\fY^\p _1/\fY^\p }(x^\p ) \; .
\eeq

The following lemma follows easily from the  description \ref{hyperb} of $\H(\ol{X}) \subset \ol{X}$. 
\begin{lemma} Let $\varphi:\fY_1 \to \fY$ be a morphism in $\cF\cS(\ol{X})$ and $\tau_{\fY}: \ol{X} \to  {\bf S}(\fY)$ (resp. $\tau_{\fY_1}: \ol{X} \to  {\bf S}(\fY_1)$, resp. 
$\tau_{\fY, \fY_1}: {\bf S}(\fY_1) \to  {\bf S}(\fY)$) be the retraction described in  (\ref{retr}) (resp. in (\ref{retr}), resp. in (\ref{limin})). Then, for any $x \in \ol{X}$, $\tau_{\fY_1}(x)$ and $\tau_{\fY}(x)$ belong to the same fiber $\cC$ of $\tau_{\fY, \fY_1}: {\bf S}(\fY_1) \to {\bf S}(\fY)$.
We have
\beq \rho_{\fY_1/\fY}(x) = \ell_{\cC}([\tau_{\fY_1}(x),\tau_{\fY}(x)])^{-1} \; ,
\eeq
where $\ell_{\cC}([\tau_{\fY_1}(x),\tau_{\fY}(x)])$ denotes the length of the segment $[\tau_{\fY_1}(x),\tau_{\fY}(x)] \subset \cC \subset \H( \ol{X})$. In particular, the function $ \rho_{\fY_1/\fY}$ is continuous.
\end{lemma}
\begin{proof} If $\varphi_1:\fY_1 \to \fY$ and $\varphi_2:\fY_2 \to \fY_1$ are morphisms in $\cF\cS(X)$, clearly $\rho_{\fY_2/\fY} = (\rho_{\fY_2/\fY_1}\circ (\varphi_1)_{\eta}) \cdot \rho_{\fY_1/\fY} $. So, we are reduced to proving the statement when either $\tau_{\varphi}$ is trivial, in which case it is trivial, or simple, in which case it is simple. 
\end{proof}
Similarly, if $\ol{\fX}$ is as before, $\fY_1 \to \fY \to \ol{\fX}$ are morphisms  in $\cF\cS(\ol{X})$, and $z_{r+1},\dots,z_{r+s} \in X(k)$ are such that $\{ D_{\fY_1}(z_1,1^-), \dots, D_{\fY_1}(z_{r+s},1^-) \}$ consists of $r+s$ distinct disks, with $\tau_{\fY_1}(z_{1}), \dots, \tau_{\fY_1}(z_{r+s})$ vertices of ${\bf S}(\fY_{1})$, 
let $\cZ = \{z_{1},\dots,z_{r}\}$ (resp. $\cZ_1 = \{z_{1},\dots,z_{r+s}\}$). We may define a function $\rho_{(\fY,\fZ)/(\fY_1,\fZ_1)} : \ol{X} \setminus \cZ_1 \to (0,1) \cap |k|$, such that
\beq  \label{eq:dilatation3}
D_{\fY_1,\fZ_1}(x,1^-) = D_{\fY,\fZ}(x,\rho_{(\fY,\fZ)/(\fY_1,\fZ_1)}(x)^-) \; , \; \forall \, x \in \ol{X}(k) \setminus \cZ_1 \; ,
\eeq
extended as before to all points of $\ol{X} \setminus \cZ_1$ and such that
\beq \rho_{(\fY,\fZ)/(\fY_1,\fZ_1)}(x) = \ell_{\cC}([\tau_{\fY_1,\fZ_1}(x),\tau_{\fY,\fZ}(x)])^{-1} \; ,
\eeq
where $\tau_{\fY,\fZ}$ is the function of (\ref{eq:openskeleton}), $\cC$ is the connected (and simply connected) component of 
$S_{\fZ_1}(\fY_1)\setminus S_{\fZ}(\fY)$ such that  both $ \tau_{\fY_1,\fZ_1}(x)$ and $\tau_{\fY,\fZ}(x)$ belong to $\cC$, and $\ell_{\cC}([\tau_{\fY_1,\fZ_1}(x),\tau_{\fY,\fZ}(x)])$ denotes the length of the segment 
$[\tau_{\fY_1,\fZ_1}(x),\tau_{\fY,\fZ}(x)] \subset \cC  \subset  \ol{X} \setminus \cZ_1$. 
 
 \par
The general situation is as follows.
\begin{lemma} Let notation be as before. Let $(\cE,\na)$ is an object of ${\bf MIC}_{\fY}(\ol{X}(\ast \cZ )/k)$, let 
$\fY_1 \to \fY$ be an admissible blow-up, and $z_{r+1},\dots,z_{r+s} \in X(k)$ be such that $\{ D_{\fY_1}(z_1,1^-),\dots$ $ \dots, D_{\fY_1}(z_{r+s},1^-) \}$ consists of $r+s$ distinct disks. Let $\cZ_1 = \{z_{1},\dots,z_{r+s}\}$.
Then 
 $(\cE,\na)$ may be regarded as an object of ${\bf MIC}_{\fY_1}(\ol{X}(\ast \cZ_1\})/k)$ and
\beq \label{eq:change}
\cR_{\fY_1,\fZ_1} (x,(\cE,\na)) = \min (1, \cR_{\fY,\fZ} (x,(\cE,\na))/ \rho_{(\fY,\fZ)/(\fY_1,\fZ_1)}(x)) \;\; ,  \;\forall \, x \in  \ol{X} \setminus \cZ_1 \;.
  \eeq
 \end{lemma} 
\begin{rmk} Let $(\cE,\na)$ be an object of ${\bf MIC}_{\fY}(\ol{X}(\ast \cZ )/k)$ and let $z = z_{r+1} \in X(k) \cap (D_{\fY}(z_1,1^-) \setminus \{z_1\})$. Let  $T_{z_1}$ be a $\fY$-normalized coordinate at $z_1$ and  $\fY_1 \to \fY$ be the admissible blow-up of the ideal $(T_{z_1}(z),T_{z_1})$ of $\cO_{\fY}$. Then, for $\cZ_{1} =  \cZ \cup \{z_{r+1}\}$, with schematic closure $\fZ_{1}$ in $\fY$,
 \beq \cR_{\fY_1,\fZ_1} (x,(\cE,\na)) = 
       \left\{\begin{array}{ccc}\cR_{\fY,\fZ} (x,(\cE,\na)) ,&\mbox{if} &x \notin D_{\fY}(z_1,T_{z_1}(z)^-) , \\
        \min(1, \cR_{\fY,\fZ} (x,(\cE,\na))/|T_{z_1}(x)|) ,&\mbox{if} & x \in D_{\fY}(z_1,1^-) \setminus \{z_1\} ,
        \end{array} \right .
\eeq
The previous formula shows that the unnecessary consideration of the apparent singularity $z_{r+1}$ results in a loss of information. 
\end{rmk}

\section{Differential systems on an annulus: review of classical results} \label{review}

In this section $k$ is a non-archimedean field extension of $\Q_p$. We  let $k^\alg$ be an algebraic closure of $k$. 
\par
We recall the $k$-Banach algebra $\sH(r_1,r_2) \subset \cO(B(r_1,r_2))$ of 
\emph{analytic elements} 
on the open annulus (\ref{openann}), with $0< r_1 \leq r_2$. Its elements are uniform limits on $B(r_1,r_2)$
of rational functions in $k(T)$ without poles in $B(r_1,r_2)$. The Banach norm of $\sH(r_1,r_2)$  is of course the supnorm on $B(r_1,r_2)$. 
An analytic element $\ol{f}  \in \sH(r_1,r_2)$ defines an  analytic function  $f: B(r_1,r_2) \to \A^1$ and a 
continuous extension of $f$ to  the closure $B^+(r_1,r_2) = B(r_1,r_2) \cup \{t_{0,r_1}, t_{0,r_2} \}$ of 
$B(r_1,r_2)$ in $\A^1$, still denoted by $f: B^+(r_1,r_2) \to \A^1$. 
We consider here a differential system $\Sigma$ of the form 
(\ref{diffCHDWintro})
with $G$ a $\mu\times\mu$ matrix with elements  in $\sH(r_1,r_2)$.  The function radius of convergence of $\Sigma$ according to  Christol-Dwork \cite{ChristolDwork} is defined for $\rho \in [r_1,r_2]$ as
 \beq \label{radCHDW} 
R(\Sigma,\rho) := \min (\rho, \liminf_{k \to \infty} |G_{[k]}(t_{0,\rho})|^{-1/k}) \;  , 
\eeq
where $G_{[k]}$ is the matrix describing the action of $(k!)^{-1} ({\frac{d}{dT}})^k$ on any formal column solution $\vec y$. Notice that 
$R(\Sigma,\rho)/\rho$ coincides with  $\cR(t_{0,\rho},\Sigma)$, as defined  in (\ref{eq:raddefsmooth2}),  if the entries of $G$ are analytic functions on some basic affinoid annulus $B = \fB_{\eta}$ with 
$B(r_{1},r_{2}) \subset B \subset B[r_{1},r_{2}]$, and  $\cR(t_{0,\rho},\Sigma)$ is taken with respect to $\fB$. 
\par
We follow the common practice of saying that  a   function $F$,  defined on a subset of $\R_{>0}$ and taking values in $\R_{>0}$, has a certain property $\cal P$ {\it logarithmically} if the function $\log \circ F \circ \exp$ has the property $\cal P$. 
Christol-Dwork observe that the function $\rho \mapsto R(\Sigma,\rho)$ is logarithmically concave (\ie logarithmically $\cap$-shaped)  on the interval $[r_1,r_2]$.  
\par
This  result (first part of \cite[2.5]{ChristolDwork})  follows immediately from the well-known fact that, for any $f \in \sH(r_1,r_2)$, the function $\rho \mapsto |f(t_{0,\rho})|$ 
is  logarithmically convex and continuous  on the interval $[r_1,r_2]$. Moreover, if for $a < b$, $\forall i \in \N$, $\varphi_{i}: [a,b] \to \R$
is a convex (resp. concave) function, then
$$
\varphi = \limsup_{i \to \infty} \varphi_i \;\; {\rm ( resp.}\;\; \varphi = \liminf_{i \to \infty} \varphi_{i}  \; {\text )}\;
$$
is convex (resp. concave). 
So, $\rho \mapsto R(\Sigma,\rho)$ is continuous on $(r_1,r_2)$, and lower semicontinuous  in $r_1$ and $r_2$, by logarithmic concavity. 
\par
In the second part of \lc, Christol-Dwork show  that, as a consequence of their theory of 
the convergence polygon of a differential operator \cf \S 2.4 of \lc,
 $\rho \mapsto R(\Sigma,\rho)$ is also USC  at $r_1$ and $r_2$. They conclude that 
$\rho \mapsto R(\Sigma,\rho)$ is continuous on $[r_1,r_2]$. 
\par
We recall that the  system $\Sigma$ is said to be {\it solvable} at $t_{0,r}$, for $r \in [r_1,r_2]$,  if  $R(\Sigma, r) = r$.
\par
\medskip
Pons  \cite[2.2]{Pons} proves the following theorem.
\begin{thm} \label{slopes} The function 
$\rho \mapsto R(\Sigma,\rho)$  is logarithmically a concave polygon with a finite number of sides  on $[r_1,r_2]$. The slopes of the sides are rational numbers with denominator at most $\mu$. 
\end{thm}
\begin{rmk}
The simpler case of this theorem, for a system solvable at $r_1$ and $r_2$, appears in \S 4 of \cite{ChMe3}. 
We cannot follow the topological arguments of Pons \lc. We prefer to review entirely the  proof of her theorem, based on \cite{ChristolDwork} and \cite{Young}, to make it completely clear. Actually, we need 
to combine the main theorem of \cite[Thm. 5.4]{ChristolDwork}, with its variation proved by Kedlaya  \cite[Thm. 6.15]{Ke}.
\end{rmk}
\begin{proof} Let $\pi \in k^{\alg}$ be such that $\pi^{p-1} = -p$. 
By logarithmic concavity, the continuous function $R(\Sigma,r)/r$ is constantly equal to its maximum value $M$ on an interval $[r_1^\p,r_2^\p] \subset [r_1,r_2]$, with $r_1^\p \leq r_2^\p$, and it is strictly increasing (resp. decreasing) for $r \leq r_1^\p$ (resp. $r \geq r_2^\p$). Then precisely one of the following holds:
\begin{enumerate}
\item \label{solvable} $M = 1$,
\item $M = |\pi|^{1/p^h}$, for some $h \in \Z_{\geq 0}$,
\item $ |\pi|^{1/p^{h-1}} < M < |\pi|^{1/p^h}$, for some $h \in \Z_{>0}$,
\item $M < |\pi|$.
\end{enumerate}

If $M<1$, then the interval $[r_1,r_2]$ can be subdivided in a finite number of intervals of the form $I = [a,b]$, where 
$R(\Sigma,r)/r$ is either strictly increasing or strictly decreasing, and for which precisely one of the following holds:
\begin{enumerate}
\item there exists $h = h(I) \in \Z_{>0}$ such that 
$ |\pi|^{1/p^{h-1}} < R(\Sigma,r)/r < |\pi|^{1/p^h}$, for $r \in (a,b)$,
\item $R(\Sigma,r)/r  <  |\pi|$, for $r \in (a,b)$.
\end{enumerate}

If $M =1$, we can  find an increasing  (resp. decreasing) sequence of points $h \mapsto a_h \in [r_1, r_1^\p)$ (resp. $h \mapsto b_h \in (r_2^\p , r_2]$), defined for $h \in \Z_{\geq 0}$ as soon as 
$R(\Sigma,r_1)/r_1 \leq |\pi|^{1/p^h}$ 
(resp. $R(\Sigma,r_2)/r_2 \leq |\pi|^{1/p^h}$), converging to $r_1^\p$ (resp. $r_2^\p$),  and  such that  $R(\Sigma,a_h)/a_h = |\pi|^{1/p^h}$
(resp. $R(\Sigma,b_h)/b_h = |\pi|^{1/p^h}$). The function $R(\Sigma,r)/r$ satisfies 
$ |\pi|^{1/p^{h-1}} < R(\Sigma,r)/r < |\pi|^{1/p^h}$ for $r \in  (a_{h-1},a_h)$ (resp. $r \in  (b_h,b_{h-1})$), as soon as $a_{h-1}$ (resp. $b_{h-1}$) is defined. If $R(\Sigma,r_1)/r_1 < |\pi|$ 
(resp. $R(\Sigma,r_2)/r_2 < |\pi|$), then $R(\Sigma,r)/r < |\pi|$ in $[r_1,a_0)$ (resp. $(b_0,r_2]$).
\par
When 
  $R(\Sigma,r)/r < |\pi|$,  the precise value of $R(\Sigma,r)$ has been evaluated by Young \cite{Young}. We recall that the residue field  $\sH(t_{0,r})$ is the completion of the rational field $k(T)$ under the absolute value $f(T) \mapsto |f(t_{0,r})|$, where 
  $$ |(a_nT^n + \dots + a_1 T + a_0)(t_{0,\rho})| = \max_{i=0,1,\dots,n} |a_i|r^i
  $$
  for a polynomial $a_nT^n + \dots + a_1 T + a_0 \in k[T]$.
 The derivation $d/dT$ (resp. $\delta := Td/dT$) extends to a continuous derivation of $\sH(t_{0,r})/k$, of operator norm $1/r$ (resp. $1$). 
  By the theorem of the cyclic vector, the  system $\Sigma$ is equivalent over the field $\sH(t_{0,r})$, for any $r \in [r_1,r_2]$, to a single differential operator 
$$L = \delta^{\mu} - C_1\delta^{\mu - 1} - \dots - C_{\mu} \; , \;\;\; \delta = T \frac d{dT} \; .
$$
Young's theorem asserts that if  $R(\Sigma,r)/r <  |\pi|$, or if $\exists \, j_0 = 1, \dots,\mu$ such that $|C_{j_0}(t_{0,r})| > 1$, then 
$$R(\Sigma,r)/r = |\pi| \min_j |C_j(t_{0,r})|^{- 1/j} \;  .
$$
In particular, if $R(\Sigma,r)/r < |\pi|$, then $R(\Sigma,r)/ r$ can only take values of the form $|\pi| (|a|r^s)^{1/j}$, where $a \in k$ and $s \in \Z$ and $j$ is an integer  between 1 and $\mu$. 

\medskip

Let us consider a finite \'etale covering  $\varphi : B(s_1,s_2) \to B(r_1,r_2)$ which is a composite of maps  of the following forms: a {\it Kummer covering} $x \mapsto x^N$, for some $N = 1,2,\dots$, where $r_i = s_i^N$, $i=1,2$, a {\it dilatation} $x \mapsto ax$, for $a \in k^{\times}$, in which case $r_i = |a| s_i$, for $i=1,2$, and an {\it inversion} $x \mapsto \gamma x^{-1}$, with $r_1 = |\gamma| s_2^{-1}$ and $r_2 = |\gamma| s_1^{-1}$, assuming such a $\gamma \in k^{\times}$ exists. Among Kummer coverings we have the \emph{Frobenius covering\footnote{The semilinear version of the map $\varphi_{1}$ is often used instead. Namely, for any continuous lifting $\sigma \in \Aut (k)$ of the absolute Frobenius of $\kt$, one considers $\phi^{(\sigma)} : B(r_1,r_2) \to B(r^{p}_1,r^{p}_2) \wt_{k,\sigma}k$, which is $\sum_ia_{i}T^{i} \mapsto \sum_ia^{\sigma}_{i}T^{pi}$ at the ring level. See section \ref{Frobenius} below. } of degree $h$}: $\varphi_h: B(r_1,r_2) \to B(r^{p}_1,r^{p}_2)$, $x \mapsto x^{p^h}$. For any $\varphi$ as above, we may pull-back the system to a similar system on $B(s_1,s_2)$, which we denote by $\varphi^{\ast}\Sigma$, in an obvious way. 

\begin{rmk}
 The dilatation map $\varphi: x \mapsto ax$, for $a \in k^{\times}$, transforms isomorphically the disc $D(0,1^\pm)$ into  the disc $D(0,|a|^\pm)$, sending $t_{0,1}$ to  $t_{0,|a|}$.  If $r \in |k^{\times}|$, and $|a| = r$, the system $\varphi^{\ast}\Sigma$ is associated to the operator  
$$\varphi^{\ast} L = \delta^{\mu} - C_1(ax)\delta^{\mu - 1} - \dots - C_{\mu} (ax)
$$
satisfies $R(\varphi^{\ast} \Sigma,1) = R(\Sigma,r)/r$. If $r \notin |k^{\times}|$, one can reason in the same way after the base change to $k_r := k\{r^{-1}T_1, rT_2\}/(T_1T_2 -1)$. In this form Young's result appears with full details in \cite[VI, 2.1]{DGS}. In the same vein, one can avoid  consideration of $R(\Sigma,r)$ as $r \to r_1$, and only discuss the case $r \to r_2$, by using the inversion $x \mapsto x^{-1}$. 
\end{rmk}

\begin{lemma}\label{invariance} Let $\gamma \in k^{\times}$ and $\varphi: B(|\gamma|,1) \to B(|\gamma|,1)$ be the inversion $x \mapsto \gamma/x$. Let $\Sigma$ be a system on $B(|\gamma|,1)$ as above. Then, for any $r \in [|\gamma|,1]$, 
$$R(\varphi^{\ast} \Sigma,r)/r = R(\Sigma,|\gamma|/r)/(|\gamma|/r) \; .
$$
\end{lemma}
\begin{proof}  An immediate calculation shows that for any $k$-rational point $a \in B(|\gamma|,1)$ and any $R \in (|\gamma|,|a|)$, $\varphi$ induces an isomorphism,  of the disk 
$D(a, R^+)$ onto the disk $D(\gamma/a,(R\, |\gamma/a^2|)^+)$. Therefore, by the interpretation of $R(\Sigma,r)$ as the radius of the maximal disk centered at the canonical point $a $ above $t_{0,r}$ in $\sH(t_{0,r}) \what{\otimes} B(|\gamma|,1)$, for which $|x(a)| = r$,  where  a fundamental solution matrix of the base change of $\Sigma$ at $a$ converges, we deduce 
$$ R(\Sigma, |\gamma|/r) = (|\gamma|/r^2) R(\varphi^{\ast} \Sigma,r) \; ,
$$
as promised. 
\end{proof}

Of special interest is the pullback by Frobenius: the main theorem of \cite{ChristolDwork} combined with \cite[Thm. 6.15]{Ke} asserts that if, for some fixed $h \in \Z_{>0}$, a system $\Sigma$ of the form (\ref{diffCHDWintro}) satisfies 
$$ |\pi|^{1/p^{h-1}} < R(\Sigma,r)/r < |\pi|^{1/p^h}
$$
for any $r \in (r_1,r_2)$, then there exists a  system $\Sigma_h$ with coefficients in 
 $\sH(r_1^{p^h},r_2^{p^h})$, unique in the sense of  $\sH(r_1^{p^h},r_2^{p^h})/k$-differential modules, such that $\Sigma \cong  \varphi_h^{\ast} \Sigma_h$, where ``$\cong$'' means isomorphism of $\sH(r_1,r_2)/k$-differential modules. Moreover 
$$R(\Sigma_h,r^{p^h})/r^{p^h} =  (R(\Sigma,r)/r)^{p^h} \; ,
$$
so that $R(\Sigma_h,r)/r <  |\pi|$ in $(r_1^{p^h},r_2^{p^h})$, and Young's theorem can be applied. Notice that if the graph of $R(\Sigma_h,r)/r$ is logarithmically affine with slope $\alpha$  in the interval $[r_1^{p^h},r_2^{p^h}]$, so is $R(\Sigma,r)/r$ in $[r_1,r_2]$: 
$$R(\Sigma_h,r)/r = C r^{\alpha}  
\Rightarrow
R(\Sigma_h,r^{p^h})/r^{p^h} = C r^{\alpha p^h} 
\Rightarrow
R(\Sigma,r)/r = C^{1/p^h}r^{\alpha} \; .
$$
So theorem \ref{slopes} is proven in the case $R(\Sigma,r)/r <1$ in $[r_1,r_2]$. If the function $R(\Sigma,r)/r$ reaches the maximum value 1 in $[r_1,r_2]$, we are in the case treated by Christol-Mebkhout, and can follow their argument \cite[4.2]{ChMe3}. Namely, we operate as in case $M=1$ above, and reduce, possibly by an inversion, to the case when $R(\Sigma,r)/r$ is strictly increasing to 1 in $[r_1,r_2]$.  Then we consider the sequence $h \mapsto a_h \in [r_1,r_2]$, described above (where $r_1^\p = r_2$), defined for $h \in \Z_{\geq 0}$ as soon as 
$R(\Sigma,r_1)/r_1 \leq |\pi|^{1/p^h}$, converging to $r_2$, and  such that  $R(\Sigma,a_h)/a_h = |\pi|^{1/p^h}$. The function $R(\Sigma,r)/r$ satisfies 
$ |\pi|^{1/p^{h-1}} < R(\Sigma,r)/r < |\pi|^{1/p^h}$ for $r \in  (a_{h-1},a_h)$, as soon as $a_{h-1}$ (resp. $b_{h-1}$) is defined. If $R(\Sigma,r_1)/r_1 < |\pi|$, then $R(\Sigma,r)/r < |\pi|$ in $[r_1,a_0)$. So, by the previous argument, the function $R(\Sigma,r)/r$ is continuous, logarithmically concave and piecewise affine with rational slopes with denominator at most $\mu$. But in this special case, those slopes are positive and must be decreasing as $r \to r_2$. The constraint on the denominator shows that  there is a non negative rational number $\beta$ with denominator bounded by $\mu$, such that, for sufficiently big $h$, the function $R(\Sigma,r)/r$ on the interval $[a_{h-1},a_h]$ is of the form $C_hr^{\beta}$. So, $C_h$ is independent of $h$, and, since $R(\Sigma,r_2)/r_2 = 1$, $C_h = r_2^{-\beta}$. So, theorem \ref{slopes} is proven in every case. 
\end{proof}
\par 
\begin{cor} \label{piecewiselinear} Let $X$ be a closed $k$-annulus  in the analytic $k$-line $\A^1_k$, and (\ref{diffCHDWintro}) be a  $\mu \times \mu$ system of linear differential equations on $X$.  Let $r:\, S(X) \xrightarrow{\ \sim\ } [r_{1},r_{2}]$ be the function ``radius of a point''  of \cite[p.~78]{Berkovich}, restricted to $S(X)$. Then, the restriction of $x \mapsto R(\Sigma,r(x))/r(x)$ to $S(X)$  is the infimum of the constant 1 and a finite set of functions of the form
\beq |p|^{1/(p-1)p^h}|b|^{1/jp^h} r(x)^{s/j}\; ,
\eeq
 where $j \in \{ 1, 2, \dots,\mu\}$, $s \in \Z$,  $h \in \Z \cup \{\infty\}$, $b \in k^{\times}$.
 \end{cor}
 \begin{exa} (\cite[IV.7.3]{DGS}) \label{xalpha} Let $\alpha \in \kc$ and regard 
 \beq 
 \Sigma_{\alpha} ~: \frac {d\, y}{dT}= \, \frac{\alpha}{1+T} \, y \;  ,
 \eeq
 as an analytic differential equation on any strictly affinoid annulus of the form 
 $$X = B[-1;r_1,r_2] \subset  \A^{1}_{k} \setminus \{-1\} \; ,$$
  with its minimal semistable formal structure $\fX$. For any $x \in X$, the solution of $\Sigma_{\alpha}$ on $\A^{1}_{\sH(x)}$ which takes the value 1 at (the canonical point over) $x$  is 
$$g_{\alpha}(\frac{T- T(x)}{1+ T(x)}) = (1 + \frac{T- T(x)}{1+T(x)})^{\alpha} \in 1 + (T- T(x))\sH(x)[[T-T(x)]] \; . $$
 Let 
$$d(\alpha,\Z_{p}) = \inf \{|n - \alpha| \, | \, n \in \Z\,\}  \in |\kc|\; ,
$$
and assume   $p^{-m-1 }  < d(\alpha,\Z_{p}) \leq   p^{-m } $,  
with $m \in \Z_{\geq 0}$.
The radius of convergence of the power series  
$$g_{\alpha}(U) = 1 + \sum_{i=1}^{\infty} {{\alpha}\choose{i}} U^{i} \;,$$ 
is then 
$$r_{\alpha} := |\pi|^{\frac{1}{p^{m}}} (d(\alpha,\Z_{p})  p^{m })^{ - \frac{1}{p^{m+1}}} =
p^{-\frac{1}{(p-1)p^{m}} - \frac{m}{p^{m+1}}}  d(\alpha,\Z_{p})^{ - \frac{1}{p^{m+1}}}  \; ,$$
so that 
$$ \cR_{\fX}(x, \Sigma_{\alpha}) = R(\Sigma_{\alpha}, x)/|1+T(x)| = r_{\alpha} \; 
$$
is of the form predicted by corollary \ref{piecewiselinear} on $S(X) = S(\fX)$ and is in fact constant all over $X$.
 \end{exa}

\section{Dwork-Robba theory over basic annuli and disks}   \label{GeneralDR}

We keep the notation and assumptions of the previous section on the field $k$, but the  system $\Sigma$ of (\ref{diffCHDWintro}) is supposed to be defined on a basic affinoid annulus or disk $(X,T)$, as in definition \ref{basicannuli}. Recall that $(X,T)$ is the generic fiber of a formal coordinate neighborhood $(\fX,T)$ either in a standard formal annulus of height $|\gamma|$, $\Spf \kc\{S,T\}/(ST-\gamma)$, with $\gamma \in \kc$, or in the standard formal disk $\Spf \kc\{T\}$ over $\kc$, and in particular that the coordinate $T$ on $X$ is the pull-back of the formal coordinate $T$ on $\fX$. In this case, $X$ has canonical strictly semistable reduction, and the corresponding strictly semistable model of $X$ is the minimum one, and coincides with $\fX$. So, $D_{\fX}(x,1^{-}) = D_{X}(x,1^{-})$, for any $x \in X \setminus S(X) =  X \setminus S(\fX)$. We denote by $|| - ||_{X}$ the supnorm on $X$. Let $\cE \cong \cO_X^{\mu}$ and let $\na$ be the connection on $\cE$, whose solutions are the column solutions  of the system $\Sigma$. We consider the system of ordinary linear differential equations
(\ref{diffCHDWintro}) on $X$. We define $G_{[i]}$, for $i \in \N$, as in (\ref{eq:stratintro}), \emph{for the global coordinate $T$ on the formal annulus or disk $\fX$}.  
\par \noindent
If $X$ is a disk, we set
\beq \label{eq:raddefsmooth41}
\cR(x,\Sigma) = \min (1, \liminf_{n \to \infty}|G_{[n]}(x)|^{-1/n}) 
\; . 
\eeq
If $X$ is an annulus, we set 
\beq \label{eq:raddefsmooth4}
R(x,\Sigma) = \min (|T(x)|, \liminf_{n \to \infty}|G_{[n]}(x)|^{-1/n}) 
\; ,  
\eeq
and 
\beq
\label{eq:raddefsmooth42}
 \cR (x,\Sigma)  = R (x,\Sigma) / |T(x)| \; .
\eeq
 Then, according to  (\ref{eq:raddefsmooth21}) and (\ref{eq:raddefsmooth2}), we have in any case 
 $$\cR_{\fX}(x,(\cE,\na)) = \cR (x,\Sigma)  \; .$$

\subsection{The Dwork-Robba theorem} \label{Dwork-Robba}

We state here a useful form of the theorem of Dwork and Robba   \cite[Thm. 3.1, Chap. IV]{DGS}. A multivariable version of it is given in \cite{Gachet} and \cite[4.2]{Lucia}.

\begin{thm}
\label{cor:DworkRobba1}
Let $x \in X$, $\Sigma$,  $G_{[i]}$, for $i \in \N$, and $R = R(x,\Sigma)$ be as above. 
 Then, for any $n \in\N$ we have the
following estimate
\beq\label{eq:DworkRobba1}
|{G}_{[n]}(x)| \leq C n^{\mu -1}
R^{-n}\ ,
\eeq
where
$$
C=\max_{i \leq\mu - 1}
\l( R^i |i!|
\l|\l|{G}_{[i]}\r|\r|_X \r)\  .
$$
\end{thm}
\begin{proof} We may assume that $(\fX,T)$ \emph{is} either a standard formal annulus of height $|\gamma|$, with $\gamma \in \kc$, or  the standard formal disk over $\kc$, \cf (\ref{basicannuli}), and that $X = \fX_{\eta}$.  We can also assume that $k$ is algebraically closed. We first explain the notation and the result of remark 3.2 in \cite[Chap. IV]{DGS}. So, let $t = t_{a,\rho} \in X$, for $a \in k$ and $0 < \rho \leq 1$, and let $\cA_{a,\rho} = \cO_{X}(D_{X}(a,\rho^{-}))$ be the ring of analytic functions on $D_{X}(a,\rho^{-})$. Let 
$\cA^{\p}_{a,\rho}$ be the quotient field of the ring $\cA_{a,\rho}$, that is the field of \emph{meromorphic functions} on $D_{X}(a,\rho^{-})$. Notice that the $k$-linear derivation $\partial : f \mapsto  \frac {d\, f}{dT}$ of $\cA_{a,\rho}$ extends uniquely to a $k$-linear derivation $\partial$ of $\cA^{\p}_{a,\rho}$. 

The \emph{boundary seminorm} 
$$||~||_{a,\rho} : \cA^{\p}_{a,\rho} \to \R_{\geq 0} \cup \{\infty \} \; ,
$$
 is defined, for any $f \in \cA^{\p}_{a,\rho}$ as
 $$|| f ||_{a,\rho} = \limsup_{R \to \rho^{-}} |f(t_{a,R})|  \; .
 $$
We refer to \lc for the properties of this seminorm, and in particular for the fact that, for any $f \in  \cA^{\p}_{a,\rho}$, 
\beq
||\frac {\partial^{s}  f}{s! \, f}||_{a,\rho} \leq \rho^{-s} \; \;\; ,  \; \;\; \forall s \geq 0 \;\;.
\eeq

Now, notice that the $k$-linear continuous So, let us assume that we are given a 
  $\mu \times \mu$ matrix of meromorhic functions on $D_{X}(a,\rho^{-})$, such that 
 \beq \label{diffCHDWintro1} 
\Sigma :  \partial \, Y =G\,Y \; . 
\eeq
Then, for any $s \geq 0$, 
$$\frac{1}{s!}\, (\partial^{s}\, Y) \, Y^{-1} = {G}_{[s]} \; ,
$$
and theorem 3.1 and remark 3.2 of \cite[Chap. IV]{DGS} assert that
\beq \label{est1}
|| {G}_{[s]}||_{a,\rho} \leq \rho^{-s} s^{\mu -1} \sup_{0 \leq i \leq \mu -1}  (\rho^{i} || {G}_{[i]}||_{a,\rho}) \; .
\eeq
Now, consider a $k$-rational point $x$ of $X$, and notice that there exists a solution matrix $Y$ of  (\ref{diffCHDWintro1}) analytic in $D_{X}(x, R^{-})$, with $R = R(x,\Sigma)$. If we insist that $Y(x)$ be the identity matrix, this solution is the sum of the convergent series of polynomial functions 
$$Y= \sum_{i=0}^{\infty}  {G}_{[s]}(x) (T-T(x))^{i} \; .$$
In this case, the matrices ${G}_{[s]}$ are analytic on $X$, and in particular on $D_{X}(x,R^{-})$, so that 
$$|| {G}_{[s]}||_{a,\rho} = \sup_{y \in D_{X}(x,R^{-})}| {G}_{[s]} (y)| \leq || {G}_{[s]}||_{X} \;\;,$$
and the previous estimates for $\rho = R$  imply
\beq
|{G}_{[s]}(x)| \leq R^{-s} s^{\mu -1} \sup_{0 \leq i \leq \mu -1}  (R^{i} || {G}_{[i]}||_{X}) \; .
\eeq
So, our result is proven for $x$ a $k$-rational point of $X$. For the general case, we extend the base field from $k$ to any algebraically closed non-archimedean field $k^{\p}$ over $\sH(x)$, and consider  the canonical $k^{\p}$-rational point $x^{\p} \in X \wt_{k} k^{\p} =: X^{\p}$ above $x$. The system $\Sigma$ is to be interpreted on $X^{\p}$, and we write $\fX^{\p} = \fX \wt_{\kc} k^{\p \circ}$ and $(\cE^{\p},\na^{\p})$ for the pull-back of $(\cE,\na)$.  The matrices ${G}_{[s]}$, however, do not change and $|| {G}_{[i]}||_{X} = || {G}_{[i]}||_{X^{\p}}$. Moreover, 
$R = R(x,\Sigma) =  R(x^{\p},\Sigma)$ (while $\cR_{\fX}(x,(\cE,\na)) = 
\cR_{\fX^{\p}}(x,(\cE^{\p},\na^{\p}))$  by definition!). 
So, we obtain 
\beq
|{G}_{[s]}(x^{\p})| \leq R^{-s} s^{\mu -1} \sup_{0 \leq i \leq \mu -1}  (R^{i} || {G}_{[i]}||_{X}) \; ,
\eeq
and the theorem is proven. 
\end{proof}
The next corollary is the prototype of a {\it transfer theorem to an ordinary contiguous disk}
\cite[V.5]{DGS}.

\begin{cor}\label{contiguity} Let $D$ be an open disk in $X$, and let $t$ be its limit point in $X\setminus D$.  Then $R(t,\Sigma) = \inf_{x \in D} R(x,\Sigma) $. 
\end{cor}
\begin{proof} Let $R = \inf_{x \in D} R(x,\Sigma)$, let 
$r$ be the radius of $D$ and let 
$$
C=\max_{i \leq\mu - 1}
\l( r^i |i!|
\l|\l|{G}_{[i]}\r|\r|_D \r) =\max_{i \leq\mu - 1}
\l( r^i |i!|
\l| {G}_{[i]}(t)\r| \r)\  .
$$
 If $R(x,\Sigma) > r$, for some $x \in D$, then $R(y,\Sigma) = R(x,\Sigma)$, for all $y$ in an open  disk $D^{\p}$ with  $D \subsetneqq D^{\p} \subset X$. The statement is obvious in this case. So, we may assume that $R \leq r$.
We have, for any $x \in D$ and $n \geq 0$, 
\beq 
|{G}_{[n]}(x)| \leq C n^{\mu -1} R(x,\Sigma)^{-n}  \leq C n^{\mu -1} R^{-n} \; .
\eeq
So, we also have 
\beq 
|{G}_{[n]}(t)| = ||{G}_{[n]}||_{D} \leq C n^{\mu -1} R^{-n} \; ,
\eeq
which shows that $R \leq  R(t,\Sigma)$. On the other hand 
\beq 
 |{G}_{[n]}(x)| \leq ||{G}_{[n]}||_{D} = |{G}_{[n]}(t)| \leq C n^{\mu -1} R(t,\Sigma)^{-n} \; ,
\eeq
shows that $R(x,\Sigma) \geq R(t,\Sigma)$, for any $x \in D$.
\end{proof}

\subsection{Upper semicontinuity of $x\mapsto  R(x,\Sigma)$}
\label{subsection:USC}

We proceed in our discussion of  the system  (\ref{diffCHDWintro}) on the basic affinoid annulus  or disk $(X,T)$ to show that the function $x\mapsto  R(x,\Sigma)$ is USC on $X$. Again, the multivariable version of  this discussion appears in \cite[4.3]{Lucia}.

We define 
\beq \label{eq:raddefsmooth5}
\wtilde{R}(x,\Sigma) =  \liminf_{n \to \infty}|G_{[n]}(x)|^{-1/n}  
\; ,  
\eeq
so that 
\beq R(x,\Sigma) = 
       \left\{\begin{array}{ccc} \min (1, \wtilde{R}(x,\Sigma) )  &\mbox{if} & \mbox{$X$ is a disk }  \\
     \min (|T(x)|, \wtilde{R}(x,\Sigma) ) &\mbox{if} & \mbox{$X$ is an annulus,}
        \end{array} \right . 
\eeq
and will equivalently prove that the function $x \mapsto \wtilde{R}(x,\Sigma)$ is USC.
We exclude the case when there exist $x \in X$ and $s$ such that $G_{[n]}(x) = 0$, for any $n >  s$, since $\wtilde R(y,\Sigma) = \infty$ and,  for any $y \in X$, 
\beq R(y,\Sigma) = 
       \left\{\begin{array}{ccc} 1 &\mbox{if} & \mbox{$X$ is a disk }  \\
      |T(y)| &\mbox{if} & \mbox{$X$ is a annulus,}
        \end{array} \right . 
\eeq
in that case. For $s = 1,2, \dots$ and for $x \in X$, let
\beq \varphi_s(x) =   \inf_{n \geq s} |G_{[n]}(\xi)|^{-1/n}  \; .
\eeq
So, $x \mapsto \varphi_s(x)$ is USC on $X$, and
\beq \wtilde R(x) = \wtilde R(x, \Sigma) = \lim_{s \to \infty} \varphi_s(x) \; ,
\eeq
is the function introduced in (\ref{eq:raddefsmooth5}). 
 The Dwork-Robba theorem   \ref{cor:DworkRobba1} implies that, $\forall\, \veps >0$,  $\exists \,s_{\veps}$ such that $\forall \, n$ with $n  \geq s_{\veps}$
\beq
  |G_{[n]}(x)|^{1/n} \leq \frac{1+\veps}{\wtilde R(x)} \; ,\;\; \forall \, x \in X \; ,
\eeq
that is 
\beq
 |G_{[n]}(x)|^{ - 1/n} \geq \frac{ \wtilde R(x)} {1+\veps} \; ,\;\; \forall \,  x \in X \; .
\eeq
Hence
\beq \begin{array} {cccc}
\forall \, \veps >0 & \exists\, s_{\veps} & \hbox{such that} & \forall \, s \geq s_{\veps}
\\
\\
 \varphi_s(x) \leq &\wtilde R(x)& \leq (1 + \veps)  \varphi_s(x)  &\forall \; x \in X\;  ,
\end{array}
\eeq
because the sequence $s \mapsto \varphi_s$ is an increasing sequence of functions on $X$. Then, $\forall\, \veps >0$,  $\exists \,s_{\veps}$ such that
\beq 0 \leq \wtilde R - \varphi_s \leq \veps \;\; {\rm on} \; X\; ,\;\;  \forall\, s \geq s_{\veps}\;.
\eeq
So,  $\wtilde R$ is a uniform limit of USC functions on $X$, and is therefore USC.

\subsection{The generalized theorem of Christol-Dwork}\label{CHDWHigherGenus}
\begin{thm} Let $\Sigma$ be a system of linear differential equations on the basic affinoid annulus\footnote{the case of a disk is here trivial.}  $(X,T)$, as above. The function $x \mapsto \wtilde R(x,\Sigma)$ restricts to a continuous function on $S(X)$. 
\end{thm}
\begin{proof}
There is a natural homeomorphism $\eta: [r_0,1]  \xrightarrow{\ \sim\ }   S(X) $, $r \mapsto \eta_r$, with $r_0 \in (0,1] \cap |k|$. We follow the method of \cite{ChristolDwork}, recalled in \S \ref{review}: the function (\ref{eq:raddefsmooth5})
$$
 \wtilde{R}(\eta_r,\Sigma) =  \liminf_{n \to \infty}|G_{[n]}(\eta_r)|^{-1/n}  
\; ,  
$$
is logarithmically concave  on $[r_0,1]$, hence continuous on $(r_0,1)$ and  LSC at $r_0$ and 1. On the other hand, we showed above that $x \mapsto \wtilde{R}(x,\Sigma)$ is USC all over $X$. This proves the theorem.
\end{proof}
\subsection{Descent by Frobenius} \label{Frobenius}

We recall that, for any morphism of schemes $\pi:Y \to S$ where $p=0$ on $S$, we have a  canonical commutative diagram 
\beq
\label{diagram3}
\begin{array}{ccccc} Y&\mathop{\hbox to 40pt{\rightarrowfill}}
\limits^{F_{\rm rel}}& Y^{(p)}&\mathop{\hbox to 40pt{\rightarrowfill}}
\limits^{\Psi}&Y \\
&\pi \searrow&
\pi^{(p)}\Bigg\downarrow&\square& \pi\Bigg\downarrow  \\
&
&
S &\mathop{\hbox to 40 pt{\rightarrowfill}}
\limits^{F_{\rm abs}}&S
\end{array}
\eeq
 where $F_{\rm abs}$ (resp. $F_{\rm rel}$) denotes absolute (resp. relative) Frobenius, and $\Psi \circ F_{\rm rel} = F_{\rm abs}$.
\par 
Let $(\fX,T)$ be a basic formal annulus over $\kc$ of radii $r_{1},r_{2}\in |k| \cap (0,1]$, as in (\ref{genbasannuli}).  We denote $\pi : \fX \to \Spf \kc$ the structural morphism.
In our case diagram \ref{diagram3} for $\fX_{s}$ becomes
\beq
\label{diagram4}
\begin{array}{ccccc} \fX_{s}&\mathop{\hbox to 40pt{\rightarrowfill}}
\limits^{F_{\rm rel}}& (\fX_{s})^{(p)}&\mathop{\hbox to 40pt{\rightarrowfill}}
\limits^{\Psi_{\fX_{s}}}&\fX_{s} \\
&\pi_{s} \searrow &
\pi_{s}^{(p)}\Bigg\downarrow&\square& \pi_{s}\Bigg\downarrow  \\
  &&
\Spec \kt &\mathop{\hbox to 40 pt{\rightarrowfill}}
\limits^{F_{\rm abs}}&\Spec \kt 
\end{array}
\eeq
Let us choose a continuous ring automorphism $\sigma : \kc \to \kc$ lifting the absolute Frobenius $F_{\rm abs}^{\ast}: x \mapsto x^{p}$ of $\kt$.
We want to show 
\begin{lemma} \label{FrobLifting}
The left hand triangle in diagram \ref{diagram4} lifts to  a commutative diagram of $\kc$-formal schemes
\beq
\label{diagram5}
\begin{array}{ccc} \fX&\mathop{\hbox to 40pt{\rightarrowfill}}
\limits^{F_{\rm rel}(\sigma)}& \fX^{(\sigma)} \\
 &\pi\searrow &
\Bigg\downarrow  \pi^{(\sigma)}\\
&&
\Spf \kc   
\end{array}
\eeq
where $\pi_{\fX^{(\sigma)}} := \pi^{(\sigma)}:\fX^{(\sigma)} \to \Spf \kc$ is a formal annulus with special fiber $(\fX_{s})^{(p)} \to \Spec \kt$. 
\end{lemma}
\begin{proof}
Let us start with the special case of $\fX = \fB(r_{1},r_{2})$. In this case we know that there exists a commutative diagram
\beq
\label{diagram6}
\begin{array}{ccc} \fB(r_{1},r_{2}) &\mathop{\hbox to 40pt{\rightarrowfill}}
\limits^{\Phi}&  \fB(r^{p}_{1},r^{p}_{2})  \\ && \\
\pi \searrow & &
\swarrow \pi^{(p)} \\
 &\Spf \kc &
\end{array}
\eeq
where 
$$ \fB(r^{p}_{1},r^{p}_{2})  = \Spf  \kc\{S,T,U\}/(a_{2}^{p}S - T, TU -a^{p}_{1}) \; ,$$
$\pi^{(p)}  : \fB(r^{p}_{1},r^{p}_{2}) \to \Spf \kc$ is the natural structural morphism, and the map $\Phi$ is defined at the ring level by
\beq
\Phi^{\ast} : \sum_{h,i,j}a_{h,i,j}S^{h}T^{i}U^{j} \mapsto  \sum_{h,i,j}a_{h,i,j}S^{ph}T^{pi}U^{pj} \; .
\eeq
We need to use a  semilinear version of this map, as follows. 
 We consider the diagram
\beq
\label{diagram7}
\begin{array}{ccc} \fB(r_{1},r_{2}) &\mathop{\hbox to 40pt{\rightarrowfill}}
\limits^{\Phi^{(\sigma)}}&  \fB(r^{p}_{1},r^{p}_{2}) \wt_{\kc,\sigma}\kc  \\ && \\
\pi  \searrow & &
\swarrow \pi^{(\sigma)} \\
 &\Spf \kc &
\end{array}
\eeq
where   the map $\Phi^{(\sigma)}$ is defined at the ring level by
\beq
(\Phi^{(\sigma)})^{\ast} : \sum_{h,i,j}a_{h,i,j}S^{h}T^{i}U^{j} \mapsto  \sum_{h,i,j}a^{\sigma}_{h,i,j}S^{ph}T^{pi}U^{pj} \; .
\eeq
We observe that the previous diagram \ref{diagram7} satisfies the requirements of the statements \ref{diagram6} for $\fX =  \fB(r_{1},r_{2})  =:\fB$, $\fX^{(\sigma)} = \fB(r^{p}_{1},r^{p}_{2}) \wt_{\kc,\sigma}\kc  =: \fB^{(\sigma)}$, and in particular that $\Phi^{(\sigma)}$ lifts $F_{\rm rel}: \fB_{s} \to \fB_{s}^{(p)}$. We call $\pi_{\fB}$, $\pi_{\fB}^{(\sigma)}$ instead of simply $\pi$, $\pi^{(\sigma)}$, the structural morphisms of $\fB$ and $\fB^{(\sigma)}$, respectively. Consider  the commutative diagram of cartesian squares
\beq
\label{diagram8}
\begin{array}{ccccc} \fX_{s}&\mathop{\hbox to 40pt{\rightarrowfill}}
\limits^{F_{\rm rel}}& (\fX_{s})^{(p)}&\mathop{\hbox to 40pt{\rightarrowfill}}
\limits^{\Psi_{\fX_{s}}}&\fX_{s} \\
\alpha_{s}\Bigg\downarrow&\square&
\alpha_{s}^{(p)}\Bigg\downarrow&\square& \alpha_{s}\Bigg\downarrow  \\
 \fB_{s}&\mathop{\hbox to 40pt{\rightarrowfill}}
\limits^{F_{\rm rel}}& (\fB_{s})^{(p)}&\mathop{\hbox to 40pt{\rightarrowfill}}
\limits^{\Psi_{\fX_{s}}}&\fB_{s}
\end{array}
\eeq
We now define $\alpha^{(\sigma)}: \fX^{(\sigma)} \to \fB^{(\sigma)}$ as the unique lifting of  $\alpha_{s}^{(p)}: \fX_{s}^{(p)} \to \fB_{s}^{(p)}$, according to \cite[Lemma 2.1]{BerkovichCycles}.  We obtain a unique map 
$F_{\rm rel}^{(\sigma)}: \fX \to \fX^{(\sigma)}$ lifting $F_{\rm rel}: \fX_{s} \to \fX_{s}^{(p)}$, and 
a cartesian square

\beq
\label{diagram9}
\begin{array}{ccc} \fX&\mathop{\hbox to 40pt{\rightarrowfill}}
\limits^{F_{\rm rel}^{(\sigma)}}& \fX^{(\sigma)} \\
\alpha\Bigg\downarrow&\square&
\Bigg\downarrow \alpha^{(\sigma)}   \\
 \fB&\mathop{\hbox to 40pt{\rightarrowfill}}
\limits^{\Phi^{(\sigma)}}& \fB^{(\sigma)} \end{array}
\eeq
We now define $\pi^{(\sigma)} :\fX^{(\sigma)} \to \Spf k$ as  $\pi_{\fB}^{(\sigma)} \circ \alpha^{(\sigma)}$, and obtain 
diagram \ref{diagram5}.
\end{proof}
\par
We recall that there exists a continuous map of $\kc$-modules 
\beq \label{psimap}
\Psi^{\ast} : \sum_{h,i,j}a_{h,i,j}S^{h}T^{i}U^{j} \mapsto  \sum_{h,i,j}a_{ph,pi,pj}S^{h}T^{i}U^{j} \; ,
\eeq
such that $\Psi^{\ast} \circ \Phi^{\ast} = \id_{A(r^{p}_{1},r^{p}_{2})}$.

The generic fibers $\varphi =  \Phi_{\eta}: B[r_{1},r_{2}] \to B[r^{p}_{1},r^{p}_{2}]$ and $\phi^{(\sigma)}  =  (\Phi^{(\sigma)})_{\eta}: B[r_{1},r_{2}] \to B[r^{p}_{1},r^{p}_{2}] \wt_{k,\sigma}k$
of the previous maps are both classically used  in the  theory of descent of differential modules by Frobenius, as we have illustrated (using the linear map) in the previous section (where $\varphi = \varphi_{1}$).  
We have the following generalization of the  theorem of descent by Frobenius of Dwork and Christol.

\begin{thm}\label{FrobDesc} {\bf (Descent by Frobenius)} Let $\fX$ be a basic formal annulus over $\kc$, and let $F_{\rm rel}^{(\sigma)}: \fX \to  \fX^{(\sigma)}$ be as in lemma \ref{FrobLifting}. We denote by $\phi^{(\sigma)} : X := \fX_{\eta} \to X^{(\sigma)} := \fX^{(\sigma)}_{\eta}$ the generic fiber of $F_{\rm rel}^{(\sigma)}$.  Let $(\cE,\na)$ be an object of ${\bf MIC}_{\fX}(X/k)$ and assume  that $\cR_{\fX}(x,(\cE,\na)) > p^{-\frac{1}{p-1}}$, $\forall x \in X$. Then, there exists a unique object  $(\cF,\na)$ of ${\bf MIC}_{\fX^{(\sigma)}}(X^{(\sigma)}/k)$, such that $(\cE,\na) = (\phi^{(\sigma)})^{\ast}(\cF,\na)$. For any $x \in X$, we have 
\beq \label{frobtrick}
\cR_{\fX^{(\sigma)}}(\phi^{(\sigma)}(x),(\cF,\na)) = \cR_{\fX}(x,(\cE,\na))^{p} \; .
\eeq
\end{thm}
\begin{proof}
The construction of $(\cF,\na)$ is the same, \emph{mutatis mutandis} as the one of Christol-Dwork \cite[4.3]{ChristolDwork} and of \cite[V.7]{DGS}. It is important however to use the coordinate-free 
  presentation of Kedlaya \cite[6.3]{Ke}, that eliminates the problem of apparent singularities, which causes so much technical complication in \cite{ChristolDwork} and \cite{DGS}. We do not give details here. Notice however that one may assume, without loss of generality, that the field $k$ contains the $p$-th roots of unity. 
The  group $\mu_p$ acts on $\fB(r_{1},r_{2})$
via $\zeta \mapsto \tau_{\zeta}$, where 
$\tau_{\zeta}:  (S,T,U) \mapsto (\zeta S,\zeta T,\zeta^{-1}U)$ and the map $\Psi^{\ast}$ of formula \ref{psimap} is 
$$\Psi^{\ast} = p^{-1} \sum_{\zeta \in \mu_{p}} \tau_{\zeta}^{\ast} \; .
$$
The quotient map $\fB(r_{1},r_{2}) \to \fB(r_{1},r_{2}) /\mu_{p}$ identifies with $\Phi$ in (\ref{diagram5}).
The action of $\mu_{p}$ on $\fB(r_{1},r_{2})$ uniquely lifts to an action on $\fX$ and on $\fX^{(\sigma)}$. The quotient map 
$\fX \to \fX^{(\sigma)}/\mu_{p}$ identifies then with $F_{\rm rel}^{(\sigma)}$ in  (\ref{diagram6}).

We may now follow the method of Kedlaya \lc to conclude. 
\end{proof}

\section{Continuity, log-piecewise linearity and log-concavity of $x \mapsto \cR_{\fY,\fZ} (x,(\cE,\na)) $}\label{Conclusion}

We use here the notation of definitions \ref{def:altradius} and \ref{mersing}. The assumption that $k$ is algebraically closed may be now relaxed. We assume that $k$ is a non-trivially valued non-archimedean  extension field of $\Q_p$, $\ol{X} = \fY_{\eta}$, for a semistable formal scheme $\fY$, $\fZ$ is a finite \'etale closed $\kc$-formal subscheme of $\fY$, supported in the smooth locus of $\fY$. The generic fiber $\cZ$ of $\fZ$ is supposed to consist of $r$ distinct $k$-rational points $\{z_1,\dots,z_r \} \subset \ol{X}(k)$ and $X = \ol{X} \setminus \fZ$.
\begin{thm} \label{context2} Let $(\cE,\na)$ be an object of ${\bf MIC}_{\fY}(\ol{X}(\ast\cZ)/k)$. 
The function $X \to \R_{>0}$, $x \mapsto \cR_{\fY,\fZ} (x,(\cE,\na)) $  is continuous.
\end{thm}  
\begin{proof}  We have already shown in lemma \ref{compact}, that the statement reduces to the case when there are no $z_i$'s, \ie $X = \ol{X}$ is compact. We then set $\fY = \fX$ for the given semistable model of $X$. 
For any non-archimedean field extension $k^\p/k$ and  \'etale morphism of semistable $k^{\p \circ}$-formal schemes $\psi:\fY \to \fX \wt k^{\p \circ}$, the assignment 
\beq
\label{localrad}
f_{\psi} (y) =  \cR_{\fY} (y,\psi^{\ast}(\cE,\na))\;,\;\; \forall y \in \fY_{\eta}\; ,
\eeq
is an \'etale-local system of functions $(f_{\psi})_{\psi}$ on $\fX$ with values in $(0,1]$, as defined in (\ref{etalefct}),  by lemma \ref{etalechange2} and the definitions.

In the end, by lemmas \ref{cover} and  \ref{etaleloc}, we 
are reduced to the situation of \S \ref{GeneralDR}, and in particular to proving the following 
\begin{lemma} 
Let $(X,T) = (\fX,T)_{\eta}$ be a basic affinoid annulus, and let $x \in X$, $\Sigma$,  $G_{[i]}$, for $i \in \N$, $R(x) := R(x,\Sigma)$ be as in theorem \ref{cor:DworkRobba1}.
Then 
the function  $R: X \to \R_{>0}$ is continuous.
\end{lemma}
\begin{proof}  Let $\wtilde{R}(x) :=\wtilde{R}(x,\Sigma)$ be as in (\ref{eq:raddefsmooth5}). 
We check  conditions $(1)$ to $(5)$ in theorem \ref{contthm} for the function $R$. Notice that   $R(x,\Sigma)$  is given by formula \ref{eq:raddefsmooth} in terms of any formal \'etale coordinate on $\fX$, and over any non-archimedean extension field $k^{\p}/k$. 
Condition $(1)$ is obvious: if $x_0 \in X(k)$, $R(x,\Sigma) = R(x_0,\Sigma)$, $\forall x \in D_X(x_0, R(x_0,\Sigma)^-)$.
Condition $(2)$ is clear since, for any affinoid $V \subset X$,
$$|G_{[n]}(x)| \leq \max_{y \in \Gamma(V)}|G_{[n]}(y)| \; .$$
and therefore 
$$\wtilde{R}(x,\Sigma) =  \liminf_{n \to \infty} |G_{[n]}(x)|^{-1/n}  \geq \min_{y \in \Gamma(V)} \wtilde{R}(y,\Sigma)\; .$$
So, the same holds true for $R(x,\Sigma)$.
Conditions $(3)$ and $(4)$ have been proved together in section \ref{CHDWHigherGenus}.
Condition $(5)$ was proven in section \ref{subsection:USC}. 
\end{proof} 
This proves the theorem, too. 
\end{proof}
\begin{rmk} In \cite[5.3]{Lucia} we gave a direct proof of this statement, in the case of a system (\ie for $\cE$ free) over an affinoid domain of $\A^1$. 
\end{rmk}  
\begin{cor}\label{genpiecewiselinear} Assume $\fY$ is strictly semistable, and let $\mu$ be the rank of $\cE$ on $X = \ol{X}  \setminus 
\cZ$.  Let $E$ be an open edge  of ${\bf S}_{\fZ}(\fY)$ and let $\ol{E}$ be the closure of $E$ in $\ol{X}$.
We have the following possibilities:
\begin{enumerate}
\item  $E$ is  of finite length $r_E^{-1} \in [1,\infty)$. Then,   for any choice of one of the two canonical homeomorphism $\rho  : \, \ol{E} \xrightarrow{\ \sim\ } [r_E,1]$,
the restriction of $x \mapsto \cR(x) := \cR_{\fY,\fZ} (x,(\cE,\na))$ to  $E$ is the infimum of the constant 1 and a finite set of functions of the form
\beq \label{pentes} |p|^{1/(p-1)p^h}|b|^{1/jp^h}\rho(x)^{s/j}\; ,
\eeq
 where $j \in \{ 1, 2, \dots,\mu\}$, $s \in \Z$,  $h \in \Z \cup \{\infty\}$, $b \in k^{\times}$.
\item
if $E$ is of infinite length, and $z_i$ is a boundary point of $E$ in $\ol{X}$, and $\rho : \, \ol{E} \xrightarrow{\ \sim\ } [0,1]$ is the canonical homeomorphism with $\rho(z_i) = 0$, then, $\forall \; \veps \in |k| \cap (0,1)$ the restriction of $x \mapsto \cR_{\fY,\fZ} (x,(\cE,\na))$ to  $\rho^{-1}([\veps,1])$ is an infimum of a finite set of functions as in \ref{pentes}. 
\end{enumerate} 
\end{cor}
\begin{proof} 
We may restrict to case 1. We know the result in case  $\tau_{\fY}^{-1}(\ol{E})$ is an annulus in $\A^1$. So, the restriction of $x \mapsto \cR_{\fY,z_1,\dots,z_r} (x,(\cE,\na))$ to $\rho^{-1}([r_E/\veps ,\veps])$ has the required properties for any $\veps \in |k| \cap (r_{E}^{1/2},1)$. 
Let $R(r) =  \cR_{\fY,\fZ} (\rho^{-1} (r),(\cE,\na))$, and let $S= \log \circ R \circ \exp $. So, $S$ is concave, piecewise linear and continuous in $[\log r_E, 0]$. We have to show that is a polygon with a finite number of sides. It is enough to show that the slope of the function $S(t)$ is bounded as $t \to 0$, since the case $t \to \log r_E$ may be obtained by inversion. We may then assume, going to an \'etale neighborhood of the singular point of $\fY$ corresponding to $E$,  that $\tau_{\fY}^{-1}(\ol{E})$ is a basic affinoid annulus. 
We may now follow the strategy of proof of theorem \ref{slopes}.  If we have $R(0) = 1$, the result is clear because in this case $S$ is never decreasing in a neighborhood of $0$, and its slopes  are then non-negative. On the other hand, the slopes are rational numbers with denominator bounded by $\mu$. So, there can only be a finite number of sides. If $R(0) < 1$, we apply  we apply our result (\ref{FrobDesc}) of descent by Frobenius instead, and reduce eventually to the case explicitly computed by Young, precisely as in the proof of theorem \ref{slopes}. 
\end{proof}
The last results show that the function 
$\log \rho \mapsto \log \cR_{\fY,\fZ} (\rho^{-1}(-),(\cE,\na))$ on $[\log r_E,0]$ as in (\ref{pentes}) is a finite concave polygon $\leq 0$.   In particular, \cite[Appendix I, Corollary]{DGS}, 
\begin{cor} \label{Robbacond} Let $E$ be an edge of finite length of $S_{\fZ}(\fY)$, parametrized as in corollary \ref{genpiecewiselinear}, 
and let $x, y \in \ol{E}$ be such that $\rho(x) \leq \rho(y)$. Then, if $\cR_{\fY,\fZ} (x,(\cE,\na)) = \cR_{\fY,\fZ} (y,(\cE,\na)) = 1$, one has $\cR_{\fY,\fZ} (z,(\cE,\na))  =1$, $\forall \, z \in 
\sp_{\fY}^{-1}([x,y])$, where $[x,y] = \rho^{-1} ([\rho(x), \rho(y)])$.
\end{cor}

\par 
Our last result makes use of the $p$-adic Turrittin theory of \cite{Advances}. A crucial point needed in that paper in order to analytically reduce the system to its formal Turrittin canonical form, was the non-appearance of exponent differences of $p$-adic Liouville type.
We recall that, for any differential system $\Sigma$ as in (\ref{diffCHDWintro}), with $G$ a matrix with entries in $k((T))$,  the \emph{exponents at 0} of $\Sigma$ are formal invariants of the (fuchsian part of the) corresponding differential module $(\cE,\na)$, and the \emph{exponent-differences at 0} are the exponents at 0 of the system corresponding to the differential module of endomorphisms $({\cE}nd_{k((T))} (\cE), \na)$. So, the difference of two exponents at 0 is an exponent-difference at 0, but the converse is false in general. 
\begin{rmk} We recall that, by the $p$-adic Roth theorem,  algebraic numbers are $p$-adically non-Liouville.
\end{rmk}
\begin{prop}  \label{advances2}
In the previous notation, assume $(\cE,\na)$ is an object of   
${\bf MIC}_{\fY}(\ol{X}(\ast \cZ)/k)$ and let $E$ be an open edge of  infinite length of $S_{\fZ}(\fY)$, with closure $\ol{E}$ ending at the $k$-rational point $z_i$.  Let $\rho_{z_i}(\cE,\na)$ be the  Poincar\'e-Katz rank  \cite{Advances} of  $(\cE,\na)$ at $z_i$, and assume that the exponent-differences of $(\cE,\na)$ at $z_i$ are $p$-adically non-Liouville numbers. Then, the restriction of $x \mapsto \cR_{\fY,\fZ} (x,(\cE,\na))$ to  $\ol{E}$ is the infimum of the constant 1 and a finite set of functions of the form
\beq \label{pentes2} |p|^{1/(p-1)p^h}|b|^{1/jp^h}\rho(x)^{s/j}\; ,
\eeq
 where $j \in \{ 1, 2, \dots,\mu\}$, $s \in \Z$,  $h \in \Z \cup \{\infty\}$, $b \in k^{\times}$ and $s/j \leq \rho_{z_i}(\cE,\na)$. We have 
 \beq  \label{advances3}
\lim_{x \to z_i}\cR_{\ol{\fX}, \fZ}(x, (\cF,\na))/ |T_{z_i}(x)|^{\rho_{z_i}(\cF,\na)}  = C \; ,
\eeq 
for a non-zero constant  $C \in |(k^{\alg})^{\times}|$. 
\end{prop}
\begin{proof} This is a combination of the previous discussion with \cite[prop. 4]{Advances}.
\end{proof}
\par
We have a variant of corollary \ref{genpiecewiselinear}:
\begin{cor} \label{Robbacond2} Let $v$ be the second end of $\ol{E}$ as in the proposition. Then, if $(\cE,\na)$ is regular at $z_i$ and 
$\cR_{\fY,\fZ} (v,(\cE,\na)) = 1$, we have  $\cR_{\fY,\fZ} (z,(\cE,\na))  =1$, $\forall \, z \in D_{\fY}(z_i,1^-) \setminus \{z_i\}$.
\end{cor}

We conclude
\begin{prop} \label{Robbacond5} Let $(\cE,\na)$ be an object of ${\bf MIC}_{\fY}(\ol{X}(\ast \cZ)/k)$, and let 
$(\fY_1,\fZ_{1})$ be another pair like $(\fY,\fZ)$. We assume that $\fY_1$, $z_{r+1}, \dots,z_{r+s}$ are as in equation \ref{eq:change} and that $\cZ_1 = \{z_1,\dots,z_{r+s}\}$, $X_1 = \ol{X} \setminus \cZ_1$. 
Then $z \mapsto \cR_{\fY,\fZ} (z,(\cE,\na))$ is identically 1 on $X$ if and only if $z \mapsto \cR_{\fY_1,\fZ_1} (z,(\cE,\na))$ is identically 1 on $X_1$.

\end{prop}
\begin{proof} By (\ref{eq:change}), $\cR_{\fY_1,\fZ_1} (z,(\cE,\na))  \geq 
\cR_{\fY,\fZ} (z,(\cE,\na))$, $\forall \, z \in X_1$.  So, let us assume that $\cR_{\fY_1,
\fZ_1} (z,(\cE,\na))  =1$, 
$\forall \, z \in X_1$. Suppose first that $\fY = \fY_1$, and $s=1$. So, the disk $D_{\fY}(z_{r+1},1^-)$ contains no singularity of $(\cE,\na)$, and, by the transfer theorem  \ref{contiguity}, for any $z \in D_{\fY}(z_{r+1},1^-)$, $1 \geq \cR_{\fY,\fZ} (z,(\cE,\na)) \geq \cR_{\fY,\fZ} (v,(\cE,\na))=  \cR_{\fY,\fZ_1} (v,(\cE,\na)) =1$, where $v$ is the boundary point of $D_{\fY}(z_{r+1},1^-)$ in ${\bf S}(\fY)$. 
Let us assume instead that  $s =0$ and  ${\bf S}(\fY_1)$ is obtained from ${\bf S}(\fY)$ by addition of one open edge $E$ and one vertex $v$. 
Let $w$ be the second end of $\ol{E}$. Then, the edge $E$ corresponds to a $k$-rational open annulus $B(r_E,1)$ in an open disk $D(0,1^-) \cong D_{\fY}(a,1^-)$, via a $\fY$-normalized coordinate at $a$, where the vertex $v$ (resp. $w$) corresponds to $t_{0,r_E^+}$ (resp. to $\tau_{\fY}(a) = t_{0,1}$). Now, again by (\ref{contiguity}), 
$\cR_{\fY,\fZ}(w, (\cE,\na)) \geq \cR_{\fY,\fZ}(v, (\cE,\na)) = \cR_{\fY_1,\fZ}(v, (\cE,\na)) =1$, 
 We conclude from proposition \ref{Robbacond2} that $\cR_{\fY,\fZ} (z,(\cE,\na))=1$, $\forall \, z \in X$. Suppose now  $s =0$ and  ${\bf S}(\fY_1)$ is obtained from ${\bf S}(\fY)$ by subdividing one  open edge $E$ of ${\bf S}(\fY)$ as $E = E_1 \cup \{v\} \cup E_2$, where $v$ (resp. $E_1$, $E_2$) is the only new vertex (resp. are the only new edges) of ${\bf S}(\fY_1)$. Then, $D_{\fY_1,\fZ} (z,1^{-}) = D_{\fY_1,\fZ} (z,1^{-})$, and therefpre, 
 for any object $(\cF,\na)$ of  ${\bf MIC}_{\fY}(\ol{X}(\ast \fZ)/k)$, $\cR_{\fY_1,\fZ} (z,(\cE,\na)) = \cR_{\fY,\fZ} (z,(\cE,\na))$, for any $z \in X$; so the result follows also in this case. The general case is an iteration of the steps we have described.
\end{proof}

\begin{cor} \label{Robbacond4}
 Theorem \ref{Robbacond3} holds. 
\end{cor}
\begin{proof} Let $(\cF,\na)$ be an object of ${\bf MIC}_{\ol{\fX}}(\ol{X}(\ast \cZ)/k)$ satisfying the assumptions of Theorem \ref{Robbacond3}. 
From the corollaries \ref{Robbacond} and  \ref{Robbacond2}  we deduce that the function 
$x \mapsto \cR_{\ol{\fX},\fZ} (z,(\cF,\na))$ is identically 1 on $X =  \ol{X} \setminus  \cZ$. Let  $D$ be an open disk in $X$. We may assume that $D$ is isomorphic to $D(0,1^-)$, and must show that the restriction of $\cF^{\na}$ to $D$ is locally constant.
If $D$ is contained in one of the disks $D_{\ol{\fX}}(z_i,1^-)$, hence in the punctured disk $D_{\ol{\fX}}(z_i,1^-) \setminus \{z_{i}\}$, 
the restriction of $\cF^{\na}$ to $D$ is in fact constant. Similarly if $D \subset D_{\ol{\fX}}(x,1^-)$, for some $x \in X(k)$, $\ds x \notin \cup_{i=1}^r D_{\ol{\fX}}(z_i,1^-)$. Otherwise, we may replace $D$ by any smaller closed $k$-rational disk, and assume that the maximal point $v$ of $D$ is an interior point of an open edge $E$ of ${\bf S}_{\fZ}(\fY)$. The closed subset  $D \cap S_{\fZ}(\fY)$ of $D$, has a canonical graph structure, whose open edges (resp. vertices) are all the open edges (resp. vertices) of ${\bf S}_{\fZ}(\fY)$ contained in $D$ and the extra open edge $E \cap D$ (resp. $v$). We call $\Gamma$ this graph structure. So, $\Gamma$ corresponds to a formal model $\fD$ of $D$, and clearly $z \mapsto \cR_{\fD} (z,(\cF,\na)_{|D})$ is identically 1 on $D$. But $\cF_{|D}$ is free, so that $(\cF,\na)_{|D}$ corresponds to a system $\Sigma$ as in \ref{diffCHDWintro}, over the closed disk $D(0,1^+)$, with $R(t_{0,1},\Sigma) = 1$. The transfer theorem shows that $\cR(x,\Sigma) = 1$, $\forall x \in D(0,1^+)$.
\end{proof}

\begin{rmk} Let $X$ be a compact rig-smooth strictly $k$-analytic  curve. Given an object $(\cE,\na)$ of ${\bf MIC}(X/k)$, we may look for  semistable models $\fY$ of $X$ such that $(\cE,\na)$ is an object of ${\bf MIC}_{\fY}(X/k)$. The function $x \mapsto  \cR_{\fY} (x,(\cE,\na))$, $X \to [0,1]$ then depends on $\fY$ as explained in subsection \ref{apparent}. For any point $v \in X$ of type $(2)$, let $E$ be a  segment  of $X$ ending at $v$. We may find a formal model $\fY$ of $X$, as before, such that $v$ is a vertex of ${\bf S}(\fY)$, and that an edge $\ell$ of ${\bf S}(\fY)$ ending at $v$ is contained in $E$. We parametrize the closure of $\ell$ by $\rho: \ol{\ell} \to [r,1]$, with $\rho(v) =1$, as before. Then, the right slope of the function $\log \rho(x) \mapsto  \log \cR_{\fY} (x,(\cE,\na))$ at $0$, is independent of the choice of $\fY$, and it is called the \emph{maximal slope of $(\cE,\na)$ at $v$ in the direction of $E$}. 
We regard theorem \ref{Robbacond3} as an indication of the existence of a purely analytic theory of slopes. 
\end{rmk}

We now  answer a question raised by Berthelot. 
\begin{prop} Let $f : Y \to X$ be a non-constant morphism of rig-smooth strictly $k$-analytic curves, where $X$ admits an embedding as an analytic domain in the analytification of a smooth projective $k$-curve. Then, if  $(\cE,\na)$ is an object of ${\bf MIC}^{Robba}(X/k)$, $f^{\ast}(\cE,\na)$ is an object of ${\bf MIC}^{Robba}(Y/k)$.
\end{prop}
\begin{proof} Let $D$ be any open disk in $Y$. Suppose first that  $X$ is isomorphic to an analytic domain in a projective curve $\ol{X}$ of genus $\geq 1$. Then the map $f$ induces an analytic map $\ol{f}: D \to \ol{X}$ and, by \cite[4.5.3]{Berkovich}, 
$\ol{f}(D)$ is contained in an open  disk in $\ol{X}$. By the arguments given in the proof of \cite[4.5.2]{Berkovich},  it follows that $\ol{f}(D)$ \emph{is} an open disk in $\ol{X}$. Therefore, $\ol{f}(D) = f(D)$ is in fact an open disk in $X$. Suppose now that $X$ is an analytic domain in $\A^{1}$. By the same arguments as before, we have that $f(D)$ is an open disk in $X$. In the cases considered so far, the fact that $f^{\ast}(\cE,\na)$ is a Robba connection is checked  directly on  the definition. 
Finally observe that on $X = \P^{1}$, ${\bf MIC}(X/k) = {\bf MIC}^{Robba}(X/k)$ only consists of direct sums of copies of
 $(\cO_{X},d_{X/k})$. But the inverse image of $(\cO_{X},d_{X/k})$ via $f$  is $(\cO_{Y},d_{Y/k})$, which is a Robba connection.  
\end{proof}


\begin{thebibliography}{10}
%

%


\bibitem{Advances}
Francesco Baldassarri.
\newblock Differential modules and singular points of $p$-adic differential
  equations.
\newblock {\em Advances in Math.}, 44(2):155--179, 1982.

\bibitem{Lucia}
Francesco Baldassarri and Lucia Di~Vizio.
\newblock Continuity of the radius of convergence of $p$-adic differential
  equations on {B}erkovich analytic spaces.
\newblock {\em arXiv:0709.2008v3 [math.NT]}.

\bibitem{FormalCurves}
Francesco Baldassarri.
\newblock Formal models of  analytic curves over a  non-archimedean field.
\newblock {\em In preparation}.


\bibitem{polygon}
Francesco Baldassarri.
\newblock Convergence polygons of  differential equations on $p$-adic analytic curves.
\newblock {\em In preparation}.


\bibitem{Berkovich}
Vladimir~G. Berkovich.
\newblock {\em Spectral theory and analytic geometry over non-{A}rchimedean
  fields}, volume~33 of {\em Mathematical Surveys and Monographs}.
\newblock American Mathematical Society, Providence, RI, 1990.

\bibitem{BerkovichEtale}
Vladimir~G. Berkovich.
\newblock \'{E}tale cohomology for non-{A}rchimedean analytic spaces.
\newblock {\em Institut des Hautes \'Etudes Scientifiques. Publications
  Math\'ematiques}, 78:5--161, 1993.

\bibitem{BerkovichCycles}
Vladimir~G. Berkovich.
\newblock Vanishing cycles for formal schemes.
\newblock {\em Inventiones Mathematicae}, 115(3):539--571, 1994.
%

    
\bibitem{BerkContr}
Vladimir~G. Berkovich.
\newblock Smooth {$p$}-adic analytic spaces are locally contractible.
\newblock {\em Invent. Math.}, 137:1--83, 1999.

\bibitem{BerkDwork}
Vladimir~G. Berkovich.
\newblock Smooth {$p$}-adic analytic spaces are locally contractible. {II}.
\newblock In {\em Geometric aspects of Dwork theory. Vol. I, II}, pages
  293--370. Walter de Gruyter GmbH \& Co. KG, Berlin, 2004.

\bibitem{BerkInt}
Vladimir~G. Berkovich.
\newblock {\em Integration of one-forms on {$p$}-adic analytic spaces}, volume
  162 of {\em Annals of Mathematics Studies}.
\newblock Princeton University Press, Princeton, NJ, 2007.

\bibitem{bosch}
Siegfried Bosch.
\newblock {\em Lectures on Formal and Rigid Geometry}.
\newblock Preprintreihe des Mathematischen Instituts. Heft 378. Westf\"alische
  Wilhelms-Universit\"at, M\"unster, Juni 2005.

\bibitem{BGR}
Siegfried Bosch, Ulrich G{\"u}ntzer, and Reinhold Remmert.
\newblock {\em Non-{A}rchimedean analysis}, volume 261 of {\em Grundlehren der
  Mathematischen Wissenschaften}.
\newblock Springer-Verlag, Berlin, 1984.

\bibitem{BL0}
Siegfried Bosch and Werner L{\"u}tkebohmert.
\newblock Stable reduction and uniformization of abelian varieties. {I}.
\newblock {\em Math. Ann.}, 270(3):349--379, 1985.

\bibitem{WL}
Werner L{\"u}tkebohmert.
\newblock Formal-algebraic and rigid-analytic geometry.
\newblock {\em Math. Ann.}, 286:341--371, 1990.

\bibitem{BL1}
Siegfried Bosch and Werner L{\"u}tkebohmert.
\newblock Formal and rigid geometry. {I}. Rigid spaces.
\newblock {\em Math. Ann.}, 295:291--317, 1993.

%
%


\bibitem{Ch1}
Gilles Christol.
\newblock Le th\'eor\`eme de Turrittin 
  {$p$}-adique.
\newblock {\em manuscript}, 2009.


\bibitem{Ch2}
Gilles Christol.
\newblock Rayons des solutions de l'\'equation de Dwork.
\newblock {\em manuscript}, 2009.


\bibitem{ChMe3}
Gilles Christol and Zoghman Mebkhout.
\newblock Sur le th\'eor\`eme de l'indice des \'equations diff\'erentielles
  {$p$}-adiques. {III}.
\newblock {\em Ann. of Math.}, 151(2):385--457, 2000.

\bibitem{ChristolDwork}
Gilles Christol and Bernard Dwork.
\newblock Modules diff\'erentiels sur des couronnes.
\newblock {\em Universit\'e de Grenoble. Annales de l'Institut Fourier},
  44(3):663--701, 1994.

\bibitem{ChMeAST}
Gilles Christol and Zoghman Mebkhout.
\newblock \'{E}quations diff\'erentielles {$p$}-adiques et coefficients
  {$p$}-adiques sur les courbes.
\newblock In {\em Cohomologies $p$-adiques et applications arithm\'etiques,
  II}, number 279, pages 125--183. SMF, 2002.

\bibitem{dejong}
A. Johann de Jong.
\newblock \'Etale fundamental groups of non-archimedead analytic spaces.
\newblock {\em Compositio Mathematica},
  97:89--118, 1995.

\bibitem{CA}
A. Johann de Jong.
\newblock Families of curves and alterations.
\newblock {\em Annales de l'Institut Fourier},
  47 n. 2 :599--621, 1997.
%



\bibitem{DGS}
Bernard Dwork, Giovanni Gerotto, and Francis~J. Sullivan.
\newblock {\em An introduction to {$G$}-functions}, volume 133 of {\em Annals
  of Mathematics Studies}.
\newblock Princeton University Press, 1994.


\bibitem{EGA}
Jean Dieudonn\'e and Alexander  Grothendieck
\newblock \'El\'ements de G\'eom\'etrie Alg\'ebrique - Chapitre IV, partie 3.
\newblock {\em Publ. Math. IHES}, 28:5--255, 1966.

\bibitem{elkik}
Ren\'ee Elkik.
\newblock Solutions d'\'equations \`a coefficients dans un anneau hens\'elien.
\newblock {\em   Annales Sci. de l'ENS},
  6(4):553--603, 1973.

\bibitem{FM}
Jean Fresnel, Michel Matignon.
\newblock Sur les espaces analytiques quasi-compacts de dimension 1 sur un corps valu\'e complet ultram\'etrique.
\newblock {\em Annali di Matematica Pura e Applicata},
  145 (4): 159--210, 1986.


\bibitem{GenTop}
Ryszard Engelking.
\newblock {\em General topology}, volume~6 of {\em Sigma Series in Pure
  Mathematics}.
\newblock Heldermann Verlag, Berlin, second edition, 1989.
%
%

\bibitem{Gachet}
Fr\'ed\'eric Gachet.
\newblock Structure fuchsienne pour des modules diff\'erentiels sur une
  polycouronne ultram\'etrique.
\newblock {\em Rend. Sem. Mat. Univ. Padova}, 102:157--218,
  1999.
%

\bibitem{Ke}
Kiran S. Kedlaya.
\newblock Local monodromy for $p$-adic differential equations: an overview.
\newblock {\em Intl. J. of Number Theory}, 1:109--154, 2005.

%

%
%


\bibitem{Pons}
\'Emilie Pons.
\newblock Modules diff\'erentiels non solubles. Rayons de convergence et
  indices.
\newblock {\em Rend. Sem. Mat. Univ. Padova}, 103:21--45, 2000.

\bibitem{Pu}
Andrea Pulita
\newblock Rank one solvable {$p$}-adic differential equations and finite abelian characters via {L}ubin-{T}ate groups.
\newblock {\em Math. Ann.}, 337(3):489--555, 2007.
%

\bibitem{RoCh}
Philippe Robba and Gilles Christol.
\newblock {\em \'Equations diff\'erentielles $p$-adiques. Applications aux
  sommes exponentielles}.
\newblock Actualit\'es Math\'ematiques. Hermann, 1994.


\bibitem{Te}
Michael Temkin.
\newblock Stable modification of relative curves.
\newblock {\em arXiv:0707.3953v2 [math.AG]}.

%

\bibitem{Young}
Paul Thomas Young.
\newblock Radii of convergence and index for $p$-adic differential operators.
\newblock {\em Trans. A.M.S.}, 333:769--785, 1992.

\end{thebibliography}
\end{document}